\RequirePackage{fix-cm}
\documentclass[letterpaper,11 pt]{article}
\usepackage{graphicx}
\usepackage{amsfonts,amsmath,fullpage,bbm}
\usepackage{amssymb,amsthm,multirow,verbatim}
\usepackage{acronym,wrapfig,plain,mathrsfs,enumerate,relsize,color}

\usepackage{algorithm}
\newtheorem{theorem}{Theorem}
\newtheorem{remark}{Remark}
\newtheorem{corollary}{Corollary}
\newtheorem{lemma}{Lemma}

\newtheorem{property}{Property}
\newtheorem{proposition}{Proposition}
\newtheorem{assumption}{Assumption}
\newtheorem{definition}{Definition}
\usepackage{subfig}
\usepackage{algorithmic}
\usepackage{pifont}
\usepackage{footmisc}
\newcommand{\cmark}{\ding{51}}%
\newcommand{\xmark}{\ding{55}}%
\newcommand{\us}[1]{{\color{black}#1}}

\newcommand{\aj}[1]{{\color{black}#1}}
\newcommand{\uvs}[1]{{\color{black}#1}}
\newcommand{\vvs}[1]{{\color{black}#1}}
\newcommand{\blue}[1]{{\color{black}#1}}
\newcommand{\af}[1]{{\color{black}#1}}
\newcommand{\red}[1]{{\color{black}#1}}
\newcommand{\fythird}[1]{{\color{black}#1}}

\newcommand{\afr}[1]{{\color{black}#1}}
\newcommand{\jal}[1]{{\color{black}#1}}
\newcommand{\usr}[1]{{\color{black}#1}}
\def\ulambda{{\underline{\lambda}}}
\def\olambda{{\overline{\lambda}}}
\begin{document}




\title{A Variable Sample-size Stochastic Quasi-Newton  Method 
for Smooth and Nonsmooth Stochastic Convex Optimization}

\author{Afrooz Jalilzadeh, Angelia Nedi{\'{c}}, Uday V. Shanbhag\footnote{Authors contactable at \texttt{azj5286@psu.edu,Angelia.Nedich@asu.edu,udaybag@psu.edu,farzad.yousefian@okstate.edu} and gratefully acknowledge the support of NSF Grants 1538605, 1538193  (Shanbhag),  1246887 (CAREER, Shanbhag), CCF-1717391 (Nedi\'{c}), and by the ONR grant no. {N000141612245} (Nedi\'{c}); Conference version \aj{was published} in IEEE Conference on Decision and Control (2018)\cite{jalilzadeh2018variable}.}, and Farzad Yousefian}
%


\date{}
\maketitle

\begin{abstract}%
Classical theory for quasi-Newton schemes has focused on smooth deterministic
unconstrained optimization while recent forays into stochastic convex
optimization have largely resided in smooth, unconstrained, and strongly convex
regimes. Naturally, there is a compelling need to address nonsmoothness, the
lack of strong convexity, and the presence of constraints. Accordingly, this
paper presents a quasi-Newton framework that can process merely convex and
possibly nonsmooth (but smoothable) stochastic convex problems.  We propose a
framework that combines iterative smoothing and regularization with a
variance-reduced scheme reliant on using an increasing sample-size of
gradients. We  make the following contributions. (i) We develop {\em a
regularized and smoothed variable sample-size BFGS update} ({\bf rsL-BFGS})
that generates a sequence of Hessian approximations and  can accommodate
nonsmooth convex objectives by utilizing iterative {regularization and}
smoothing. (ii) In {\em strongly convex} regimes with state-dependent noise,
the proposed variable sample-size stochastic quasi-Newton \aj{({\bf VS-SQN})}
scheme admits a non-asymptotic linear rate of convergence while the oracle
complexity of computing an $\epsilon$-solution is
$\mathcal{O}({\kappa^{m+1}}/\epsilon)$ where $\kappa$ denotes the condition
number and $m\geq 1$.  In nonsmooth (but smoothable) regimes, {using Moreau
smoothing {retains the} linear convergence rate} for the resulting smoothed
{\bf VS-SQN} (or {\bf sVS-SQN}) scheme.  Notably, {the nonsmooth regime allows
for accommodating convex} constraints. {To contend with the possible unavailability of Lipschitzian and strong convexity parameters, we also provide sublinear rates for diminishing steplength variants that do not rely on the knowledge of such parameters}; (iii) In merely convex but smooth
settings, the regularized {\bf VS-SQN} scheme {\bf rVS-SQN} displays a rate of
$\mathcal{O}(1/k^{(1-\varepsilon)})$ with an oracle complexity of
$\mathcal{O}(1/\epsilon^3)$. When the smoothness requirements are weakened, the
rate for the regularized and smoothed {\bf VS-SQN} scheme {\bf rsVS-SQN}
worsens to $\mathcal{O}(k^{-1/3})$.  Such statements allow for a
state-dependent noise assumption under  a quadratic growth property on the
objective. To the best of our knowledge, the rate results are {amongst the
first available rates for QN methods in nonsmooth regimes.} Preliminary numerical evidence
suggests that the schemes compare well with accelerated gradient counterparts
on selected problems in stochastic optimization and machine learning with
significant benefits in ill-conditioned regimes.
\end{abstract}
\section{Introduction}
We consider the stochastic convex optimization problem 
\begin{align}\label{main problem} 
\min_{x\in \mathbb{R}^n} \ f(x)\triangleq \mathbb{E}[\aj{F(x,\xi(\omega))}],
 \end{align}
where $ \aj{\xi}: \Omega \rightarrow
\mathbb{R}^o$, ${F}: \mathbb{R}^n \times \mathbb{R}^o \rightarrow
\mathbb{R}$, {and} $(\Omega,\mathcal{F},\mathbb{P})$ denotes the associated
probability space. Such problems have broad applicability in engineering, economics, statistics, and machine learning.
Over the last two decades, two avenues for solving such problems have emerged
via sample-average approximation~(SAA)~\cite{kleywegt2002sample} and  stochastic
approximation (SA)~\cite{robbins51sa}. In this paper, we focus on
quasi-Newton variants of the latter. Traditionally, SA schemes have
been afflicted by a key shortcoming in {that such 
schemes display a markedly poorer convergence rate
than their deterministic variants.}  For instance, in standard stochastic
gradient schemes for strongly convex smooth problems {with 
Lipschitz continuous gradients}, the mean-squared error diminishes at a rate
of $\mathcal{O}(1/k)$ while deterministic schemes display a geometric rate of
convergence. \afr{This gap can be reduced by utilizing an increasing sample-size of
gradients, an approach
considered in~\cite{homem2003variable,deng2009variable,pasupathy2010choosing,FriedlanderSchmidt2012,byrd12},} and subsequently
refined for gradient-based methods for  strongly
convex~\cite{shanbhag15budget,jofre2017variance,jalilzadeh2018optimal},
convex~\cite{jalilzadeh16egvssa,ghadimi2016accelerated,jofre2017variance,jalilzadeh2018optimal},
and nonsmooth convex regimes~\cite{jalilzadeh2018optimal}.  Variance-reduced techniques have also been considered for stochastic quasi-Newton (SQN)
techniques~\cite{lucchi2015variance,zhou2017stochastic,bollapragada2018progressive}
under twice differentiability and strong convexity requirements. To the best
of our knowledge, the only available SQN scheme for merely convex but smooth
problems is the regularized SQN scheme presented in our prior
work~\cite{yousefian2017stochastic} where an iterative regularization of
the form ${1 \over 2} \mu_k \|x_k - x_0\|^2$ is employed to address the lack
of strong convexity while $\mu_k$ is driven to zero at a suitable rate.
Furthermore, a sequence of matrices $\{H_k\}$ is generated using a
regularized L-BFGS \us{update} or ({\bf rL-BFGS}) update. However, much
of the extant schemes in this regime either have gaps in the rates (compared
to deterministic counterparts) or cannot contend with nonsmoothness. \\

\begin{wrapfigure}[12]{r}{0.35\textwidth}
\vspace*{-1cm}
{\includegraphics[scale=0.23]{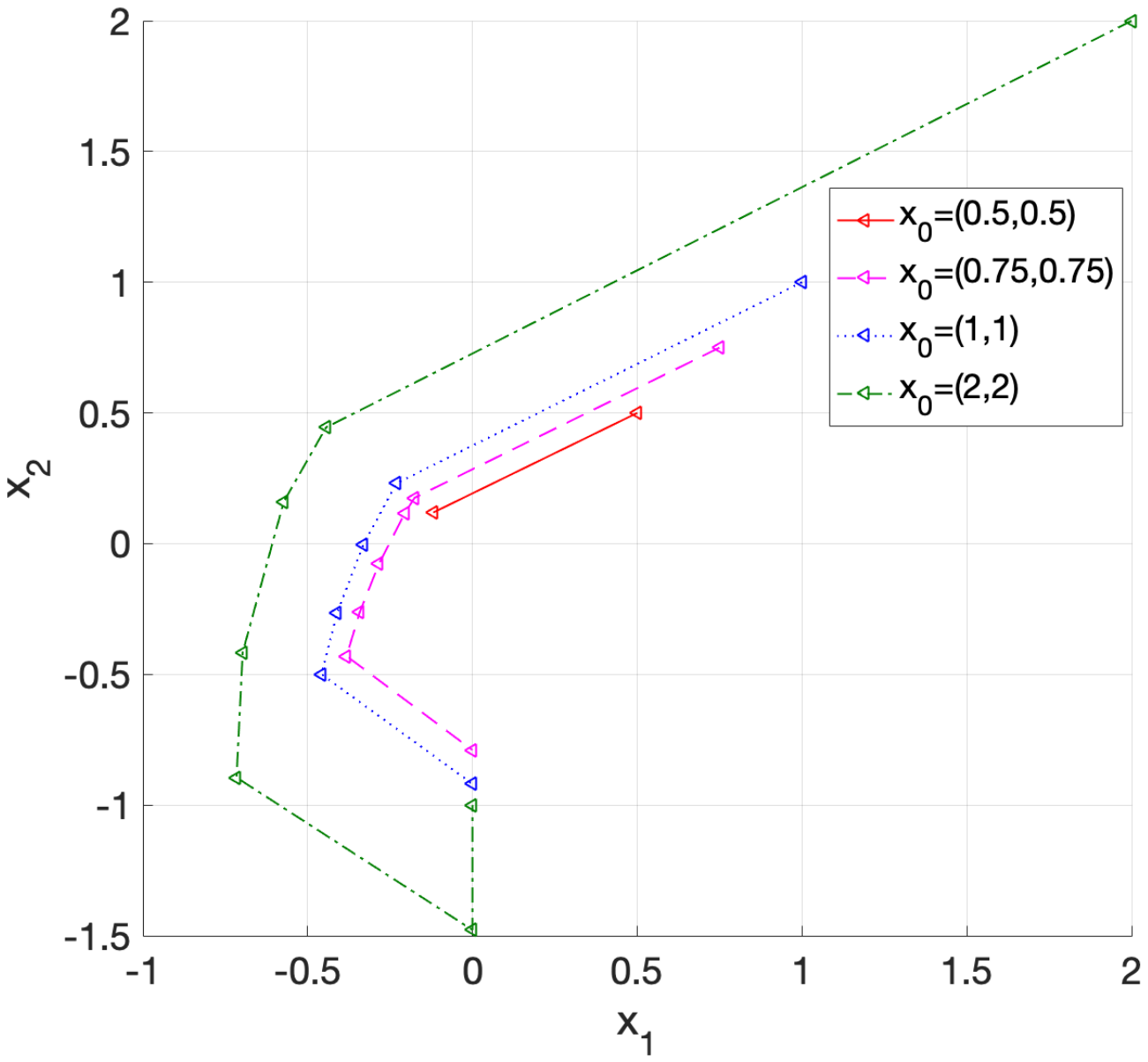}}\caption{Lewis-Overton example
	\label{fig:nssqn}}{}
	\end{wrapfigure}

\noindent {\bf Quasi-Newton schemes for nonsmooth convex problems.} { There
have been some attempts to apply (L-)BFGS directly to the deterministic
nonsmooth convex problems. But the method may fail as shown
in~\cite{lukvsan1999globally, haarala2004large, lewis2008behavior}; e.g.
in~\cite{lewis2008behavior}, the authors consider minimizing ${1\over
2}\|x\|^2+\max\{2|x_1|+x_2,3x_2\}$ in $\mathbb{R}^2$, BFGS takes a null step
(steplength is zero) for different starting points and  fails to converge to
the optimal solution $(0,-1)$ (except when initiated from $(2,2)$) (See
Fig.~\ref{fig:nssqn}).  Contending with nonsmoothness has been considered via a
subgradient quasi-Newton method \cite{yu2010quasi} for which global
convergence can be recovered by identifying a descent direction and
utilizing a line search.  An alternate approach~\cite{yuan2013gradient}
develops a globally convergent trust region quasi-Newton method in which Moreau
smoothing was employed. Yet, there appear to be neither non-asymptotic
rate statements available nor considerations of stochasticity in nonsmooth
regimes.\\  


\noindent {\bf Gaps.} Our research is motivated by several gaps. (i)
First, can we develop smoothed generalizations of ({\bf rL-BFGS}) that can
contend with nonsmooth problems in a seamless fashion? (ii) Second, can one
recover determinstic convergence rates (to the extent possible) by leveraging
variance reduction techniques? (iii) Third, can one address nonsmoothness on
stochastic convex optimization, which would allow for addressing more general
problems as well as accounting for the presence of constraints? (iv) Finally,
much of the prior results have stronger assumptions on the moment assumptions
on the noise which require weakening to allow for wider applicability of the
scheme.   
\subsection{{A survey of literature}} Before proceeding, we review some
relevant prior research in stochastic quasi-Newton methods and variable
sample-size schemes for stochastic optimization. {In Table~\ref{table
results}, we summarize the key advances in SQN methods where much of prior work
focuses on strongly convex (with a few exceptions). Furthermore, from
Table~\ref{table assumption}, it  can be seen that an assumption of twice
continuous differentiability and boundedness of eigenvalues on the true Hessian
is often made. In addition, almost all results rely on having a uniform bound
on the conditional second moment of stochastic gradient error. 
\begin{table}[htb]
	\scriptsize
	\begin{tabular}{|c|c|c|c|c|c|c|c|} \hline 
&Convexity&Smooth&$N_k$&$\gamma_k$&Conver. rate&Iter. complex.&Oracle complex.\\ \hline\hline
RES \cite{mokhtari2014res}&SC&\cmark&N&$1/k$&$\mathcal O(1/k)$&-&-\\ \hline
Block BFGS \cite{gower2016stochastic}&\multirow{3}{*}{SC}&\multirow{3}{*}{\cmark}& \multirow{2}{*}{N (full grad}&\multirow{3}{*}{$\gamma$}&\multirow{3}{*}{$\mathcal O(\rho^k)$}&\multirow{3}{*}{-}&\multirow{3}{*}{-}\\
Stoch. L-BFGS \cite{moritz2016linearly}&&&&&&&\\ 
&&&periodically) &&&&\\ \hline
SQN  \cite{wang2017stochastic}&NC&\cmark&$N$&$k^{-0.5}$& $\mathcal O(1/\sqrt k)$& $\mathcal O(1/\epsilon^2)$&- \\ \hline
\multirow{2}{*}{SdLBFGS-VR \cite{wang2017stochastic}}&\multirow{2}{*}{NC}&\multirow{2}{*}{\cmark}&$N$(full grad&\multirow{2}{*}{$\gamma$}&\multirow{2}{*}{ $\mathcal O(1/k)$}&\multirow{2}{*}{ $\mathcal O(1/\epsilon)$}&\multirow{2}{*}{-} \\ 
&&&periodically)&&&&\\ \hline
r-SQN \cite{yousefian2017stochastic}&C&\cmark&$1$&$k^{-2/3+\varepsilon}$&$\mathcal O(1/k^{1/3-\varepsilon})$&-&-\\ \hline
SA-BFGS \cite{zhou2017stochastic}&SC&\cmark&$N$&$\gamma_k$&$\mathcal O(\rho^k)$&$\mathcal O(\ln(1/\epsilon))$&$\mathcal O({1/ \epsilon^2}(\ln({1/\epsilon}))^4)$\\ \hline
Progressive&\multirow{2}{*}{NC}&\multirow{2}{*}{\cmark}&\multirow{2}{*}{-}&\multirow{2}{*}{$\gamma$}&\multirow{2}{*}{$\mathcal O(1/k)$}&\multirow{2}{*}{-}&\multirow{2}{*}{-}\\ 
Batching \cite{bollapragada2018progressive}&&&&&&&\\ \hline
Progressive &\multirow{2}{*}{SC}&\multirow{2}{*}{\cmark}&\multirow{2}{*}{-}&\multirow{2}{*}{$\gamma$}&\multirow{2}{*}{$\mathcal O(\rho^k)$}&\multirow{2}{*}{-}&\multirow{2}{*}{-}\\ 
Batching \cite{bollapragada2018progressive}&&&&&&&\\ \hline\hline
\eqref{VS-SQN}&SC&\cmark&$\lceil \rho^{-k}\rceil$&$\gamma$&$\mathcal O(\rho^k)$&$\mathcal O({\kappa}\ln(1/\epsilon))$&$\mathcal O(\kappa/\epsilon)$\\ \hline
\eqref{sVS-SQN}&SC&\xmark&$\lceil {\rho^{-k}}\rceil$&$ {\gamma}$&$\mathcal O({\rho^k})$&$\mathcal O({\ln(1/\epsilon)})$&$\mathcal O({1/\epsilon})$\\ \hline
\eqref{rVS-SQN}&C&\cmark&$\lceil k^a\rceil$&$k^{-\varepsilon}$&$\mathcal O(1/k^{1-\varepsilon})$&$\mathcal O(1/\epsilon^{1\over 1-\varepsilon})$&$\mathcal{O}(1/\epsilon^{(3+\varepsilon)/(1-\varepsilon)})$\\ \hline
\eqref{rsVS-SQN}& C&\xmark&$\lceil k^a\rceil$&$k^{-1/3+\varepsilon}$&$\mathcal O(1/k^{1/3})$&$\mathcal O(1/\epsilon^{{3}})$& $\mathcal O\left({1/ {\epsilon}^{{(2+\varepsilon)/( 1/3)}}}\right)$\\ \hline
	 	\end{tabular}
		\caption{Comparing convergence rate of related schemes (note that $a>1$)} \label{table results}
\end{table}

\begin{table}[htb]
\scriptsize
	\begin{tabular}{|c|c|c|c|p{3in}|} \hline 
&Convexity&Smooth&state-dep. noise&Assumptions\\ \hline\hline
RES \cite{mokhtari2014res}&SC&\cmark&\xmark&$\ulambda\mathbf{I}  \preceq H_k\preceq \olambda \mathbf{I}, \quad 0<\ulambda\leq \olambda$, $f$ is twice differentiable \\ \hline
Stoch. block BFGS \cite{gower2016stochastic}&\multirow{3}{*}{SC}&\multirow{3}{*}{\cmark}&\multirow{3}{*}{\xmark}&\multirow{2}{*}{$\ulambda\mathbf{I}  \preceq \nabla^2 f(x) \preceq \olambda \mathbf{I}, \quad 0<\ulambda\leq \olambda$, $f$ is twice differentiable}\\ 
Stoch. L-BFGS \cite{moritz2016linearly}&&&&\\ \hline
SQN for non convex \cite{wang2017stochastic}& NC &\cmark& \xmark& $\preceq \nabla^2 f(x) \preceq \olambda \mathbf{I}, \quad 0<\ulambda\leq \olambda$, $f$ is  differentiable\\ \hline
SdLBFGS-VR \cite{wang2017stochastic}&NC&\cmark&\xmark&$\nabla^2 f(x) \preceq \olambda \mathbf{I},\quad \olambda\geq 0$, $f$ is twice differentiable \\ \hline
r-SQN \cite{yousefian2017stochastic}&C&\cmark&\xmark&$\ulambda\mathbf{I}  \preceq H_k\preceq \olambda \mathbf{I}, \quad 0<\ulambda\leq \olambda$, $f$ is differentiable\\ \hline
\multirow{2}{*}{SA-BFGS \cite{zhou2017stochastic}}&\multirow{2}{*}{SC}&\multirow{2}{*}{\cmark}&\multirow{2}{*}{\xmark}&$f_k(x)$ is standard self-concordant for every possible sampling, The Hessian is Lipschitz continuous, \\
&&&&$\ulambda\mathbf{I}  \preceq \nabla^2 f(x) \preceq \olambda \mathbf{I}, \quad 0<\ulambda\leq \olambda$, $f$ is C$^2$\\  \hline
Progressive Batching \cite{bollapragada2018progressive}&NC&\cmark&\xmark& $\nabla^2f(x) \preceq \olambda \mathbf{I}, \quad \olambda\geq 0$, sample size is controlled by the exact inner product quasi-Newton test, $f$ is C$^2$\\ \hline
Progressive Batching \cite{bollapragada2018progressive}&SC&\cmark&\xmark&$\ulambda\mathbf{I}  \preceq \nabla^2 f(x) \preceq \olambda \mathbf{I}, \quad 0<\ulambda\leq \olambda$, sample size controlled by exact inner product quasi-Newton test,  $f$ is C$^2$\\ \hline \hline
\eqref{VS-SQN} &SC&\cmark&\cmark&$\ulambda\mathbf{I}  \preceq H_k\preceq \olambda_k \mathbf{I}, \quad 0<\ulambda\leq \olambda_k$, $f$ is C$^1$ \\ \hline
\eqref{sVS-SQN} &SC&\xmark&\cmark&$\ulambda_k\mathbf{I}  \preceq H_k\preceq \olambda_k \mathbf{I}, \quad 0<\ulambda_k\leq \olambda_k$\\ \hline
\multirow{2}{*}{\eqref{rVS-SQN} }&\multirow{2}{*}{C}&\multirow{2}{*}{\cmark}&\cmark&$\ulambda\mathbf{I}  \preceq H_k\preceq \olambda_k \mathbf{I},  0<\ulambda\leq \olambda_k$, $f$ is C$^1$, has quad. growth prop.\\  
&&&\xmark&$\ulambda\mathbf{I}  \preceq H_k\preceq \olambda_k \mathbf{I}$, $f$ is C$^1$ \\
\hline \multirow{2}{*}{\eqref{rsVS-SQN} }&\multirow{2}{*}{C}&\multirow{2}{*}{\xmark}&\cmark&$\ulambda_k\mathbf{I}  \preceq H_k\preceq \olambda_k \mathbf{I},  0<\ulambda\leq \olambda_k$, has quad. growth prop.\\ 
&&&\xmark&$\ulambda_k\mathbf{I}  \preceq H_k\preceq \olambda_k \mathbf{I}$  \\ \hline	 	\end{tabular}
		\caption{Comparing assumptions of related schemes }
	\label{table assumption}
\end{table}

\noindent {\bf (i) Stochastic quasi-Newton~(SQN) methods.} QN
schemes~\cite{liu1989limited,nocedal99numerical} have proved enormously influential in solving nonlinear programs, motivating the use of 
stochastic Hessian information~\cite{byrd12}. \aj{In 2014, Mokhtari and Riberio ~\cite{mokhtari2014res} introduced  a regularized stochastic version of the Broyden- Fletcher-Goldfarb-Shanno (BFGS) quasi-Newton method \cite{Fletcher} by updating the
matrix $H_k$  using a modified BFGS update rule} to ensure
convergence while limited-memory variants~\cite{byrd2016stochastic,
mokhtari2015global} and nonconvex generalizations~\cite{wang2017stochastic}
were subsequently introduced. In our prior work~\cite{yousefian2017stochastic},
an SQN method was presented for merely convex smooth problems, characterized by
rates of $\mathcal O(1/k^{{1\over 3}-\varepsilon})$ and $\mathcal
O(1/k^{1-\varepsilon})$ for the stochastic and deterministic case,
respectively. In~\cite{yousefian17smoothing}, via convolution-based
smoothing to address nonsmoothness, we provide a.s. convergence guarantees and
rate statements. 

\noindent {\bf (ii) Variance reduction schemes for stochastic optimization.}
Increasing sample-size schemes for finite-sum machine learning
problems~\cite{FriedlanderSchmidt2012,byrd12} have provided the basis for a
range of variance reduction schemes in machine
learning~\cite{roux2012stochastic,xiao2014proximal},
amongst reduction others.   By utilizing variable sample-size (VS) stochastic
gradient schemes, linear convergence rates were obtained for strongly convex
problems~\cite{shanbhag15budget,jofre2017variance} and these rates were
subsequently improved (in a constant factor sense) through a VS-{\em accelerated} proximal method developed by Jalilzadeh et
al.~\cite{jalilzadeh2018optimal} ({called ({\bf VS-APM})}). In convex regimes, Ghadimi and
Lan~\cite{ghadimi2016accelerated}  developed an accelerated framework that
admits the optimal rate of $\mathcal{O}(1/k^2)$ and the optimal oracle complexity (also see \cite{jofre2017variance}), improving the rate statement presented
in~\cite{jalilzadeh16egvssa}. More recently, in ~\cite{jalilzadeh2018optimal},
Jalilzadeh et al. present a smoothed accelerated scheme that admits the optimal
rate of $\mathcal{O}(1/k)$ and optimal oracle complexity for nonsmooth problems, 
recovering the findings in~\cite{ghadimi2016accelerated} in the smooth regime.
Finally, more intricate sampling rules are developed
in~\cite{bollapragada2017adaptive,pasupathy2018sampling}.

\noindent {\bf (iii) Variance reduced SQN schemes.} Linear~\cite{lucchi2015variance}
and superlinear~\cite{zhou2017stochastic} convergence statements for variance
reduced SQN schemes were provided in twice differentiable regimes under suitable
assumptions on the Hessian.  A ({\bf VS-SQN}) scheme with
L-BFGS~\cite{bollapragada2018progressive} was presented in strongly convex
regimes under suitable bounds on the Hessian. }  

\subsection{Novelty and contributions} 
{In this paper, we consider four variants of our proposed variable sample-size stochastic quasi-Newton method, \us{distinguished by whether the function $F(x,\omega)$ is strongly convex/convex and smooth/nonsmooth. The vanilla scheme is given by 
\begin{align}
x_{k+1}:=x_k-\gamma_kH_k{\frac{\sum_{j=1}^{N_k} u_k(x_k,\omega_{j,k})}{N_k}},
\end{align}
where $H_k$ denotes an approximation of the inverse of the Hessian, \af{$\omega_{j,k}$ denotes the $j^{th}$ realization of $\omega$ at the $k^{th}$ iteration}, 
$N_k$ denotes the sample-size at iteration $k$, and $u_k(x_k,\omega_{j,k})$ is
given by one of the following: (i)  ({\bf VS-SQN}) where $F(.,\omega)$ is
strongly convex and smooth, $u_k (x_k, \omega_{j,k}) \triangleq \nabla_x
F(x_k,\omega_{j,k})$; (ii) Smoothed ({\bf VS-SQN}) or ({\bf sVS-SQN}) where
$F(.,\omega)$ is strongly convex and nonsmooth and $F_{\eta_k}(x,\omega)$ is a smooth approximation of $F(x,\omega)$, $u_k (x_k, \omega_{j,k}) \triangleq
\nabla_x F_{\eta_k}(x_k,\omega_{j,k})$; (iii) Regularized ({\bf VS-SQN}) or ({\bf
rVS-SQN}) where $F(.,\omega)$ is convex and smooth  {and $F_{\mu_k}(.,\omega)$ is a regularization of $F(.,\omega)$}, $u_k (x_k, \omega_{j,k})
\triangleq \nabla_x F_{\mu_k}(x_k,\omega_{j,k})$; (iv)  regularized and smoothed
({\bf VS-SQN}) or ({\bf rsVS-SQN}) where $F(.,\omega)$ is convex and possibly
nonsmooth and $F_{\eta_k,\mu_k}(.,\omega)$ denotes a regularized smoothed approximation, $u_k (x_k, \omega_{j,k}) \triangleq \nabla_x
F_{\eta_k,\mu_k}(x_k,\omega_{j,k})$. We recap these definitions in the
relevant sections.  We briefly discuss our contibutions and accentuate the novelty of our work.\\

\noindent (I) {\em A regularized smoothed L-BFGS  update.} A regularized
smoothed L-BFGS update ({\bf rsL-BFGS}) is developed in Section~\ref{sec:rslbfgs},
extending the realm of L-BFGS scheme to merely convex and possibly nonsmooth
regimes by integrating both regularization and smoothing. As a consequence, SQN techniques can now contend with merely convex  {and nonsmooth} problems with convex constraints.\\

\noindent (II) {\em Strongly convex problems.} (II.i) ({\bf VS-SQN}). In Section~\ref{sec:3}, we
present a variable sample-size SQN scheme and prove that the convergence rate
is $\mathcal{O}(\rho^k)$ (where $\rho < 1$) while the iteration and oracle
complexity are proven to be $\mathcal{O}({\kappa^{m+1}} \ln(1/\epsilon))$ and
$\mathcal{O}(1/\epsilon)$, respectively. Notably, our findings are under a
weaker assumption of  {either} state-dependent noise (thereby extending the result
from~\cite{bollapragada2018progressive}) and do not necessitate assumptions of
twice continuous
differentiability~\cite{moritz2016linearly,gower2016stochastic} or Lipschitz
continuity of Hessian~\cite{zhou2017stochastic}.  (II.ii).  ({\bf sVS-SQN}). By integrating a smoothing parameter, we extend ({\bf VS-SQN}) to
contend with nonsmooth but smoothable objectives. Via Moreau smoothing,
we show that ({\bf sVS-SQN}) retains the optimal rate and complexity
statements of ({\bf VS-SQN}).} \jal{Additionally, in (II.i) and (II.ii), we derive rates that do not necessitate knowing either strong convexity or Lipschitzian parameters and rely on employing diminishing steplength sequences.} \\

\noindent (III) {\em Convex problems.} (III.i) ({\bf rVS-SQN}). A {\em
regularized ({\bf VS-SQN}) }  scheme is presented in Section~\ref{sec:4} based on  {the}
({\bf rL-BFGS}) update and admits a rate of
{$\mathcal{O}(1/k^{1-2\varepsilon})$} with an oracle complexity of $\mathcal
O\left({ {\epsilon}^{-{3+\varepsilon\over 1-\varepsilon}}}\right)$, improving
prior rate statements for SQN schemes for smooth convex problems and obviating
prior inductive arguments. In addition, we show that ({\bf rVS-SQN})
produces sequences that converge to the solution in an a.s. sense. Under a
suitable growth property, these statements can be extended to the
state-dependent noise regime.  (III.ii)  ({\bf rsVS-SQN}). {\em A regularized
smoothed $(${\bf VS-SQN}$)$} is presented that leverages the ({\bf rsL-BFGS})
update and allows for developing rate $\mathcal O(k^{-{1\over 3}})$ amongst the
first known rates for SQN schemes for nonsmooth convex programs. Again imposing
a growth assumption allows for weakening the requirements to
state-dependent noise.\\

\noindent (IV) {\em Numerics.} Finally, in Section~\ref{sec:5}, we apply the
({\bf VS-SQN}) schemes on strongly convex/convex and smooth/nonsmooth stochastic
optimization problems. In comparison with  variable sample-size accelerated
proximal gradient schemes, we observe that ({\bf VS-SQN}) schemes compete well and 
outperform gradient schemes for ill-conditioned problems when the number of QN
updates increases. {In addition, SQN schemes do far better in computing sparse solutions, in contrast with standard subgradient and variance-reduced accelerated gradient techniques.} {Finally}, via smoothing, ({\bf VS-SQN}) schemes can be seen to  resolve both nonsmooth and constrained problems.\\

{\bf Notation.} $\mathbb{E}[\bullet]$ denotes the expectation with respect to
the probability measure $\mathbb{P}$ and we refer to \aj{${\nabla_x}
{F}(x,\xi(\omega))$} by ${\nabla_x} {F}(x,\omega)$.  We denote the optimal objective
value (or solution) of  \eqref{main problem} by $f^*$ (or $x^*$) and the set of the optimal
solutions by $X^*$, {which is assumed to be nonempty}. {For a vector $x\in \mathbb R^n$ and a {nonempty} set $X \subseteq\mathbb R^n$, the Euclidean distance of $x$ from $X$ is denoted by ${\rm dist}(x,X)$.} \jal{Throughout the paper, {unless specified otherwise}, $k$ denotes the iteration counter {while} $K$ {represents} the total number of {steps employed in} the proposed methods.}

\section{{Background and Assumptions}}
In Section~\ref{sec:smooth}, we provide some background on smoothing
techniques  {and} then proceed to
define the {\em regularized and smoothed L-BFGS method} or {\bf(rsL-BFGS)}
update rule {employed for generating} the sequence of Hessian approximations
$H_k$ in Section~\ref{sec:rslbfgs}. We conclude this section with a summary of the main assumptions in Section~\ref{sec:assump}.

\subsection{Smoothing of nonsmooth convex functions} \label{sec:smooth}
We begin by defining of {$L$-smoothness} and $(\alpha,\beta)$-{\em smoothability}~\cite{beck17fom}.  
\begin{definition}
A function $f:\mathbb R^n\to \mathbb R$ is said to be $L$-smooth if it {is} differentiable and {there exists an $L > 0$ such that} $\|\nabla f(x)-\nabla f(y) \|\leq L\|x-y\|$ for all $x,y\in \mathbb R^n$. 
\end{definition} 
\begin{definition}{\bf [($\alpha,\beta$)-smoothable~\cite{beck17fom}]} {A
convex function $f: \mathbb{R}^n \to \mathbb{R}$ is 
$(\alpha,\beta)$-smoothable if there exists a convex C$^1$ function
$f_{\eta}: \mathbb{R}^n \to \mathbb{R}$ satisfying the following: (i)
$f_{\eta}(x) \leq f(x) \leq f_{\eta}(x)+\eta \beta$ for all $x$; and (ii)
$f_{\eta}(x)$ is $\alpha/\eta$-smooth.}  \end{definition} 
{Some instances of \usr{$(\alpha,\beta)$-smoothings}~\cite{beck17fom} include the following}: 
\noindent (i) If  $f(x) \triangleq \|x\|_2$ and $f_\eta(x) { \triangleq } \sqrt{\|x\|_2^2 + \eta^2} -
\eta$, then $f$ is $(1,1)-$smoothable function; \noindent (ii) If $f(x) \triangleq
\max\{x_1,x_2, \hdots, x_n\}$ and $f_{\eta}(x) { \triangleq } \eta
\ln(\sum_{i=1}^n e^{x_i/\eta})-\eta \ln(n)$, then $f$ is $(1,\ln(n))$-smoothable;  (iii) If $f$ is a proper, closed, and convex function and 
\begin{align} \label{moreau}
f_\eta(x) \triangleq \min_{u} \ \left\{f(u)+{1\over 2\eta}\|u-x\|^2 \right\},
\end{align} 
(referred to as Moreau proximal smoothing)~\cite{moreau1965proximite}, {then}
$f$ is $(1,B^2)$-smoothable  where 
$B$ denotes a uniform bound on $\|s\|$ where $s \in \partial f(x)$.  

It may be recalled that Newton's method is the de-facto standard for computing a zero
of a nonlinear equation~\cite{nocedal99numerical} while variants such as semismooth Newton methods
have been employed for addressing nonsmooth equations~\cite{facchinei1996inexact, facchinelt1997semismooth}. More
generally, in constrained regimes, such techniques take the form of
interior point schemes which can be viewed as the application of
Newton's method on the KKT system. Quasi-Newton variants of such
techniques can then we applied when second derivatives are either
unavailable or challenging to compute. However, in {constrained} stochastic
regimes, there has been far less available  via {a direct application of} {quasi-Newton} schemes. We consider a smoothing approach that leverages the unconstrained
reformulation of a constrained convex program
where $X$ is a closed and convex set and {${\mathbb I}_{X}(x)$ is an indicator
function}: \begin{align}
\tag{P} \min_x f(x) + {\mathbb I}_X(x).
\end{align}
Then the smoothed problem can be represented as follows:
\begin{align} \tag{P$_{\eta}$} 
\min_x f(x) + {\mathbb I}_{X,\eta}(x),
\end{align}
where ${\mathbb I}_{X,\eta}(\cdot)$ denotes the Moreau-smoothed variant of ${\mathbb I}_X(\cdot)$~\cite{moreau1965proximite} defined as follows. 
\begin{align}
{\mathbb I}_{X,\eta}(x) \triangleq \min_{u \in \mathbb{R}^n} \left\{ {\mathbb I}_X(u) + {1\over 2\eta} \|x-u\|^2\right\} = {1\over 2\eta} d_X^2(x), \usr{d_X(x) \triangleq (x-\mbox{prox}_{\mathbb{I}_X}(x)) = (x-\Pi_X(x))}, 
\end{align}
\usr{$\Pi_{X}(x)\triangleq \mbox{argmin}_{y\in X}\{\|x-y\|^2\}$}, and  the second equality follows from ~\cite[Ex.~6.53]{beck17fom}. Note that \afr{${\mathbb I}_{X,\eta}$} is continuously differentiable with gradient \afr{at $x$} given by $\tfrac{1}{2\eta} \nabla_x (d_X^2(x)) = {1\over \eta} (x - \mbox{prox}_{{\mathbb I}_X}(x)) = {1\over \eta} (x - \Pi_X(x))$. Our interest lies in
reducing the smoothing parameter $\eta$ after every iteration, a class of
techniques {(called {\em iterative smoothing schemes}) that have been applied for solving
} stochastic optimization~\cite{yousefian17smoothing,jalilzadeh2018optimal} and stochastic
variational inequality problems~\cite{yousefian17smoothing}. {Motivated by our recent
work~\cite{jalilzadeh2018optimal}  in which a smoothed variable sample-size
accelerated proximal gradient scheme is proposed for nonsmooth stochastic
convex optimization,} we consider a framework where at iteration $k$, an $\eta_k$-smoothed
function $f_{\eta_k}$ is utilized where the Lipschitz
constant of ${\nabla} f_{\eta_k}(x)$ is $1/\eta_k$. 

\subsection{Regularized and Smoothed L-BFGS Update }\label{sec:rslbfgs}
{When \vvs{the} function $f$ is strongly convex \vvs{but possibly nonsmooth}, we \vvs{adapt} the standard L-BFGS \vvs{scheme} (\vvs{by replacing the true gradient by a sample average}) where the approximation of the inverse Hessian $H_k$ is \vvs{defined} as follows using pairs $(s_i,y_i)$ and $\eta_i$ \vvs{denotes} a smoothing parameter:
\begin{align}\label{lbfgs}
s_i &:= x_{i}-x_{i-1},\\ \notag
\tag{\bf Strongly Convex (SC)} \ {y_i} & := { \sum_{j=1}^{{N_{i-1}}}{\nabla_x} {F}_{{\eta_{i}}}({x_{i}},{\omega}_{{j,i-1}})\over {N_{i-1}}} -{ \sum_{j=1}^{{N_{i-1}}}{\nabla_x} {F}_{{\eta_{{i}}}}({x_{i-1}},{\omega}_{{j,i-1}})\over {N_{i-1}}},\\ \notag
 H_{k,j}&:=\left(\mathbf{I}-\frac{y_is_i^T}{{y_i^Ts_i}}\right)^TH_{k,j-1}\left(\mathbf{I}-\frac{y_is_i^T}{y_i^Ts_i}\right)+\frac{s_is_i^T}{y_i^Ts_i},\quad i :=k-2(m-j), \ 1 \leq j\leq m, \ \forall i
\end{align}
where $H_{k,0}=\frac{s_k^Ty_k}{y_k^Ty_k}\mathbf{I}$.} At iteration $i$, we generate $\nabla_x F_{\eta_i}(x_i,\omega_{j,i-1})$ and $\nabla_x F_{\eta_i}(x_{i-1},\omega_{j,i-1})$, implying there are twice as many sampled gradients generated. 
\vvs{Next,} we discuss how the sequence of approximations $H_k$ is generated
{when} $f$ is merely convex {and} not necessarily smooth. We overlay the
regularized \vvs{L-BFGS}~\cite{mokhtari2014res,yousefian2017stochastic} scheme with a smoothing, refering to the proposed scheme as
the ({\bf rsL-BFGS}) update. As in {({\bf rL-BFGS})}~\cite{yousefian2017stochastic}, {we update the regularization and smoothing parameters $\{\eta_k ,\mu_k\}$ and matrix $H_k$ at alternate {iterations} to keep the secant condition satisfied.} We update the regularization parameter
$\mu_k$ and smoothing parameter $\eta_k$ as follows.
\begin{align}\label{eqn:mu-k}
  \begin{cases}
   \mu_{k}:=\mu_{k-1}, \quad \us{\eta_k := \eta_{k-1}},      &   \text{if } k \text{ is odd}\\
    \mu_{k}<\mu_{k-1}, \quad \us{\eta_k < \eta_{k-1}},  &    {\text{otherwise}.}
  \end{cases}
\end{align}
 We construct the update in terms of $s_i$ and $y_i$  {for convex problems},
\begin{align}\label{equ:siyi-LBFGS}  s_i &:= x_{i}-x_{i-1},\\\notag
\tag{\bf  Convex (C)}\ {y_i} & := { \sum_{j=1}^{{N_{i-1}}}{\nabla_x} {F}_{{\eta_{{i}}^{\delta}}}({x_{i}},{\omega}_{{j,i-1}})\over {N_{i-1}}} -{ \sum_{j=1}^{{N_{i-1}}}{\nabla_x} {F}_{{\eta_{{i}}^\delta}}({x_{i-1}},{\omega}_{{j,i-1}})\over {N_{i-1}}}+{\mu_i^{\bar \delta}}s_i,\end{align} 
where $i$ is odd and $0 < \delta,\bar \delta \leq 1$ are scalars controlling the level of smoothing and regularization in updating matrix $H_k$, respectively. The update policy for $H_k$ is given as follows: 
\begin{align}\label{eqn:H-k}H_{k}:=
  \begin{cases}
    H_{k,m},  &    \text{if } k \text{ is odd} \\
    H_{k-1},  &   \text{otherwise}
  \end{cases}
\end{align}
where $m<n$ (in large scale settings, $m<<n$) is a fixed integer that determines the number of pairs $(x_i,y_i)$ to be used to estimate $H_k$. The matrix $H_{k,m}$, for any $k\geq 2m-1$, is updated using the following recursive formula:
\begin{align}\label{eqn:H-k-m}
H_{k,j}&:=\left(\mathbf{I}-\frac{y_is_i^T}{{y_i^Ts_i}}\right)^TH_{k,j-1}\left(\mathbf{I}-\frac{y_is_i^T}{y_i^Ts_i}\right)+\frac{s_is_i^T}{y_i^Ts_i},\quad i :=k-2(m-j), \quad 1 \leq j\leq m, \quad \forall i,
\end{align}
{where $H_{k,0}=\frac{s_k^Ty_k}{y_k^Ty_k}\mathbf{I}$. It is important to note that our regularized method inherits the computational efficiency from ({\bf L-BFGS}). 
Note that {Assumption \ref{assum:convex2}} holds for our choice of
smoothing. 
%
%

\subsection{Main assumptions}\label{sec:assump}
\vvs{A subset of our results require smoothness of \usr{$F(\cdot,\omega)$} as formalized by the next assumption.}
\begin{assumption}\label{assum:convex-smooth}  
$($a$)$   The function $\usr{F(\cdot,\omega)}$ is {convex and continuously differentiable} over
$\mathbb R^n$ for any $\omega \in \Omega$.  
$($b$)$ The function $f$ is C$^1$ \usr{and $L$-smooth}
	over $\mathbb R^n$. 
\end{assumption}
We introduce the following assumptions of $F(\cdot,\omega)$, parts of which are imposed in a subset of results.
\begin{assumption}\label{assum:as_convex_smooth}
\noindent (a) For every $\omega$, $F(\cdot,\omega)$ is $\tau$-strongly convex.  \noindent (b)  For every $\omega$, $F(\cdot,\omega)$ is $L$-smooth.  (c) $f(x) \triangleq g(x)+h(x)$, where $g(x) \triangleq \mathbb{E}[F(x,\omega)]$, $F(\cdot,\omega)$ is $L$-smooth and $\tau$-strongly convex for every $\omega$, and $h$ is a closed, convex, and proper function. 
\end{assumption}
\noindent In Sections 3.2 (II) and 4.2, we assume the following on the smoothed functions \usr{${F}_{\eta}(\cdot,\omega)$}.
\begin{assumption}\label{assum:convex2}  
For any $\omega \in \Omega$, $\afr{{F}(\cdot,\omega)}$ is $(1, \beta)$ {smoothable}, i.e. \usr{for any $\eta > 0$, there exists $\afr{F_{\eta}(\cdot,\omega)}$ that is $C^1$, convex, ${1\over \eta}$-smooth, and satisfies $F_{\eta}(z,\omega) \leq F(z,\omega)\leq F_{\eta}(z,\omega) +\eta \beta$ for any $z \in \mathbb{R}^n$ and any $\omega \in \Omega$.}
	\end{assumption}
\noindent \afr{Let $\mathcal{F}_k \triangleq \sigma\{x_0, \{\omega_{j,0}\}_{j=1}^{N_0}, \hdots,
 \{\omega_{j,k}\}_{j=1}^{N_k}\}$. Now assume  the following on the conditional second moment on the
sampled gradient (in either the smooth or the smoothed regime) produced by the
stochastic first-order oracle.}
\vvs{\begin{assumption}[{\bf Moment requirements for state-dependent noise}]\label{state noise} 
\begin{enumerate}
\item[]
\item[] (Smooth)  Suppose $\bar{w}_{k,N_k} \triangleq \nabla_x f(x_k) -\tfrac{\sum_{j=1}^{N_k}\nabla_x
{F}(x_k,\omega_{j,k})}{N_k}$.

\item[(S-M)] There exist
$\nu_1, \nu_2>0$ such that $\mathbb E[\|{\bar{w}}_{k,N_k}\|^2\mid \mathcal F_k]\leq
{\tfrac{\nu_1^2\|x_k\|^2+\nu_2^2}{N_k}}$ a.s. for $k \geq 0$.  

\item[(S-B)] For $k \geq 0$, $\mathbb E[{\bar{w}}_{k,N_k}\mid \mathcal F_k] = 0$, a.s.  

\item[] (Nonsmooth)  Suppose $\bar{w}_{k,N_k} \triangleq {\nabla} f_{\eta_k}(x_k) - {\tfrac{\sum_{j=1}^{N_k}\nabla_x {F}_{\eta_k}(x_k,\omega_{j,k})}{N_k}}$, $\eta_k > 0$. 
\item[(NS-M)] There exist $\nu_1,\nu_2>0$ such that $\mathbb E[\|{\bar{w}}_{k,N_k}\|^2\mid \mathcal F_k]\leq {\tfrac{\nu_1^2\|x_k\|^2+\nu_2^2}{N_k}}$ a.s. for $k \geq 0$.  

\item[(NS-B)] For $k \geq 0$, $\mathbb E[{\bar{w}}_{k,N_k}\mid \mathcal F_k] = 0$, a.s. 

\usr{\item[] (Structured smooth) Suppose ${\bar{u}_k}=\nabla_{x}  g(x_{k})-{\tfrac{\aj{\sum_{j=1}^{N_k}}\nabla_x F(x_k,\omega_{j,k})}{N_k}}$.

\item[(SS-M)] There exist $\nu_1,\nu_2$ such that  $\mathbb{E}[\|{\bar u}_{k,N_k}\|^2\mid \mathcal F_k] \leq \usr{\tfrac{\nu_1^2\|x_k\|^2+\nu_2^2}{N_k}}$ a.s. for $k \geq 0$. 

\item[(SS-B)] For $k \geq 0$, $\mathbb{E}[\bar{u}_{k,N_k} \mid \mathcal F_k] = 0$, a.s. 
}
\end{enumerate}
\end{assumption}}
Finally, we impose Assumption~\ref{assump:Hk} on the sequence of
Hessian approximations $\{H_k\}$. These properties follow when
either the regularized update ({\bf rL-BFGS}), \usr{the smoothed update ({\bf sL-BFGS})}, or the regularized smoothed
update ({\bf rsL-BFGS}) \vvs{is} employed (see Lemmas \ref{H_k sc}, \ref{H_k ns sc}, \ref{rLBFGS-matrix}, and \ref{rsLBFGS-matrix}). 
\begin{property}[{\bf Properties of $H_k$}]\label{assump:Hk}
(i) \ $H_k$ is $\mathcal{F}_{k}$-measurable; (ii) \ $H_k$ is symmetric and positive definite and there exist  $\aj{\ulambda_k},{\olambda_k}>0$ such that 
$\aj{\ulambda_k}\mathbf{I} \preceq H_k \preceq {\olambda_k} \mathbf{I}$ {a.s.} for all $k\geq 0.$
\end{property}

\section{Smooth and nonsmooth strongly convex problems}\label{sec:3}
In this section, we {derive the} rate and oracle complexity of the \eqref{rVS-SQN} scheme for  smooth {and nonsmooth} strongly convex problems by
considering the \eqref{VS-SQN} and \eqref{sVS-SQN} schemes.

\subsection{Smooth strongly convex optimization} We begin by considering
~\eqref{main problem} when $f$ is {$\tau-$}strongly convex and $L-$smooth. {Suppose
$\kappa$ is defined as $\kappa \triangleq L/{\tau}$.}  {Throughout, we consider the
({\bf VS-SQN})} scheme, {defined next,  where $H_k$ is generated by the ({\bf
L-BFGS}) scheme. }
\begin{align}\tag{\bf VS-SQN}\label{VS-SQN}
x_{k+1}:=x_k-\gamma_kH_k\frac{\sum_{j=1}^{N_k} \nabla_x F(x_k,{\omega}_{j,k})}{N_k}.
\end{align}
Next, we derive bounds on the eigenvalues of $H_k$ under strong convexity (see {Appendix} for proof).
\begin{lemma}[{\bf Properties of $H_k$ {produced by} 
(L-BFGS)}]\label{H_k sc}
\usr{Suppose Assumptions~\ref{assum:convex-smooth} and \ref{assum:as_convex_smooth} (a,b) hold.} Consider the
\eqref{VS-SQN} method.  Let $s_i$, $y_i$ and $H_k$ be given by 
\eqref{lbfgs}, \usr{where $F_\eta=F$}. 
 Then $H_k$ satisfies \usr{Property} \ref{assump:Hk}{(S)}, with \usr{$\ulambda_k = \ulambda={1\over L(m+n)}$ and $\olambda_k = \olambda={\frac{\left((m+n)L\right)^{n+m-1}}{(n-1)!\tau^{n+m}}}$ for all $k$}.
\end{lemma}}

\begin{proposition}[{\bf Convergence in mean}]\label{thm:mean:smooth:strong}
Consider the iterates generated by the \eqref{VS-SQN} scheme. \usr{Suppose Assumptions~\ref{assum:convex-smooth}, \ref{assum:as_convex_smooth} (a,b),} and   \vvs{\ref{state noise} (S-M), (S-B)} hold.
In addition, suppose $\{N_k\}$ is an increasing sequence. 
Then  the following inequality holds
for all $k\geq 1$, where 
$N_0 \geq {2\nu_1^2\olambda\over \tau^2\ulambda}$ and
$\gamma_k \triangleq {1\over L\olambda}$ for all $k$.
\begin{align*}
 \mathbb E\left[f(x_{k+1})-f(x^*)\right]&
\leq\left(1-{\tau  \ulambda\over L\olambda}\red{+{2\nu_1^2\over L\tau N_0}}\right)\mathbb E\left[f(x_{k})-f(x^*)\right]+\red{ 2 \nu_1^2\|x^*\|^2+\nu_2^2\over 2LN_k}.
\end{align*}
\end{proposition}
\begin{proof} See Appendix. \end{proof}

\uvs{We now provide a result pivotal in deriving a rate and complexity statements under diminishing steplengths, an avenue that obviates knowing strong convexity and Lipschitzian parameters.}
\begin{lemma}{\em [\cite{polyak1987introduction}, Lemma 5]}\label{poly-rate} Suppose $\{u_k\}$ is a nonnegative sequence, where 
\begin{align}
	u_{k+1} \leq \left(1-\frac{c}{k^s}\right) u_k + \frac{d}{k^t}, \quad k \geq 0 
\end{align} 
where $0 < s < 1$, $s < t$, and $c, d > 0$. Then for $k \geq K$, 
$$ u_k \leq \frac{d}{ck^{t-s}} + o\left(\frac{1}{k^{t-s}}\right). $$ 
\end{lemma}

\begin{theorem}[{\bf Optimal rate and oracle  complexity}]\label{th1}
Consider the iterates generated by the \eqref{VS-SQN} scheme. \usr{Suppose Assumptions~\ref{assum:convex-smooth}, \ref{assum:as_convex_smooth} (a,b) and}   \usr{\ref{state noise} (S-M), (S-B)} hold. 
In addition, suppose $\gamma_k={1\over
L\olambda}$ for all $k$.

(i) If $a\triangleq \left(1-{\tau  \ulambda\over L\olambda}+{2\nu_1^2\over L\tau N_0}\right)$, $N_k \triangleq \lceil N_0\rho^{-k}\rceil$ where $\rho<1$ and $N_0   \geq  {2\nu_1^2\olambda\over \tau^2\ulambda}$. Then for {every $k \geq 1$} and some scalar $C$, the following holds:
$\mathbb E\left[f(x_{K+1})-f(x^*)\right]\leq C(\max\{a,\rho\})^{{K}}.$

\blue{(ii) Suppose $x_{K+1}$ is an $\epsilon$-solution such that $\mathbb{E}[f(x_{{K+1}})-f^*]\leq \epsilon$. 
Then the iteration and oracle complexity of \eqref{VS-SQN}  
are $\mathcal{O}({\kappa^{m+1}} \ln (1/\epsilon))$ 
and 
$\mathcal{O}({\kappa^{m+1} \over\epsilon})$, respectively implying that $\sum_{k=1}^K N_k \leq {\mathcal O\left({ \kappa^{m+n+1}\over \epsilon}\right)}.$  }

\jal{(iii) Suppose $\gamma_k = k^{-s}$ and $N_k = \lceil k^{{p-s}}\rceil$ for every $k$ where $0 < s < 1$ and $s < p$. \jal{In addition, suppose $c \triangleq \tfrac{\ulambda\tau}{2}$ and $d \triangleq { \olambda^2L(2\nu_1^2\|x^*\|^2+\nu_2^2)\over 2}$}. Then for $K$ sufficiently large, we have that 
$$\mathbb E\left[f(x_{k+1})-f(x^*)\right]  \leq \left(\frac{d}{ck^{p}}\right)  + o\left(\frac{1}{k^{p}}\right), \quad k \geq K.
$$}
\end{theorem}
\begin{proof} See Appendix. \end{proof}

We prove a.s. convergence of iterates by using the super-martingale convergence lemma from~\cite{polyak1987introduction}. 
\begin{lemma}[{\bf super-martingale convergence theorem}]\label{almost sure}
Let $\{v_k\}$ be a sequence of nonnegative random variables, where $\mathbb
E{[v_0]}<\infty$ and let $\{{\chi_k}\}$ and $\{\beta_k\}$ be deterministic scalar
sequences such that $0\leq {\chi_k} \leq 1$ and $\beta_k\geq 0$ for all $k\geq
0$, $\sum_{k=0}^\infty{\chi_k}=\infty$, $\sum_{k=0}^\infty \beta_k<\infty $, and
$\lim_{k\rightarrow \infty}{\beta_k\over {\chi_k}}=0$, and $\mathbb
E{[v_{k+1}\mid \mathcal F_k]\leq (1-{\chi_k})v_k+\beta_k}$ a.s. for all $k\geq
0$. Then, $v_k\rightarrow 0$ almost surely as $k\rightarrow \infty$.
\end{lemma}
\begin{theorem}[{\bf a.s. convergence under strong convexity}]
\usr{Consider the iterates generated by the \eqref{VS-SQN} scheme. \usr{Suppose Assumptions~\ref{assum:convex-smooth}, \ref{assum:as_convex_smooth} (a,b)}, and   \usr{\ref{state noise} (S-M), (S-B)} hold.}
 In addition, suppose $\gamma_k={1\over L\olambda}$ for all $k \geq 0$. Let $\{N_k\}_{k\geq 0}$ be an increasing sequence such that $\sum_{k=0}^\infty {1\over N_k}<\infty$ and $N_0>
{2\nu_1^2\olambda\over \tau^2\ulambda}$. Then $\lim_{k\rightarrow
\infty}f(x_k)=f(x^*)$ almost surely. 
\end{theorem}
\begin{proof} \usr{From Assumption~\ref{assum:as_convex_smooth} (a,b), $f$ is $\tau$-strongly convex and $L$-smooth.} Recall that in \eqref{strong:smooth}, we derived {the following for $k \geq 0$.}
\begin{align*}
\mathbb E\left[f(x_{k+1})-f(x^*)\mid \mathcal F_k\right]\nonumber&\leq\left(1-\tau\gamma_k\ulambda+{2\nu_1^2\over L\tau N_k}\right) (f(x_{k})-f(x^*))+{ 2\nu_1^2\|x^*\|^2+\nu_2^2\over 2LN_k}.
\end{align*}
If $v_k \triangleq f(x_k)-f(x^*)$, $\chi_k \triangleq
\tau\gamma_k\ulambda-{2\nu_1^2\over L\tau N_k}$, $\beta_k \triangleq {
2\nu_1^2\|x^*\|^2+\nu_2^2\over 2LN_k}$, $\gamma_k={1\over L{\olambda}}$, and
$\{N_k\}_{k\geq 0}$ be an increasing sequence such that $\sum_{k=0}^\infty
{1\over N_k}<\infty$  where $N_0> {2\nu_1^2\olambda\over \tau^2\ulambda}$,
(e.g. $N_k\geq {\lceil N_0k^{1+\epsilon}\rceil}$) the requirements of
Lemma~\ref{almost sure} are {seen to be} satisfied. {Hence},
$f(x_k)-f(x^*)\rightarrow 0$ {a.s.} as $k\rightarrow \infty$ and by strong
convexity of $f$, it follows that $\|x_k-x^*\|^2\to 0$ a.s.
\end{proof}

Having presented the variable sample-size SQN method, we now consider the special case where $N_k=1$. Similar to Proposition~\ref{thm:mean:smooth:strong}, the following inequality holds for $N_k=1$:
\begin{align}\label{bound sqn}
&\nonumber \ \mathbb E\left[f(x_{k+1})-f(x^*)\right]\leq f(x_k)-f(x^*)-\gamma_k \left(1-{L\over 2 }\gamma_k\olambda\right)\|H_k^{1/2}\nabla f(x_k)\|^2+{\gamma_k^2\olambda^2L(\nu_1^2\|x_k\|^2+\nu_2^2)\over 2}\\&\nonumber\leq \left(1-2\gamma_k{L^2\over \tau}\olambda(1-{L\over 2}\gamma_k\olambda)\right)\left(f(x_k)-f(x^*)\right)+{\gamma_k^2\olambda^2L(\nu_1^2\|x_k-x^*+x^*\|^2+\nu_2^2)\over 2}\\&
\leq \left(1-2\gamma_k\olambda{L^2\over \tau}+\gamma_k^2\olambda^2{L^3\over\tau}+{2\nu_1^2\gamma_k^2\olambda^2L\over \tau}\right)\left(f(x_k)-f(x^*)\right)+{\gamma_k^2\olambda^2L(2\nu_1^2\|x^*\|^2+\nu_2^2)\over 2},
\end{align}
where the second inequality is obtained by using Lipschitz continuity of $\nabla f(x)$ and the strong convexity of $f(x)$. Next, to obtain the convergence rate of SQN, we use the following lemma~\cite{xie2016si}. 
\begin{lemma}\label{induction}
Suppose $e_{k+1}\leq (1-2a\gamma_k+\gamma_k^2b)e_k+\gamma_k^2c$ for all $k\geq 1$. Let $\gamma_k=\gamma/k$, $\gamma>1/(2a)$, $K\triangleq\lceil {\gamma^2b\over 2a\gamma-1} \rceil+1$ and $Q(\gamma,K)\triangleq \max \left\{{\gamma^2c\over 2a\gamma-\gamma^2b/K-1},Ke_K\right\}$. Then $\forall k\geq K$, $e_k\leq {Q(\gamma,K)\over k}$. 
\end{lemma}
Now from inequality \eqref{bound sqn} and Lemma \ref{induction},  the following proposition follows.
\begin{proposition}[{\bf Rate of convergence of SQN with $N_k = 1$}]
\usr{Consider the iterates generated by the \eqref{VS-SQN} scheme. \usr{Suppose Assumptions~\ref{assum:convex-smooth}--\ref{assum:as_convex_smooth},} \usr{and}   \vvs{\ref{state noise} (S-M),(S-B)} hold.} 
Let $a={L^2\olambda\over \tau}$, $b={\olambda^2L^3+2\nu_1^2\olambda^2L\over \tau}$ and $c={\olambda^2L(2\nu_1^2\|x^*\|^2+\nu_2^2)\over 2}$. 
Then, $\gamma_k={\gamma\over k}$, $\gamma>{1\over L \olambda}$ and $N_k=1$ the following holds:
$\mathbb E\left[f(x_{k+1})-f(x^*)\right]\leq {Q(\gamma,K)\over k}$,
where $Q(\gamma,K)\triangleq \max \left\{{\gamma^2c\over 2a\gamma-\gamma^2b/K-1},K(f(x_K)-f(x^*))\right\}$ and $K\triangleq\lceil {\gamma^2b\over 2a\gamma-1} \rceil+1$. 
\end{proposition}
\blue{\begin{remark}
It is worth emphasizing that the
proof techniques, while aligned with avenues adopted in~\cite{byrd12,FriedlanderSchmidt2012,bollapragada2018progressive}, {extend results}
in~\cite{bollapragada2018progressive} to the regime of state-dependent
noise~\cite{FriedlanderSchmidt2012}. We also observe that  in the analysis of deterministic/stochastic first-order methods,
any non-asymptotic rate statements rely on utilizing problem parameters (e.g. the strong convexity modulus, Lipschitz constants, etc.). \uvs{Similarly, in the context of 
QN} methods, obtaining non-asymptotic bounds also requires $\ulambda$ and $\olambda$ (cf.~\cite[Theorem 3.1]{bollapragada2018progressive},
\cite[Theorem 3.4]{berahas2016multi}, and \cite[Lemma~2.2]{wang2017stochastic})
{since the impact of $H_k$ needs to be addressed.} One {avenue for weakening the dependence on such parameters lies} 
in using line search schemes. However when the problem is
expectation-valued, the steplength arising from a line search leads to a
{dependence between the} steplength (which is now random) and the direction.
Consequently, standard analysis fails and one has to appeal to more refined analysis
(cf.~\cite{iusem2017variance,cartis18global,paquette20stochastic,shashaani18astro}). {This remains the focus of future work.}
\end{remark}}
\jal{\begin{remark}
Note that Assumption \ref{assum:as_convex_smooth}, $F(\cdot,\omega)$ is $L$-smooth and $\tau$-strongly convex for every $\omega$ and is commonly employed in stochastic quasi-Newton schemes; cf.~\cite{berahas2016multi,byrd2016stochastic,mokhtari2015global}. \uvs{This is necessitated by the need to provide bounds on the eigenvalues of $H_k$.} 
{While deriving such Lipschitz constants is challenging, there are instances when this is possible. 

\noindent (i) Consider the following problem. 
$$ \min_{x \in X} f(x) \triangleq  \mathbb{E}\left[\tfrac{1}{2} x^TQ(\omega) x + c^Tx\right]. $$ If one has access to the form of $Q(\omega)$, then by leveraging Jensen's inequality and under suitable integrability requirements, one may be able to prove that $f$ is $L$-smooth. More generally, if $\nabla_x f(x) = \mathbb{E}[\nabla_x f(x,\omega)]$ and $\nabla_x f(\cdot,\omega)$ is $L(\omega)$-Lipschitz where $L(\omega)$ has finite mean given by $L$, then one may conclude that $f$ is $L$-smooth. We may either derive $L$ (if we have access to the structure of $L(\omega)$) or postulate the existence of $L$ if $L(\omega)$ has finite expectation.}  

\noindent (ii) Suppose we consider the setting when $F(\cdot,\xi)$ is the $\ell_2$-squared loss function; i.e. 
$F(x,\xi_i) \triangleq \frac{1}{2}(a_i^Tx - b_i)^2$ where $a_i \in
\mathbb{R}^{n}$ and $b_i \in \mathbb{R}$ denote the $i^{th}$ pair of the input
and output data, respectively, and $\xi_i \triangleq (a_i; b_i) \in
\mathbb{R}^{n+1}$. We obtain that $\nabla F(x,\xi_i) = (a_i^Tx - b_i)a_i =
(a_ia_i^T)x - b_ia_i.$ This implies that $\nabla F(\cdot,\xi_i)$ is a Lipschitz
continuous mapping with the parameter $L_{\xi_i}:=\|a_ia_i^T\|_2$. If 
$\xi$ is assumed to have finite support based on an empirical distribution, we have that
$$\nabla f(x) = \mathbb{E}\left[\nabla F(x,\xi)\right] =
\frac{1}{N}\sum_{i=1}^N \nabla F(x,\xi_i).$$
From the preceding relation and that every sample path function \uvs{ is $L(\xi)$-smooth}, for any $x,y \in \mathbb{R}^n$:
\begin{align*}\|\nabla f(x) -\nabla f(y)\|_2 &= \left\|\mathbb{E}\left[\nabla F(x,\xi_i)-\nabla F(y,\xi_i)\right]\right\|_2 \leq \frac{\sum_{i=1}^N\|\nabla F(x,\xi_i)-\nabla F(y,\xi_i)\|_2}{N} \\&\leq \frac{\sum_{i=1}^N \|a_ia_i^T\|_2\|x-y\|_2}{N}.\end{align*}
This implies that $f(x) \triangleq \mathbb{E}\left[F(x,\xi)\right]$ has Lipschitz gradients with the parameter $L\triangleq \frac{1}{N}\sum_{i=1}^N \|a_ia_i^T\|_2$. We, however, note that the computation of $L$ may become costly in cases where either $N$ or $n$ are massive. This can be addressed to some extent by deriving an upper bound on $L$ as follows: 
$$L\triangleq \frac{\sum_{i=1}^N \|a_ia_i^T\|_2}{N} \leq \frac{\sum_{i=1}^N \|a_i\|_{\infty}\|a_i\|_2}{N}  \leq \frac{\max_{i}\|a_i\|_{\infty}}{N}\sum_{i=1}^N \|a_i\|_2 \leq \fythird{\frac{\|A\|_F\max_{i}\|a_i\|_{\infty}}{\sqrt{N}}},$$
where $A \in \mathbb{R}^{N \times n}$ is defined as $A = [a_1^T;\ldots;a_N^T]$ and $\|A\|_F$ denotes the Frobenius norm of $A$.\\

\end{remark}}
\vspace{-0.2in}
\subsection{Nonsmooth strongly convex optimization} 
Consider \eqref{main problem} where $f$ is a strongly convex but
nonsmooth function. In this section, we focus on the case where $f(x) \triangleq h(x) + g(x)$,  $h$ is a deterministic, closed, convex, and proper function, \afr{$g$}  is $L-$smooth and strongly convex, $F(\afr{\cdot},\omega)$ is a convex function for every $\omega$}, \afr{where} $g(x) \triangleq \mathbb{E}[F(x,{\omega})]$. We begin by noting that the Moreau envelope of $f$, denoted by \afr{$f_{\eta}$} and defined as \eqref{moreau}, retains both the minimizers of $f$ as well as its strong convexity as captured
by the following result based on \cite[Lemma~2.19]{planiden2016strongly}.
\begin{lemma}\label{feta}
Consider a convex, closed, and proper function $f$ and its Moreau envelope $f_{\eta}$. Then the following hold:
(i) $x^*$ is a minimizer of $f$ over $\mathbb{R}^n$ if and only if $x^*$ is a minimizer of $f_{\eta}(x)$; (ii) $f$ is $\sigma$-strongly convex on $\mathbb{R}^n$ if and only if $f_{\eta}$ is $\tfrac{\sigma}{\eta\sigma+1}$-strongly convex on $\mathbb{R}^n$.
\end{lemma}

Consequently, it suffices to minimize the (smooth) Moreau envelope with a
{\em fixed} smoothing parameter $\eta$, as shown in the next result. For
notational simplicity, we choose $m=1$ but the rate results hold for $m>1$ and  define $f_{N_k}(x) \triangleq h(x) + \tfrac{1}{N_k}\sum_{j=1}^{N_k}
F(x,\aj{\omega_{j,k}})$.
{Throughout this subsection, we consider the smoothed variant of
\eqref{VS-SQN}, referred to the \eqref{sVS-SQN} scheme, defined next, where
$H_k$ is generated by the ({\bf sL-BFGS})} update rule, $\nabla_x
f_{\eta_k}(x_k)$ denotes the gradient of the Moreau-smoothed function,
{given by $\tfrac{1}{\eta_k}(x_k- \mbox{prox}_{\eta_k,f}(x_k))$, while
$\nabla_{x} f_{\eta_k,N_k}(x_k) $, the gradient of the Moreau-smoothed and
sample-average function $f_{N_k}(x)$, is defined as $\tfrac{1}{\eta_k}(x_k-
\mbox{prox}_{\eta_k,f_{N_k}}(x_k))$ {and} ${\bar w_k} \triangleq \nabla_x
{f_{\eta_k,N_k}(x_k)-\nabla_x
{f_{\eta_k}(x_k)}}$. Consequently the update rule for $x_k$ becomes the following}.  
\begin{align}\tag{\bf sVS-SQN}\label{sVS-SQN}
x_{k+1}:=x_k-\gamma_kH_k{\left(\nabla_{x} f_{\eta_k}(x_k)+{\bar w_k}\right)},
\end{align}
\jal{and the update rule of L-BFGS that we use in this section is as follows:
\begin{align*}
s_i &:= x_{i}-x_{i-1},
\quad {y_i} := \nabla_x f_{\eta_i,N_{i-1}}(x_i)- \nabla_x f_{\eta_i,N_{i-1}}(x_{i-1}),\\ \notag
 H_{k,j}&:=\left(\mathbf{I}-\frac{y_is_i^T}{{y_i^Ts_i}}\right)^TH_{k,j-1}\left(\mathbf{I}-\frac{y_is_i^T}{y_i^Ts_i}\right)+\frac{s_is_i^T}{y_i^Ts_i},\quad i :=k-2(m-j), \ 1 \leq j\leq m, \ \forall i.
\end{align*}}
At each iteration of \eqref{sVS-SQN}, {the error in the gradient is captured by  ${\bar{w}}_k$.}
We show that ${\bar{w}}_k$ satisfies Assumption \ref{state noise} (NS) by utilizing the following {assumption on the gradient of function}.

\begin{assumption}\label{Lip_moreau}
Suppose there exists {$\nu_1,\nu_2>0$} such that for all $i\geq 1$, $\mathbb{E}[\|{\bar u}_{{k}}\|^2\mid \mathcal F_k] \leq {\tfrac{\nu_1^2\|x_k\|^2+\nu_2^2}{N_k}}$ holds almost surely, where {${\bar{u}_k}=\nabla_{x}  g(x_{k})-{\tfrac{{\sum_{j=1}^{N_k}}\nabla_x F(x_k,\omega_{j,k})}{N_k}}$}. 

\end{assumption}
\begin{lemma}\label{bound sub}
\usr{Suppose Assumptions~\ref{assum:convex-smooth}, \ref{assum:as_convex_smooth} (a,c), and ~\ref{state noise} hold.}  Let $f_{\eta}$ {denote the Moreau smoothed approximation of $f$} {and $\eta<2/L$}. Then, \usr{$\mathbb E[\|\bar w_k\|^2\mid \mathcal F_k]\leq \tfrac{4(\nu_1^2\|x_k\|^2+\nu_2^2)}{{\usr{\tau^2 \eta^2}N_k}}$ for all $k\geq 0$.} 
\end{lemma}
\begin{proof}
\afr{We begin by noting that $f_{N_k}(x)$ is convex.} 
Consider the two problems:
\begin{align}
	\mbox{prox}_{\eta, f}(x_k) & \triangleq \mbox{arg} \min_{u} \left[ f(u) + {1\over 2\eta} \|x_k-u\|^2\right], \label{prox1} \\  	
\mbox{prox}_{\eta, f_{N_k}}(x_k) & \triangleq \mbox{arg} \min_{u} \left[ f_{N_k}(u) + {1\over 2\eta} \|x_k-u\|^2\right]. \label{prox2}
\end{align}
{Suppose  $x^*_{{k}}$ and $x^*_{N_k}$ denote
the optimal unique} solutions of \eqref{prox1} and \eqref{prox2}, respectively. From the definition of Moreau smoothing, {it follows that}
\begin{align*}
 \bar{w}_k & = \nabla_x f_{\usr{\eta},N_k}(x_k) - \nabla_x f_{\usr{\eta}}(x_k) = {1\over \usr{\eta}} (x_k - \mbox{prox}_{\usr{\eta}, f_{N_k}}(x_k)) - 
{1\over \usr{\eta}} (x_k - \mbox{prox}_\aj{\usr{\eta}, f}(x_k)) \\
	&  = \mbox{prox}_{\usr{\eta},f_{N_k}}(x_k) - \mbox{prox}_{\usr{\eta}, f} (x_k) = {1\over \eta} (x_{N_k}^* - x_k^*),\end{align*}  
{which implies} $\mathbb E[\|\bar w_k\|^2\mid \mathcal F_k]={1\over \eta^2}\mathbb E[\|x^*_{{k}}-x^*_{N_k}\|^2\mid \mathcal F_k]$. \afr{The following inequalities are a consequence of invoking strong convexity of $f$, convexity of $f_{N_k}$} and the  optimality conditions of \eqref{prox1} and \eqref{prox2}: 
\begin{align*}
	f(x^*_{N_k}) + {1\over 2\eta} \|x^*_{N_k} - x_k\|^2 &\geq f(x^*_{{k}}) + {1\over 2\eta} \|x^*_{{k}} - x_k\|^2 + {1\over 2} \left(\tau + {1\over \eta}\right) \|x^*_{{k}}-x^*_{N_k}\|^2, \\  		 
	f_{N_k}(x^*_{{k}}) + {1\over 2\eta} \|x^*_{{k}} - x_k\|^2 &\geq f_{N_k}(x^*_{N_k})+ {1\over 2\eta} \|x^*_{N_k} - x_k\|^2\afr{+ {1\over 2\eta} \|x^*_{N_k}-x^*_{{k}}\|^2.}
\end{align*}  		 
Adding the above inequalities, we have that 
\begin{align*}
	f(x^*_{N_k}) -f_{N_k}(x^*_{N_k}) + f_{N_k}(x_{{k}}^*)-f(x_{{k}}^*)  &\geq \left(\afr{\tau\over 2}+{1\over \eta}\right)\|x^*_{N_k}-x_{{k}}^*\| ^2. 
\end{align*}
 From the definition of $f_{N_k}(x_k)$ and \afr{$\beta \triangleq {\tau\over 2}+\tfrac{1}{\eta}$}, 
and \afr{by the convexity of $F(\cdot,\omega)$ in $x$ for a.e. $\omega$ and  and $L-$smoothness of function $g$}, {we may prove} the following.
 \begin{align*}
\beta \|x^*_{{k}}&-x^*_{N_k}\|^2  \leq {f(x^*_{N_k}) -f_{N_k}(x^*_{N_k}) + f_{N_k}(x_{{k}}^*)-f(x_{{k}}^*)}
			 \\ 
& = \frac{\sum_{j=1}^{N_k} (g(x^*_{N_k})-F(x^*_{N_k},\aj{\omega_{j,k}}))}{N_k} + \frac{\sum_{j=1}^{N_k} (F(x^*_{{k}},\aj{\omega_{j,k}})-g(x^*_{{k}}))}{N_k}  \\
	& \leq \frac{\sum_{j=1}^{N_k} \left(g(x^*_k) +\nabla_{x}  g(x^*_{k})^T(x^*_{N_k}-x^*_k)+ \tfrac{L}{2}\|x^*-x^*_{N_k}\|^2 - F(x^*_k,\aj{\omega_{j,k}}) - \nabla_x F(x^*_{k},\aj{\omega_{j,k}})^T(x^*_{N_k}-x^*_k)\right)}{N_k}\\
& + \frac{\sum_{j=1}^{N_k} (F(x^*_{{k}},\aj{\omega_{j,k}})-g(x^*_{{k}}))}{N_k}  
	 = \frac{\sum_{j=1}^{N_k} (\nabla_{x}  g(x^*_{k})-\nabla_x F(x^*_k,\aj{\omega_{j,k}}))^T(x^*_{N_k}-x^*_k)}{N_k}+ \tfrac{L}{2}\|x^*-x^*_{N_k}\|^2\\
	& = {\bar{u}_k^T}(x^*_{N_k}-x^*_k)+ \tfrac{L}{2}\|x^*_{{k}}-x^*_{N_k}\|^2.
\end{align*} 
Consequently, by taking conditional expectations {and using Assumption \ref{Lip_moreau}}, we have the following.
\begin{align*}
{\mathbb{E}[\beta \|x^*_{{k}}-x^*_{N_k}\|^2 \mid \mathcal{F}_k]} & =  
	 \mathbb E[{\bar{u}_k^T}(x^*_{N_k}-x^*_k)\mid \mathcal F_K]+ \tfrac{L}{2}\mathbb E[\|x^*_{{k}}-x^*_{N_k}\|^2\mid \mathcal F_k]\\
	& \leq \tfrac{1}{\tau} \mathbb{E}[\|{\bar{u}_k}\|^2 \mid \mathcal{F}_k] + (\tfrac{\tau}{4}+\tfrac{L}{2})\mathbb{E}[\|x^*_k-x^*_{N_k}\|^2 \mid \mathcal{F}_k]\\
\implies {\mathbb{E}[ \|x^*_{{k}}-x^*_{N_k}\|^2 \mid \mathcal{F}_k] }&  \leq \tfrac{4}{{\tau^2}} \mathbb{E}[\|{\bar u_k}\|^2 \mid \mathcal{F}_k]  \leq \tfrac{4}{{\tau^2}} \tfrac{\usr{\nu_1^2\|x_k\|^2 + \nu_2^2}}{N_k}, \mbox{ if $\eta < 2/L$.} 
\end{align*}
We may then conclude that \afr{$\mathbb E[\|\bar w_{k,N_k}\|^2\mid \mathcal F_k]\leq \tfrac{4\usr{(\nu_1^2\|x_k\|^2+\nu_2^2)}}{\eta^2\tau^2N_k}$}.
\end{proof}

{Next, we derive bounds on the eigenvalues of $H_k$ under strong convexity (similar to  Lemma~\ref{H_k sc})}.
\begin{lemma}[{\bf Properties of  {$H_k$ produced by} 
 (sL-BFGS)}]\label{H_k ns sc}
\usr{Suppose Assumptions~\ref{assum:convex-smooth} and \ref{assum:as_convex_smooth} (a,c) hold.} 
Let $s_i$, $y_i$ and $H_k$ be given by \eqref{lbfgs}.
 Then $H_k$ satisfies \usr{Property} \ref{assump:Hk} 
with ${\ulambda_k}={\eta_k\over (m+n)}$ and {${\olambda_k}=\left({n+m\over \eta_k}\right)^{m+n-1}\left({1\over (n-1)!\tau^{n+m}}\right)$}.
\end{lemma}
We now show that under Moreau smoothing, {a} linear rate of convergence is retained. 

\begin{theorem}\label{moreau_strong}
Consider the iterates generated by the \eqref{sVS-SQN} scheme \usr{where $\eta_k = \eta$ for all $k$. Suppose Assumptions~\ref{assum:convex-smooth}, \ref{assum:as_convex_smooth}(a,c), ~\ref{state noise}(SS-M, SS-B) {and \ref{Lip_moreau} hold}.} 
 In addition, suppose $m = 1$, {$\eta\leq \min\{2/L,(\usr{8}(n+1)^2/\tau^2)^{1/3}\}$}, $d \triangleq 1-{\tau^2\eta^3\over {\usr{8}}(n+1)^2(1+\eta\tau)}$, $N_k\triangleq \lceil N_0{q^{-k}}\rceil$ for all $k\geq 1$, $N_0 \geq \frac{5(1+\eta \tau)^2}{\tau^4} \usr{(\nu_1^2)\over \eta \gamma \ulambda}$ $\gamma\triangleq{\tau\eta^2\over
{4}(1+n) }$, $c_1 \triangleq \max\{q,d\}$, and $c_2 \triangleq \min\{q,d\}$. (i) \usr{(a) Suppose $c_1 > c_2$.}  Then $\mathbb E[\|x_{{k}+1}-x^*\|^2]
 \leq D c_1^{{k}+1}$ for all $k$ where
where \usr{$$  D  \triangleq \left({2\mathbb E[f_\eta(x_0)-f_\eta(x^*)](1+\eta\tau)\over \tau}\right) \usr{ \ + \ } \left(\usr{10(1+\eta \tau)(2\nu_1^2\|x^*\|^2+\nu_2^2)\over 4\tau^3\eta N_0(\usr{c_1}-\usr{c_2})}\right). $$}
\usr{(b) Suppose $c_1 = c_2$. Then 
$\mathbb E[\|x_{K+1}-x^*\|^2]\leq
D \tilde{d}^{K+1}$  where 
$\tilde{d} \in (d,1)$, $\tilde{D} > \tfrac{1}{\ln(\tilde{d}/d)^e}$, and 
\usr{$$  D  \triangleq \left({2\mathbb E[f_\eta(x_0)-f_\eta(x^*)](1+\eta\tau)\over \tau}\right) \usr{ \ + \ } \left(\usr{10(1+\eta \tau)(2\nu_1^2\|x^*\|^2+\nu_2^2)\tilde{D}\over 4\tau^3\eta N_0d }\right). $$
}}
(ii) \blue{Suppose $x_{{K+1}}$ is an $\epsilon$-solution such that
$\mathbb{E}[f(x_{{K+1}})-f^*]\leq \epsilon$. {Then, the iteration and
oracle complexity of computing $x_{K+1}$} {are} 
$\mathcal{O}(\ln(1/\epsilon))$ steps and {$\mathcal
O\left({ 1/ \epsilon}\right)$}, respectively.}

\noindent (iii) \jal{Suppose $\gamma_k = k^{-s}$ and $N_k = \lceil k^p\rceil$ where $0 < s < 1$ and $s < p$. Suppose $c \triangleq \tfrac{\ulambda\tau}{2(1+\eta\tau)}$, $d\triangleq {\left({\olambda^2\over \eta}+{\eta\over 4}\right){{(8\nu_1^2\|x^*\|^2+4\nu_2^2)}\over {\tau^2 \eta^2} }}$, and $\tilde{d} \triangleq \tfrac{(1+\eta \tau)d}{\tau}$.  Then for $K$ sufficiently large, we have that 
\begin{align}
\notag & \quad \mathbb E\left[\|x_k-x^*\|^2 \right] \leq \frac{\tilde{d}}{ck^p} + o\left(\frac{1}{k^p}\right), \mbox{ for } k \geq K.
\end{align}}
\end{theorem}
\begin{proof}
(i) From Lipschitz continuity of $\nabla f_{\eta}$ and update \eqref{sVS-SQN}, we have the following:
\begin{align*}
f_{\eta}(x_{k+1})&\leq f_{\eta}(x_k)+\nabla f_{\eta}(x_k)^T(x_{k+1}-x_k)+{1\over 2\eta}\|x_{k+1}-x_k\|^2\\&
=f_{\eta}(x_k)+\nabla f_{\eta}(x_k)^T\left(-{\gamma}H_k(\nabla f_\eta(x_k)+\bar w_{k,N_k})\right)+{1\over 2\eta}{\gamma}^2\left\|H_k(\nabla f_\eta(x_k)+\bar w_{k,N_k})\right\|\uvs{^2}\\
&{=f_{\eta}(x_k)-\gamma\nabla f_\eta(x_k)^TH_k\nabla f_\eta(x_k)-\gamma\nabla f_\eta(x_k)^TH_k\bar w_{k,N_k}+{\gamma^2\over 2\eta}\|H_k\nabla f_\eta(x_k)\|^2}\\
&{+{\gamma^2\over 2\eta}\|H_k\bar w_{k,N_k}\|^2+{\gamma^2\over \eta}{H_k\nabla f_\eta(x_k)^T}H_k\bar w_{k,N_k}}\\
&{\leq f_{\eta}(x_k)-\gamma\nabla f_\eta(x_k)^TH_k\nabla f_\eta(x_k)+{\eta\over 4}\|\bar w_{k,N_k}\|^2+{\gamma^2\over \eta}\vvs{\|\nabla f_{\eta}(x_k)^TH_k\|}^2+{\gamma^2\over 2\eta}\|H_k\nabla f_\eta(x_k)\|^2}\\
&+{\olambda^2\gamma^2\over 2\eta}\|\bar w_{k,N_k}\|^2+{\gamma^2\over 2\eta}\|H_k\nabla f_\eta(x_k)\|^2+{\olambda^2\gamma^2\over 2\eta}\|\bar w_{k,N_k}\|^2,
\end{align*}
where  in the last inequality, we use the fact that \usr{$2a^Tb\leq {\eta\over 2\gamma}\|a\|^2+{2\gamma\over \eta}\|b\|^2$}. From Lemma \ref{bound sub}, $\mathbb E[\|\bar w_k\|^2\mid \mathcal F_k]\leq {4(\usr{\nu_1^2\|x_k\|^2+\nu_2^2})\over {\eta^2\tau^2}{N_k}}$. 
Now by taking conditional expectations with respect to $\mathcal F_k$, \usr{using Lemma \ref{H_k ns sc},}  we obtain the following.
\begin{align}\label{sc_nonsmooth_bound1}
& \quad \nonumber\mathbb E\left[f_{\eta}(x_{k+1})-f_{\eta}(x_k)\mid \mathcal F_k\right]\\
 & \leq -{\gamma}\nabla f_{\eta}(x_k)^TH_k\nabla f_{\eta}(x_k)+{2\gamma^2\over \eta}\|H_k\nabla f_{\eta}(x_k)\|^2
	 +{\left({{\olambda}^2\gamma^2\over \eta}+{\eta\over 4}\right){\usr{4(\nu_1^2\|x_k\|^2+\nu_2^2)}\over \usr{\tau^2 \eta^2} N_k}}\\ \nonumber&
={\gamma}\nabla f_{\eta}(x_k)^TH_k^{1/2}\left(-I+{2\gamma\over \eta}H_k^T\right)H_k^{1/2}\nabla f_{\eta}(x_k) +{\left({{\olambda}^2\gamma^2\over \eta}+{\eta\over 4}\right){\usr{4(\nu_1^2\|x_k\|^2+\nu_2^2)}\over \usr{\tau^2 \eta^2} N_k}}\\ \nonumber&
\leq -{\gamma} \left(1-{2\gamma\over \eta}\olambda\right)\|H_k^{1/2}\nabla f_{\eta}(x_k)\|^2+{\left({\olambda^2\gamma^2\over \eta}+{\eta\over 4}\right){\usr{8(\nu_1^2\|x_k-x^*\|^2)}\over \usr{\tau^2 \eta^2} N_k}}\\ \notag
	& +{\left({\olambda^2\gamma^2\over \eta}+{\eta\over 4}\right){\usr{(8\nu_1^2\|x^*\|^2+4\nu_2^2)}\over \usr{\tau^2 \eta^2} N_k}}\\ \notag
&= {-{\gamma}\over 2}\|H_k^{1/2}\nabla f_{\eta}(x_k)\|^2+{\left(\usr{5\eta\over 16}\right){\usr{8(\nu_1^2\|x_k-x^*\|^2)}\over \usr{\tau^2 \eta^2} N_k}}
	 +{\left(\usr{5\eta\over 16}\right){\usr{(8\nu_1^2\|x^*\|^2+4\nu_2^2)}\over \usr{\tau^2 \eta^2} N_k}} \\ \notag
&= {-{\gamma}\over 2}\|H_k^{1/2}\nabla f_{\eta}(x_k)\|^2+{\left(\usr{5\eta\over2 }\right){\usr{(\nu_1^2\|x_k-x^*\|^2)}\over \usr{\tau^2 \eta^2} N_k}}+{\left(\usr{5\eta\over 4}\right){\usr{(2\nu_1^2\|x^*\|^2+\nu_2^2)}\over \usr{\tau^2 \eta^2} N_k}}, 
\end{align} 
{where {the second equality follows from} $\gamma={\eta\over 4\olambda}$.} Since $f_{\eta}$ is $\tau/(1+\eta\tau)$-strongly convex (Lemma~\ref{feta}), $\|\nabla f_{\eta}(x_k)\|^2\geq 2\tau/(1+\eta\tau)
\left(f_{\eta}(x_k)-f_{\eta}(x^*)\right)$ \usr{ and $f_{\eta}(x_k)-f_{\eta}(x^*) \geq \tfrac{\tau}{(1+\eta \tau)}\|x_k-x^*\|^2$.}  {Consequently,} by subtracting $f_\eta(x^*)$ {from both sides}  
 by {invoking} Lemma \ref{H_k ns sc}, we obtain
\begin{align}\label{strong_nonsmooth_moreau}
\notag & \quad \mathbb E\left[f_{\eta}(x_{k+1})-f_\eta(x^*)\mid \mathcal F_k\right] \\
\notag& \leq f_{\eta}(x_{k})-f_\eta(x^*)-{\gamma\ulambda\over 2}\|\nabla f_{\eta}(x_k)\|^2+{{\usr{5(\nu_1^2\|x_k-x^*\|^2)}\over \usr{2\tau^2 \eta} N_k}}+{{\usr{5(2\nu_1^2\|x^*\|^2+\nu_2^2)}\over \usr{4\tau^2 \eta} N_k}}\\ &
\notag\leq\left(1-{\tau\over 1+\eta\tau}\gamma\ulambda+ \usr{\frac{(1+\eta \tau)}{\tau} {{\usr{(5\nu_1^2)}\over \usr{2\tau^2 \eta} N_k}}}\right)(f_{\eta}(x_{k})-f_\eta(x^*))+{{\usr{5(2\nu_1^2\|x^*\|^2+\nu_2^2)}\over \usr{4\tau^2 \eta} N_k}}\\
& \leq\left(1-{\tau\over 1+\eta\tau}\gamma\ulambda+ \usr{\frac{5\nu_1^2(1+\eta \tau)}{2\tau^3\eta N_0} }\right)(f_{\eta}(x_{k})-f_\eta(x^*))+{{\usr{5(2\nu_1^2\|x^*\|^2+\nu_2^2)}\over \usr{4\tau^2 \eta} N_k}}.
\end{align}
\usr{By observing the following relations,
\begin{align*} -{\tau\over 1+\eta\tau}\gamma\ulambda+ \usr{\frac{5\nu_1^2(1+\eta \tau)}{2\tau^3 \eta N_0} } 
\leq -\frac{1}{2}{\tau\over 1+\eta\tau}\gamma\ulambda & \iff \usr{\frac{5\nu_1^2(1+\eta \tau)}{2\tau^3 \eta N_0} } 
\leq \frac{1}{2}{\tau\over 1+\eta\tau}\gamma\ulambda\\
& \iff N_0 \geq \frac{5\nu_1^2(1+\eta \tau)^2}{\tau^4 \eta \gamma \ulambda},
\end{align*}
we may then conclude that if $N_0 \geq \frac{5(1+\eta \tau)^2}{\tau^4} \usr{(\nu_1^2)\over \eta \gamma \ulambda}$, we have that}  
\begin{align*}
  \mathbb E\left[f_{\eta}(x_{k+1})-f_\eta(x^*)\mid \mathcal F_k\right] 
& \leq\left(1-{\tau\over 2(1+\eta\tau)}\gamma\ulambda\right)(f_{\eta}(x_{k})-f_\eta(x^*))+{{\usr{4(2\nu_1^2\|x^*\|^2+\nu_2^2)}\over \usr{4\tau^2 \eta} N_k}}.
\end{align*}
Then by taking 	unconditional expectations, we obtain the following sequence of inequalities:
\begin{align}
 \notag \mathbb E\left[f_{\eta}(x_{k+1})-f_\eta(x^*)\right]
 & \leq  \left(1-{\tau\over \usr{2}(1+\eta\tau)}\gamma\ulambda\right)\mathbb E\left[f_{\eta}(x_{k})-f_\eta(x^*)\right]+{{\usr{5(2\nu_1^2\|x^*\|^2+\nu_2^2)}\over \usr{4\tau^2 \eta} N_k}}\\
\label{bound f_eta_moreau}
& ={\left(1-{(\tau\eta)^{n+2}(n-1)!\over \usr{8}(n+1)^{n+1}(1+\eta\tau)}\right)}\mathbb E\left[f_{\eta}(x_{k})-f_\eta(x^*)\right]+{{\usr{5(2\nu_1^2\|x^*\|^2+\nu_2^2)}\over \usr{4\tau^2 \eta} N_k}},
\end{align}
where the last equality arises from choosing $\ulambda={\eta\over 1+n}$,
$\olambda={1+n\over\tau \eta}$ {(by Lemma \ref{H_k ns sc} for
$m=1$)},  $\gamma={\eta\over 4\olambda}={\tau\eta^2\over 4(1+n)
}$ and using the fact that $N_k \geq N_0$ for all $k>0$. Let $d
\triangleq  {\left(1-{(\tau\eta)^{n+2}(n-1)!\over \usr{8}(n+1)^{n+1}(1+\eta\tau)}\right)}$ and $b_k
\triangleq  \usr{5(2\nu_1^2\|x^*\|^2+\nu_2^2)\over 4\tau^2 \eta {N_k}}$. Then {for
$\eta<{{1\over \tau}\left({8(n+1)^{n+1}\over (n-1)!}\right)^{1/(n+2)}}$, we have $d<1$}. Furthermore, by recalling that {$N_k=\lceil N_0q^{-k}\rceil$}, {it follows that $\usr{b_k \leq \tfrac{5(2\nu_1^2\|x^*\|^2+\nu_2^2)q^k}{4\tau^2 \eta N_0}},$ we obtain the following bound from \eqref{bound f_eta_moreau}.}
\begin{align*}
\mathbb E\left[f_{\eta}(x_{K+1})-f_\eta(x^*)\right]&
\leq d^{K+1}\mathbb E[f_\eta(x_0)-f_\eta(x^*)]+\sum_{i=0}^{K}d^{K-i}b_i \\
&
\leq d^{K+1}\mathbb E[f_\eta(x_0)-f_\eta(x^*)]+\usr{5(2\nu_1^2\|x^*\|^2+\nu_2^2)\over 4\tau^2\eta {N_0}}\sum_{i=0}^{K}d^{K-i}q^{i}.
\end{align*}
\usr{We now consider three cases.}

\noindent \usr{Case (1) $q < d$.}
If $q<d$, then $\sum_{i=0}^{K}d^{K-i}q^{i}=d^K \sum_{i=0}^K (q/d)^i\leq d^K\left({1\over 1-q/d}\right)$. {Since} $f_{\eta}$ retains the minimizers of $f$,  ${\tau\over 2(1+\eta\tau)}\|x_k-x^*\|^2\leq  f_{\eta}(x_k)-f_\eta(x^*)$)  by strong convexity of $f_\eta$, implying the following.
\begin{align*}
{\tau\over 2(1+{\eta}\tau)}\mathbb E[\|x_{K+1}-x^*\|^2]&
\leq d^{K+1}\mathbb E[f_\eta(x_0)-f_\eta(x^*)]+d^K\usr{5(2\nu_1^2\|x^*\|^2+\nu_2^2)\over 4\tau^2\eta {N_0}\left(1-q/d\right)}.
\end{align*}
Dividing both {sides} by ${\tau\over 2(1+\eta\tau)}$, the desired result is obtained. 
\begin{align*}
\mathbb E[\|x_{K+1}-x^*\|^2]
& \leq d^{K+1}\left({2\mathbb E[f_\eta(x_0)-f_\eta(x^*)](1+\eta\tau)\over \tau}\right)+d^K\usr{10(1+\eta \tau)(2\nu_1^2\|x^*\|^2+\nu_2^2)\over 4\tau^3\eta {N_0}\left(1-q/d\right)} = D d^{K+1}, \\ 
\mbox{ where } D & \triangleq \left({2\mathbb E[f_\eta(x_0)-f_\eta(x^*)](1+\eta\tau)\over \tau}\right) \usr{ \ + \ } \left(\usr{10(1+\eta \tau)(2\nu_1^2\|x^*\|^2+\nu_2^2)\over 4\tau^3\eta N_0(d-q)}\right). 
\end{align*}
\usr{Case (2) $q > d$.} {Similarly, if} $d<q$, 
$\mathbb E[\|x_{K+1}-x^*\|^2]\leq
D q^{K+1}$ 
where \usr{$$  D  \triangleq \left({2\mathbb E[f_\eta(x_0)-f_\eta(x^*)](1+\eta\tau)\over \tau}\right) \usr{ \ + \ } \left(\usr{10(1+\eta \tau)(2\nu_1^2\|x^*\|^2+\nu_2^2)\over 4\tau^3\eta N_0(q-d)}\right). $$}

\noindent \usr{Case (3) $q = d$. Then, we have that $$\sum_{i=0}^{K}d^{K-i}q^{i}=(K+1) d^K \leq\frac{1}{d} (K+1)d^{k+1} \leq \tfrac{1}{d} \tilde{D} \tilde{d}^{K+1}, $$
where $\tilde{d} \in (d,1)$ and $\tilde{D} > \tfrac{1}{\ln(\tilde{d}/d)^e}$. Consequently, we have that 
$\mathbb E[\|x_{K+1}-x^*\|^2]\leq
D \tilde{d}^{K+1}$ 
where \usr{$$  D  \triangleq \left({2\mathbb E[f_\eta(x_0)-f_\eta(x^*)](1+\eta\tau)\over \tau}\right) \usr{ \ + \ } \left(\usr{10(1+\eta \tau)(2\nu_1^2\|x^*\|^2+\nu_2^2)\tilde{D}\over 4\tau^3\eta N_0d }\right). $$}

}

\blue{(ii) To find an $x_{K+1}$ such that $\mathbb E[\|x_{K+1}-x^*\|^2]\leq \epsilon$, suppose $d<q$ with no loss of generality. Then for some $C>0$, ${Cq^K}\leq \epsilon$, implying that $K=\lceil{\log}_{1/q}(C/\epsilon)\rceil$. {It follows that} 
\begin{align*}
\sum_{k=0}^K N_k\leq\sum_{k=0}^{1+{\log}_{1/q}\left({C\over \epsilon}\right)} N_0q^{-k} = N_0 \frac{\left({\left({1\over q}\right) \left( {1\over q}\right)^{\log{1/q}\left(\tfrac{C}{\epsilon}\right)}-1}\right)} {\left({1/q-1}\right)} \leq N_0 \frac{\left(\tfrac{C}{\epsilon}\right)}{1-q}
={\mathcal O(1/\epsilon)}. \end{align*}}

(iii) {Similar to \eqref{sc_nonsmooth_bound1} and \eqref{strong_nonsmooth_moreau} we can obtain:

\begin{align*}
 \quad \nonumber\mathbb E\left[f_{\eta}(x_{k+1})-f_{\eta}(x_k)\mid \mathcal F_k\right]
 & \leq -{\gamma_k} \left(1-{2\gamma_k\over \eta}\olambda\right)\|H_k^{1/2}\nabla f_{\eta}(x_k)\|^2+{\left({\olambda^2\gamma_k^2\over \eta}+{\eta\over 4}\right){{8(\nu_1^2\|x_k-x^*\|^2)}\over {\tau^2 \eta^2} N_k}}\\ \notag
	& +{\left({\olambda^2\gamma_k^2\over \eta}+{\eta\over 4}\right){{(8\nu_1^2\|x^*\|^2+4\nu_2^2)}\over {\tau^2 \eta^2} N_k}}\end{align*}\begin{align*}
\implies E\left[f_{\eta}(x_{k+1})-f_{\eta}(x^*)\mid \mathcal F_k\right]
 &\leq \left(1-{\gamma_k} \left(1-{2\gamma_k\over \eta}\olambda\right) \tfrac{\ulambda\tau}{1+\eta\tau}+\tfrac{1+\tau\eta}{\tau}\left(\tfrac{\olambda^2\gamma_k^2}{\eta}+\tfrac{\eta}{4}\right)\tfrac{8\nu_1^2}{\eta^2\tau^2N_k}\right)(f_{\eta}(x_{k})-f_\eta(x^*))\\&+{\left({\olambda^2\gamma_k^2\over \eta}+{\eta\over 4}\right){{(8\nu_1^2\|x^*\|^2+4\nu_2^2)}\over {\tau^2 \eta^2} N_k}}.
\end{align*}}
Since $\gamma_k$ is a diminishing sequence and $N_k$ is an increasing sequence, for {sufficiently large} $K$ we have that $\tfrac{2\gamma_k^2\olambda\ulambda\tau}{\eta(1+\eta\tau)}+\tfrac{1+\tau\eta}{\tau}\left(\tfrac{\olambda^2\gamma_k^2}{\eta}+\tfrac{\eta}{4}\right)\tfrac{8\nu_1^2}{\eta^2\tau^2N_k}\leq{1\over 2}\left(\tfrac{\gamma_k\ulambda\tau}{1+\eta\tau}\right)$. Therefore, we obtain:

\begin{align*}
E\left[f_{\eta}(x_{k+1})-f_{\eta}(x^*)\mid \mathcal F_k\right]\leq \left(1-{c \gamma_k}\right)(f(x_k)-f(x^*))+\tfrac{d}{N_k}, \mbox{ for } k \geq K,
\end{align*}
where we use the fact that $\gamma_k\leq 1$ and we set $c=\tfrac{\ulambda\tau}{2(1+\eta\tau)}$ and $d={\left({\olambda^2\over \eta}+{\eta\over 4}\right){{(8\nu_1^2\|x^*\|^2+4\nu_2^2)}\over {\tau^2 \eta^2} }}$. Since $\gamma_k = \tfrac{1}{k^{s}}$ and $N_k = \lceil k^{p+s}\rceil$ where $0 < s < 1$ and $p > 0$, by taking unconditional expecations, 
\begin{align*}
\notag & \quad \mathbb E\left[f_{\eta}(x_{k})-f_\eta(x^*)\right] \leq \frac{d}{ck^p} + o\left(\tfrac{1}{k^p}\right), \mbox{ for } k \geq K.
\end{align*}
By leveraging the $\tfrac{\tau}{1+\eta \tau}$-strong convexity of $f_{\eta}$, we may claim that 
\begin{align*}
\notag & \quad \mathbb E\left[\|x_k-x^*\|^2 \right] \leq \frac{(1+\eta \tau)d}{c\tau k^p} + o\left(\frac{1}{k^p}\right), \mbox{ for } k \geq K.
\end{align*}
\end{proof}

\section{Smooth and nonsmooth convex optimization}\label{sec:4}
{In this section, we weaken the {strong convexity} requirement and analyze the rate
and oracle complexity of \eqref{rVS-SQN} and \eqref{rsVS-SQN} {in smooth
and nonsmooth regimes}, respectively.}
\subsection{Smooth convex optimization}
{Consider the setting when $f$ is an $L$-smooth convex function. In such an instance, a regularization of $f$ and its gradient can be defined as follows.}
\begin{definition}[{\bf Regularized function and gradient map}]\label{def:regularizedf}
Given a sequence $\{\mu_k\}$ of positive scalars, 
the function $f_{\mu_k}$ and its gradient $\nabla f_k(x)$ are defined as follows for {any $x_0\in \mathbb R^n$}:
\begin{align*}
f_{\mu_k}(x)&\triangleq  f(x)+\frac{\mu_k}{2}{\|x-x_0\|^2},\quad  \hbox{for any } k \geq 0, \qquad 
\nabla f_{\mu_k}(x)\triangleq\nabla f(x)+\mu_k(x-x_0),\quad  \hbox{for any } k \geq 0.
\end{align*}
Then{,} $f_{\mu_k}$ and ${\nabla} f_{\mu_k}$ satisfy the following:
(i) $f_{\mu_k}$ is $\mu_k$-strongly convex;
(ii) $f_{\mu_k}$ has Lipschitzian gradients with parameter $L+\mu_k$;
 (iii) $f_{\mu_k}$ has a unique minimizer over $\mathbb R^n$, denoted by $x^*_k$. Moreover, for any $x \in \mathbb R^n$~\cite[sec. 1.3.2]{polyak1987introduction}, 
\begin{align*} 2\mu_k (f_{\mu_k}(x)-f_{\mu_k}(x^*_k))& \leq \|\nabla f_{\mu_k}(x)\|^2\leq 2(L+\mu_k) \left(f_{\mu_k}(x)-f_{\mu_k}(x^*_k)\right).\end{align*}
\end{definition}
We consider the following update rule \eqref{rVS-SQN}, { where $H_k$ is generated by {\bf rL-BFGS} scheme.}
\begin{align}\tag{\bf rVS-SQN}\label{rVS-SQN}
x_{k+1}:=x_k-\gamma_kH_k{\frac{\sum_{j=1}^{N_k} \nabla_x F_{\mu_k}(x_k,\omega_{j,k})}{N_k}}.
\end{align}
 {For a subset of the results, we assume quadratic growth property. }
\begin{assumption}{\bf(Quadratic growth)}\label{growth}
Suppose that the function $f$ has a nonempty set $X^*$ of minimizers. There exists $\alpha>0$ such that $f(x)\geq f(x^*)+{\alpha\over 2}\mbox{dist}^2(x,X^*)$ holds for all $x\in \mathbb R^n$:
\end{assumption}}
In the next lemma the bound for eigenvalues of $H_k$ is derived (see Lemma 6 in \cite{yousefian2017stochastic}). 
\begin{lemma}[{\bf Properties of Hessian approximations produced by (rL-BFGS)}]\label{rLBFGS-matrix}
Consider the \eqref{rVS-SQN} method. Let $H_k$ be given by the update rule
\eqref{eqn:H-k}-\eqref{eqn:H-k-m} \us{with $\eta_k = 0$ for all $k$,} and
$s_i$ and $y_i$ are defined in \eqref{equ:siyi-LBFGS}. {Suppose $\mu_k$ is} updated according to the procedure \eqref{eqn:mu-k}. Let
Assumption.~\ref{assum:convex-smooth}(a,b) hold. Then the following hold.
\begin{itemize}
\item [(a)]  For any odd $k > 2m$,  $s_k^T{y_k} >0$;
(b)  For any odd $k > 2m$, $H_{k}{y}_k=s_k$; 
\item [(c)] For any $k > 2m$, $H_k$ satisfies Assumption \ref{assump:Hk}{(S)}, ${\ulambda}={\frac{1}{(m+n)(L+\mu_0^{\bar \delta})}}$, $\lambda = {\frac{(m+n)^{n+m-1}{(L+\mu_0^{\bar \delta})}^{n+m-1}}{(n-1)!}}$ and 
${\olambda_k}=  \lambda \mu_k^{-\bar \delta(n+m)},$ {for scalars $\delta,\bar \delta>0$.} 
Then {for all $k$, we have that $H_k = H_k^T$ and $\mathbb E[{H_k\mid\mathcal F_k}]=H_k$ and
$
{\ulambda\mathbf{I}  \preceq H_{k} \preceq \olambda_k \mathbf{I}}$ both hold in an a.s. fashion.}
\end{itemize}
\end{lemma}
\begin{lemma}[An error bound]\label{lemma:main-ineq}
Consider the \eqref{VS-SQN} method and suppose Assumptions \ref{assum:convex-smooth}, \ref{state noise}(S-M), \ref{state noise}(S-B), \ref{assump:Hk}(S) \red{ and \ref{growth}} hold. {Suppose} $\{\mu_k\}$ is a non-increasing sequence, and $\gamma_k$ satisfies
\begin{align}\label{mainLemmaCond}\gamma_k \leq \frac{{\ulambda}}{{{\olambda}_k ^2}(L+\mu_0)},\quad \hbox{for all }k\geq 0.
\end{align}Then, the following inequality holds for all $k$:
\begin{align}\label{ineq:cond-recursive-F-k}
\nonumber\mathbb E[{f_{\mu_{k+1}}(x_{k+1})\mid\mathcal F_k}]-f^* &\leq (1-{{\ulambda}}\mu_k\gamma_k)(f_{\mu_k}(x_k)-f^*)  +\frac{{\ulambda}\mbox{dist}^2(x_0,X^*)}{2}\mu_k^2\gamma_k\\&+\frac{ (L+\mu_k){\olambda_k ^2}{( \nu_1^2\|x_k\|^2+\nu_2^2)}}{2N_k}\gamma_k^2.
\end{align}
\end{lemma}
\begin{proof}
By the Lipschitzian property of $\nabla f_{\mu_k}${,
 update rule \eqref{rVS-SQN} and Def.~\ref{def:regularizedf}}, we obtain  
{\begin{align}\label{ineq:term1-2}
 & \quad f_k(x_{{k+1}}) \nonumber\leq f_{\mu_k}(x_k)+\nabla f_{\mu_k}(x_k)^T(x_{k+1}-x_k)+\frac{ (L+\mu_k)}{2}\|x_{k+1}-x_k\|^2 
\\&\leq  f_{\mu_k}(x_k)-\gamma_k\underbrace{\nabla f_{\mu_k}(x_k)^TH_k(\nabla f_{\mu_k}(x_k)+\bar w_{k,N_k})}_{\tiny\hbox{Term } 1}+ \frac{ (L+\mu_k)}{2}\gamma_k^2\underbrace{\|H_k(\nabla f_{\mu_k}(x_k)+\bar w_{k,N_k})\|^2}_{\tiny\hbox{ Term } 2},
\end{align}}
 {where} $\bar{w}_{k,N_k} \triangleq \frac{\sum_{j=1}^{N_k} \left({\nabla_{x}}
{F}_{\mu_k}(x_k,\omega(\omega_{j,k}))-\nabla f_{\mu_k}(x_k)\right)}{N_k}$. Next, we estimate the conditional expectation of Terms 1 and 2. From Assumption \ref{assump:Hk}, we have 
\begin{align*}
\hbox{Term }1 &= \nabla f_{\mu_k}(x_k)^TH_k\nabla f_{\mu_k}(x_k)+\nabla f_{\mu_k}(x_k)^TH_k\bar w_{k,N_k}\geq {\ulambda}\|\nabla f_{\mu_k}(x_k)\|^2+\nabla f_{\mu_k}(x_k)^TH_k\bar w_{k,N_k}.
\end{align*}
{Thus, taking conditional expectations, from \eqref{ineq:term1-2},} we obtain
\begin{align}\label{equ:Term1}
\notag\mathbb E[{\hbox{Term } 1\mid\mathcal F_k}] &\notag\geq {\ulambda}\|\nabla f_{\mu_k}(x_k)\|^2+\mathbb E[{\nabla f_{\mu_k}(x_k)^TH_k\bar w_{k,N_k}\mid\mathcal F_k}]\\ &={\ulambda}\|\nabla f_{\mu_k}(x_k)\|^2+\nabla f_{\mu_k}(x_k)^TH_k\mathbb E[{\bar w_{k,N_k}\mid\mathcal F_k}] ={\ulambda}\|\nabla f_{\mu_k}(x_k)\|^2,\end{align}
where $\mathbb E[{\bar w_{k,N_k}\mid\mathcal F_k}]=0$ and $\mathbb E[{H_k\mid\mathcal F_k}]=H_k$ {a.s.} Similarly,  invoking Assumption~\ref{assump:Hk}(S), we may bound Term 2. 
\begin{align*}
\hbox{Term } 2&= (\nabla f_{\mu_k}(x_k)+\bar w_{k,N_k})^TH_k^2(\nabla f_{\mu_k}(x_k)+\bar w_{k,N_k}) \leq {{\olambda_k} ^2}\|\nabla f_{\mu_k}(x_k)+\bar w_{k,N_k}\|^2 \\&={{\olambda_k} ^2}\left(\|\nabla f_{\mu_k}(x_k)\|^2+\|\bar w_{k,N_k}\|^2+2\nabla f_{\mu_k}(x_k)^T\bar w_{k,N_k}\right).\end{align*}
Taking conditional expectations in the preceding inequality and using Assumption \ref{state noise} \vvs{(S-M), \ref{state noise} (S-B)}, we obtain
\begin{align}\label{equ:Term2}
\mathbb E[{\hbox{Term } 2\mid\mathcal F_k}]\notag&\leq{\olambda_k^2}\Big(\|\nabla f_{\mu_k}(x_k)\|^2+\mathbb E[{\|\bar w_{k,N_k}\|^2\mid\mathcal F_k}]\\&+2\nabla f_{\mu_k}(x_k)^T\mathbb E[{\bar w_{k,N_k}\mid\mathcal F_k}]\Big) \leq {\olambda^2_k}\left(\|\nabla f_{\mu_k}(x_k)\|^2+{{\nu_1^2\|x_k\|^2+\nu_2^2}\over N_k}\right).\end{align}
By taking conditional expectations in \eqref{ineq:term1-2}, and by \eqref{equ:Term1}--\eqref{equ:Term2}, 
\begin{align*}
\quad \mathbb E[{f_{\mu_k}(x_{k+1})\mid\mathcal F_k}] &\leq f_{\mu_k}(x_k)-\gamma_k{\ulambda}\|\nabla {\mu_k}(x_k)\|^2+{{\olambda}_k ^2}\frac{ (L+\mu_k)}{2}\gamma_k^2\left(\|\nabla f_{\mu_k}(x_k)\|^2+{{\nu_1^2\|x_k\|^2+\nu_2^2}\over N_k}\right) \\
&\leq f_{\mu_k}(x_k)-\frac{\gamma_k\ulambda}{2}\|\nabla f_{\mu_k}(x_k)\|^2\left(2-\frac{{\olambda_k ^2}\gamma_k(L+\mu_k)}{{\ulambda}}\right)+{\olambda_k ^2}\frac{ (L+\mu_k)}{2}{\gamma_k^2{( \nu_1^2\|x_k\|^2+\nu_2^2)}\over N_k}.
\end{align*}
\us{From} \eqref{mainLemmaCond}, $\gamma_k\leq \frac{{\ulambda}}{{\olambda_k ^2}(L+\mu_0)}$ for any $k \geq 0$. 
Since $\{\mu_k\}$ is {a} non-increasing sequence, it follows that
\begin{align*}
\gamma_k \leq \frac{{\ulambda}}{{\olambda_k ^2}(L+\mu_k)} 
 \implies 2-\frac{{\olambda_k ^2}\gamma_k(L+\mu_k)}{{\ulambda}}  \geq 1.\end{align*}
  Hence, the following holds.
\begin{align*}
 \mathbb E[{f_{\mu_k}(x_{k+1}) \mid\mathcal F_k}]&\leq f_{\mu_k}(x_k)-\frac{\gamma_k{\ulambda}}{2}\|\nabla f_{\mu_k}(x_k)\|^2+{\olambda_k ^2}\frac{ (L+\mu_k)}{2}{\gamma_k^2{( \nu_1^2\|x_k\|^2+\nu_2^2)}\over N_k}\\
	&\hspace{-0.2in} \overset{\tiny \mbox{(iii) in Def.~\ref{def:regularizedf}}}{\leq} 
 f_{\mu_k}(x_k)-{\ulambda}\mu_k\gamma_k(f_{\mu_k}(x_k)-f_{\mu_k}(x^*_k))+{\olambda_k ^2}\frac{ (L+\mu_k)}{2}{\gamma_k^2{( \nu_1^2\|x_k\|^2+\nu_2^2)}\over N_k}.
\end{align*}
By using Definition \ref{def:regularizedf} and {non-increasing  {property} of } $\{\mu_k\}$,
\begin{align} \label{ineq:lemmaLastIneq}
\notag
&\mathbb E[{f_{\mu_{k+1}}(x_{k+1})\mid \mathcal F_k}]  \leq\mathbb E[{f_{\mu_k}(x_{k+1})\mid \mathcal F_k}]\implies\\&
 \mathbb E[{f_{\mu_{k+1}}(x_{k+1})\mid\mathcal F_k}]\leq  f_{\mu_k}(x_k)-{\ulambda}\mu_k\gamma_k(\overbrace{f_{\mu_k}(x_k)-f_{\mu_k}(x^*_k)}^{\tiny{\mbox{Term 3}}})+{\olambda^2 _k}\frac{ (L+\mu_k)}{2}{\gamma_k^2{( \nu_1^2\|x_k\|^2+\nu_2^2)}\over N_k}.
\end{align}
Next, we derive a lower bound for Term 3. Since $x_k^*$ is the unique minimizer of $f_{\mu_k}$, we have $f_{\mu_k}(x_k^*) \leq f_{\mu_k}(x^*)$. Therefore, invoking Definition \ref{def:regularizedf}, for an arbitrary optimal solution $x^* \in X^*$, 
\begin{align*}
f_{\mu_k}(x_k)-f_{\mu_k}(x^*_k) & \geq f_{\mu_k}(x_k)-f_{\mu_k}(x^*) =f_{\mu_k}(x_k)-f^*-\frac{\mu_k}{2}\|x^*-x_0\|^2.\end{align*}
From the preceding relation and \eqref{ineq:lemmaLastIneq}, we have
\begin{align*}
\mathbb E[{f_{\mu_{k+1}}(x_{k+1})\mid\mathcal F_k}]& \leq f_{\mu_k}(x_k)-\ulambda\mu_k\gamma_k\mathbb E[f_{\mu_k}(x_k)-f^*]+\frac{\ulambda\|x^*-x_0\|^2\mu_k^2\gamma_k}{2}
\\&+\frac{(L+\mu_k){\olambda^2 _k}{( \nu_1^2\|x_k\|^2+\nu_2^2)}\gamma_k^2}{2N_k}.
\end{align*}
By subtracting $f^*$ from both sides {and {by noting that 
this inequality holds for all $x^* \in X^*$ where $X^*$ denotes the solution set},
the desired result is obtained.}
\end{proof}

{We now {derive the rate for} sequences produced by \eqref{rVS-SQN} {under the following assumption.}}
\begin{assumption}
\label{assum:sequences-ms-convergence}
Let the positive sequences {$\{N_k,\gamma_k,\mu_k,t
_k\}$} satisfy the following conditions:
{\begin{itemize}
\item [(a)] $\{\mu_k\}, \{\gamma_k\}$ are non-increasing sequences such that  $\mu_k,\gamma_k \to 0$; $\{t_k\}$ is {an} increasing sequence;
\item [(b)] $\left(1-{\ulambda\mu_k\gamma_k}{+{2(L+\mu_0)\olambda_k^2\nu_1^2\gamma_k^2\over N_k\alpha}}\right)t_{k+1}\leq t_k, \ \forall k\geq \tilde K$ for some $\tilde K\geq 1$;
\item [(c)] $\sum_{k=0}^{\infty}{\mu_k^2\gamma_k}{t_{k+1}}={\bar c_0}<\infty$;
(d) $\sum_{k=0}^\infty {\mu_k^{-{2}\bar\delta(n+m)}\gamma_k^2\over N_k}{t_{k+1}}={\bar c_1}<\infty$;
\end{itemize}}
\end{assumption}

\begin{theorem}[{\bf Convergence of \eqref{rVS-SQN} in mean}]\label{thm:mean}
Consider the \eqref{rVS-SQN} scheme and  suppose Assumptions ~\ref{assum:convex-smooth}, \ref{state noise}(S-M), \ref{state noise}(S-B),~\ref{assump:Hk}(S),~\ref{growth} and~\ref{assum:sequences-ms-convergence} hold. 
{There exists} $\tilde K\geq 1$ and scalars $\bar c_0, \bar c_1$ (defined in Assumption ~\ref{assum:sequences-ms-convergence}) such that the
following inequality holds for all {$K\geq \tilde
K+1$}:\begin{align}\label{ineq:bound}
\mathbb E[{f(x_{K})-f^*}]  
 \leq {{t_{\tilde K}}\over t_K}\mathbb E[{{f_{\mu_{\tilde K}}(x_{\tilde K})}-f^*}] +{\bar c_0+\bar c_1\over t_K}.
 \end{align}
\end{theorem}
\begin{proof}
We begin by noting that Assumption~\ref{assum:sequences-ms-convergence}(a,b) implies that \eqref{ineq:cond-recursive-F-k} holds for  $k \geq \tilde K$, where $\tilde K$ is defined in Assumption~\ref{assum:sequences-ms-convergence}(b). Since the conditions of Lemma \ref{lemma:main-ineq} are met, taking expectations on both sides of \eqref{ineq:cond-recursive-F-k}:
\begin{align*}
\mathbb E[{f_{\mu_{k+1}}(x_{k+1})-f^*}] \notag 
& \leq \left(1-\ulambda{\mu_k\gamma_k}\right)\mathbb E[{f_{\mu_k}(x_k)-f^*}]  +\frac{\ulambda\mbox{dist}^2(x_0,X^*)}{2}\mu_k^2\gamma_k \\&+\frac{ (L+\mu_0){\olambda^2 _k}{( \nu_1^2\|x_k-x^*+x^*\|^2+\nu_2^2)}}{2N_k}\gamma_k^2 \quad \forall k\geq \tilde K.
\end{align*}
{Now by using the quadratic growth property i.e. $\|x_k-x^*\|^2\leq {2\over \alpha}\left(f(x)-f(x^*)\right)$ and the fact that $\|x_k-x^*+x^*\|^2\leq 2\|x_k-x^*\|^2+2\|x^*\|^2$, we obtain the following relationship}
\begin{align*}
\mathbb E[{f_{\mu_{k+1}}(x_{k+1})-f^*}] \notag 
& \leq \left(1-\ulambda\mu_k\gamma_k{+{2(L+\mu_0)\olambda_k^2\nu_1^2\gamma_k^2\over N_k\alpha}}\right)\mathbb E[{f_{\mu_k}(x_k)-f^*}]  +\frac{\ulambda\mbox{dist}^2(x_0,X^*)}{2}\mu_k^2\gamma_k \\&+\frac{ (L+\mu_0){\olambda^2 _k}( {2\nu_1^2\|x^*\|^2+\nu_2^2})}{2N_k}\gamma_k^2.
\end{align*}
By multiplying both sides by $t_{k+1}$, using Assumption~\ref{assum:sequences-ms-convergence}(b) and $\olambda_k=\lambda \mu_k^{-\bar\delta(n+m)}$, we obtain
\begin{align}\label{ineq:cond-recursive-F-k-expected2}
& \quad t_{k+1}\mathbb E[{f_{\mu_{k+1}}(x_{k+1})-f^*}]  
 \leq t_k\mathbb E[{f_{\mu_k}(x_k)-f^*}]  +A_1\mu_k^2\gamma_kt_{k+1} +\frac{ A_2{ \mu_k^{-2\bar\delta(n+m)}}}{N_k}\gamma_k^2t_{k+1},
\end{align}
where $A_1\triangleq \tfrac{ \underline{\lambda} \mbox{\scriptsize dist}^2(x_0,X^*)}{2}$ and $A_2\triangleq\frac{ (L+\mu_0){\lambda^2 }( {2\nu_1^2\|x^*\|^2+\nu_2^2})}{2}$.  By summing \eqref{ineq:cond-recursive-F-k-expected2} from {$k=\tilde K$} to $K-1$, for {$K\geq \tilde K+1$}, and dividing both sides by $t_K$, we obtain
\begin{align*}
&\nonumber \mathbb E[{f_{\mu_K}(x_{K})-f^*}]  
 \leq {{t_{\tilde K}}\over t_K}\mathbb E[{{f_{\mu_{\tilde K}}(x_{\tilde K})}-f^*}]  +{\sum_{k={\tilde K}}^{K-1}A_1\mu_k^2\gamma_kt_{k+1}\over t_K} +{\sum_{k={\tilde K}}^{K-1}A_2\mu_k^{-2\bar\delta(n+m)}\gamma_k^2t_{k+1}N_k^{-1}\over t_K}. 
\end{align*}
From Assumption \ref{assum:sequences-ms-convergence}(c,d),  $\sum_{k={\tilde K}}^{K-1}\left( A_1 \mu_k^2\gamma_kt_{k+1}+ A_2 \mu_k^{-2\bar\delta(n+m)}\gamma_k^2{t_{k+1}\over N_k}\right)\leq {A_1\bar c_0+A_2\bar c_1}$. Therefore, \af{by using the fact that $f(x_K)\leq f_{\mu_K}(x_K)$, we obtain} \\ $ \mathbb E[{f(x_{K})-f^*}]  
 \leq {{t_{\tilde K}}\over t_K}\mathbb E[{{f_{\mu_{\tilde K}}(x_{\tilde K})}-f^*}] +{{\bar c_0+\bar c_1}\over t_K}.$
\end{proof}

We now show that the requirements of Assumption~\ref{assum:sequences-ms-convergence} are satisfied under suitable assumptions.
\begin{corollary}
 Let  $N_k\triangleq\lceil N_0 k^a\rceil$, $\gamma_k\triangleq\gamma_0k^{-b}$, $\mu_k\triangleq\mu_0k^{-c}$ and {$t_k\triangleq t_0(k-1)^{h}$} for some
$a,b,c,h>0$. {Let $2\bar \delta (m+n)=\varepsilon$ for 
$\varepsilon>0$}. {Then Assumption~\ref{assum:sequences-ms-convergence} holds if} {${a+2b-c\varepsilon\geq b+c}, \ {N_0 {\geq}{(L+\mu_0)\lambda^2\nu_1^2\gamma_0\over \alpha\ulambda\mu_0}},\ b+c<1$, $h\leq 1$, $b+2c-h>1$ and $a+2b-h-c\varepsilon>1$}. 
\end{corollary}
\begin{proof}
{From} $N_k=\lceil N_0 k^a\rceil\geq N_0 k^a$, $\gamma_k=\gamma_0k^{-b}$ and $\mu_k=\mu_0k^{-c}$, {the} requirements to satisfy Assumption \ref{assum:sequences-ms-convergence} are as follows:
\begin{itemize}
\item [(a)] $\lim_{k \to \infty }{\gamma_0}k^{-b}=0, \lim_{k \to \infty }{\mu_0}k^{-c}=0  \Leftrightarrow b ,c>0$ ;
\item [(b)] $\left(1-{\ulambda\mu_k\gamma_k}{+{2(L+\mu_0)\olambda_k^2\nu_1^2\gamma_k^2\over N_k\alpha}}\right)\leq {t_k\over t_{k+1}} \Leftrightarrow \left(1-{1\over k^{b+c}}+{1\over k^{a+2b-c\varepsilon}}\right)\leq (1-1/k)^h$. From the Taylor expansion of right hand side and assuming $h\leq 1$, we get $\left(1-{1\over k^{b+c}}+{1\over k^{a+2b-c\varepsilon}}\right)\leq 1-M/k$ for some $M>0$ and $\forall k\geq \tilde K$ which means  $\left(1-{\ulambda\mu_k\gamma_k}{+{2(L+\mu_0)\olambda_k^2\nu_1^2\gamma_k^2\over N_k\alpha}}\right)\leq {t_k\over t_{k+1}} \Leftrightarrow h\leq1, \ b+c<1, \ {a+2b-c\varepsilon\geq b+c} \ \mbox{and}\ {N_0={(L+\mu_0)\lambda^2\nu_1^2\gamma_0\over \alpha\ulambda\mu_0}}$;
\item [(c)] $\sum_{k=0}^{\infty}{\mu_k^2\gamma_k}{t_{k+1}}<\infty\Leftarrow \sum_{k=0}^\infty {1\over k^{b+2c-h}}<\infty\Leftrightarrow b+2c-h>1$;
\item [(d)]  $\sum_{k=0}^\infty {\mu_k^{-{2}\bar\delta(n+m)}\gamma_k^2\over N_k}{t_{k+1}}<\infty\Leftarrow \sum_{k=0}^\infty {1\over k^{a+2b-h-c\varepsilon}}<\infty\Leftrightarrow a+2b-h-c\varepsilon>1$;
\end{itemize}
\end{proof}

One can easily verify that $a=2+\varepsilon$, $b=\varepsilon$ and {$c=1-{2\over
3}\varepsilon$} and $h=1-\varepsilon$ satisfy these
conditions. {We derive complexity statements for \eqref{rVS-SQN} for a specific choice of parameter sequences.} 

\begin{theorem}[{\bf Rate statement and Oracle complexity}]\label{oracle smooth}
Consider the \eqref{rVS-SQN} scheme and  suppose Assumptions  ~\ref{assum:convex-smooth}, \ref{state noise}(S-M), \ref{state noise}(S-B), \ref{assump:Hk}(S), { \ref{growth}}
 and \ref{assum:sequences-ms-convergence} hold.  Suppose $\gamma_k\triangleq{\gamma_0k^{-b}}$, $\mu_k\triangleq{\mu_0k^{-c}}$, $\triangleq t_k=t_0(k-1)^h$ and $N_k\triangleq\lceil N_0k^{a}\rceil$ where {$\red{N_0={(L+\mu_0)\lambda^2\nu_1^2\gamma_0\over \alpha\ulambda\mu_0}}$, $a=2+\varepsilon$,
$b=\varepsilon$ and {$c=1-{2\over 3}\varepsilon$}} and $h=1-\varepsilon$. 

\noindent (i)  {Then the following  holds for $K \geq
\tilde{K}$ where $\tilde K\geq 1$ and $\tilde C \triangleq {{f_{\mu_{\tilde
K}}(x_{\tilde K})}-f^*} $.} \begin{align}\label{rate K}
& \mathbb E[{f(x_{K})-f^*}]  
 \leq {\tilde C+\bar c_0+\bar c_1\over K^{1-\varepsilon}}.
\end{align}
(ii) Let ${\epsilon>0}$ and {$ K\geq \tilde K+1$} such that $\mathbb E[f(x_{ K})]-f^*\leq {\epsilon}$. Then{,} {$\sum_{k=0}^{ K}N_k\leq  {\mathcal O\left({ {\epsilon}^{-{3+\varepsilon\over 1-\varepsilon}}}\right)}$}.
\end{theorem}
\begin{proof}
(i) {By choosing the sequence parameters as specified, the result follows immediately from Theorem \ref{thm:mean}.}
\noindent (ii) To find an $x_{ K}$ such that $\mathbb E[f(x_{ K})]-f^*\leq {\epsilon}$ we have ${\tilde C+\bar c_0+\bar c_1\over \tilde K^{1-\varepsilon}}\leq {\epsilon}$ which implies that $ K=\lceil {\left(\tilde C+\bar c_0+\bar c_1\over {\epsilon}\right)^{1\over1-{\varepsilon}}}\rceil$. Hence, the following holds. 
\begin{align*}
& \sum_{k=0}^{ K} N_k\leq \sum_{k=0}^{1+{(C/{\epsilon})^{1\over 1-\varepsilon}}}2N_0 k^{2+\varepsilon} \leq 2N_0\int_0^{{1+{(C/{\epsilon})}^{1\over 1-\varepsilon}}} {x}^{2+\varepsilon} \ d{x}=\frac{2N_0\left({1+\left(C/{\epsilon}\right)^{1\over 1-\varepsilon}}\right)^{3+\varepsilon}}{3+\varepsilon}
  \leq \mathcal O\left({{\epsilon}^{-{3+\varepsilon\over 1-\varepsilon}}}\right).
\end{align*}
\end{proof}

{One may instead consider the following requirement on the conditional second moment on the sampled gradient instead of state-dependent noise (Assumption \ref{state noise}). 

\begin{assumption}\label{assum_error}
\vvs{Let $\bar{w}_{k,N_k} \triangleq \nabla_x f(x_k) - \tfrac{\sum_{j=1}^{N_k} \nabla_x F(x_k,\omega_{j,k})}{N_k}$. Then 
there exists $\nu>0$ such that $\mathbb{E}[\|\bar{w}_{k,N_k}\|^2\mid \mathcal{F}_k] \leq \tfrac{\nu^2}{N_k}$ and $\mathbb{E}[\bar{w}_{k,N_k} \mid \mathcal{F}_k] = 0$ {hold} almost surely 
 for all $k$, where $\mathcal{F}_k
	\triangleq \sigma\{x_0, x_1, \hdots, x_{k-1}\}$.}
	\end{assumption}
By invoking Assumption \ref{assum_error}, we can derive the rate result without requiring a quadratic growth property of objective function. 

\begin{corollary}
[{\bf Rate statement and Oracle complexity}]
Consider  \eqref{rVS-SQN}  and  suppose Assumptions  ~\ref{assum:convex-smooth}, \ref{assump:Hk}(S),
  \ref{assum:sequences-ms-convergence} and \ref{assum_error} hold.  Suppose $\gamma_k={\gamma_0k^{-b}}$, $\mu_k={\mu_0k^{-c}}$, $t_k=t_0(k-1)^h$ and $N_k=\lceil k^{a}\rceil$ where $a=2+\varepsilon$,
$b=\varepsilon$ and $c=1-{4\over 3}\varepsilon$ and $h=1-\varepsilon$. 

\noindent (i)  {Then for $K \geq
\tilde{K}$ where $\tilde K\geq 1$ and $\tilde C \triangleq {{f_{\mu_{\tilde
K}}(x_{\tilde K})}-f^*} $,}
$
 \mathbb E[{f(x_{K})-f^*}]  
 \leq {\tilde C+\bar c_0+\bar c_1\over K^{1-\varepsilon}}.
$
(ii) Let ${\epsilon>0}$ and {$ K\geq \tilde K+1$} such that $\mathbb E[f(x_{ K})]-f^*\leq {\epsilon}$. Then, {$\sum_{k=0}^{ K}N_k\leq  {\mathcal O\left({ {\epsilon}^{-{3+\varepsilon\over 1-\varepsilon}}}\right)}$}.
\end{corollary}
\blue{\begin{remark}
Although the oracle complexity of (\ref{rVS-SQN}) is poorer than the canonical $\mathcal{O}(1/\epsilon^2)$, 
there are several reasons to consider
using the SQN schemes when faced with a choice between gradient-based
counterparts. (a) Sparsity. In many machine learning problems, the sparsity
properties of the estimator are of relevance. However, averaging schemes tend
to have a detrimental impact on the sparsity properties while non-averaging
schemes do a far better job in preserving such properties. Both accelerated and
unaccelerated gradient schemes for smooth stochastic convex optimization rely
on averaging and this significantly impacts the sparsity of the estimators.
(See Table \ref{compare_spars} in Section \ref{sec:5}). (b) Ill-conditioning.
As is relatively well known, quasi-Newton schemes do a far better job of
contending with ill-conditioning in practice, in comparison with gradient-based
techniques. (See Tables \ref{quad_ill} and \ref{convex_ill} in Section
\ref{sec:5}.) \end{remark}}
}

\subsection{Nonsmooth convex optimization}
We now consider problem~\eqref{main problem} when $f$ is nonsmooth but $(\alpha,\beta)$-smoothable and consider the \eqref{rsVS-SQN} scheme,
 defined as follows, where $H_k$ is generated by {\bf rsL-BFGS} scheme. 
\begin{align}\tag{\bf rsVS-SQN}\label{rsVS-SQN}
x_{k+1}:=x_k-\gamma_kH_k{\frac{\sum_{j=1}^{N_k} \nabla_x F_{\eta_k,\mu_k}(x_k,\omega_{j,k})}{N_k}}.
\end{align}
{Note that in this section, we set $m=1$ for the sake of simplicity but the analysis can be extended to $m>1$. 
Next, we generalize Lemma \ref{rLBFGS-matrix} to {show that Assumption \ref{assump:Hk} is
satisfied and both the secant condition ({{\bf SC}}) and the secant equation
({{\bf SE}}). ({See Appendix for Proof.})} 
\begin{lemma}[{\bf Properties of Hessian approximation produced by
(rsL-BFGS)}]\label{rsLBFGS-matrix} Cons-\\ ider the \eqref{rsVS-SQN} method, {where} $H_k$ {is updated} by 
\eqref{eqn:H-k}-\eqref{eqn:H-k-m}, $s_i$ and $y_i$ are defined in
\eqref{equ:siyi-LBFGS} {and} $\eta_k$ and $\mu_k$ are updated according to procedure
\eqref{eqn:mu-k}. Let Assumption \ref{assum:convex2} holds. Then the
following hold.  \begin{itemize}
\item [(a)]  For any odd $k > 2m$, {(SC) holds},  i.e., $s_k^T{y_k} >0$;
\item [(b)]  For any odd $k > 2m$, {(SE) holds}, i.e., $H_{k}{y}_k=s_k$. 
\item [(c)] For any $k > 2m$, $H_k$ satisfies Assumption~\ref{assump:Hk}{(NS)} with 
${{\ulambda_{k}}={1\over (m+n)(1/\eta_k^\delta+\mu_0^{\bar \delta})}}$ and 
$\\ {{\olambda_{k}}={(m+n)^{n+m-1}(1/\eta_k^\delta+\mu_0^{\bar \delta})^{n+m-1}\over (n-1)!\mu_k^{(n+m)\bar \delta}}}$, {for scalars $\delta,\bar \delta>0$.} 
Then {for all $k$, we have that $H_k = H_k^T$ and $\mathbb E[{H_k\mid\mathcal F_k}]=H_k$ and
$
{\ulambda_{k}\mathbf{I}  \preceq H_{k} \preceq \olambda_k \mathbf{I}}$ both hold in an a.s. fashion.}
\end{itemize}
\end{lemma}

We now derive a rate statement for the mean sub-optimality.} 
\begin{theorem}[{\bf Convergence in mean}]\label{thm:mean:nonsmooth}
Consider the \eqref{rsVS-SQN} scheme. Suppose Assumptions  ~\ref{assum:convex2}, \ref{state noise} (NS-M), \ref{state noise} (NS-B), \ref{assump:Hk} (NS),  {and \ref{growth}} 
 hold. Let {$\gamma_k=\gamma$, $\mu_k=\mu$, and $\eta_k=\eta$ be chosen such that  \eqref{mainLemmaCond} holds ({where $L=1/\eta$}).} 
{If {$\bar x_K \triangleq \frac{\sum_{k=0}^{K-1}x_k(\ulambda\mu\gamma-C/N_k)}{\sum_{k=0}^{K-1}(\ulambda\mu\gamma-C/N_k)}$}, then \eqref{non_smooth_lemma} holds for $K \geq 1$ and {$C={2(1+\mu\eta)\olambda^2\nu_1^2\gamma^2\over \alpha \eta}$}.}
 \begin{align}\label{non_smooth_lemma}
\left(K\ulambda \mu \gamma{-\sum_{k=0}^{K-1}{C\over N_k}}\right)\mathbb E[{f_{\eta,\mu}(\bar x_{K})-f^*}] 
\nonumber&\leq\mathbb E[f_{\eta,\mu}(x_0)-f^*]+\eta B^2+{\ulambda \mbox{dist}^2(x_0,X^*)\over 2}\mu^2\gamma K\\&+\sum_{k=0}^{K-1}{(1+\mu\eta)\olambda^2( {2\nu_1^2\|x^*\|^2+\nu_2^2})\gamma^2\over 2N_k\eta}.
\end{align}
\end{theorem}
\begin{proof} Since Lemma \ref{lemma:main-ineq} {may be invoked}, by taking
expectations on both sides of \eqref{ineq:cond-recursive-F-k}, for
any $k\geq  0$   {letting $\bar{w}_{k,N_k} \triangleq \frac{\sum_{j=1}^{N_k} \left({\nabla_{x}}
{F}_{\eta_k,\mu_k}(x_k,\omega_{j,k})-\nabla f_{\eta_k,\mu_k}(x_k)\right)}{N_k},$} and by letting  {${{\ulambda}\triangleq {1\over (m+n)(1/\eta^\delta+\mu^{\bar \delta})}}$,
{${\olambda}\triangleq {(m+n)^{n+m-1}(1/\eta^\delta+\mu^{\bar \delta})^{n+m-1}\over (n-1)!\mu^{(n+m)\bar \delta}}$}}, {  using the quadratic growth property i.e. $\|x_k-x^*\|^2\leq {2\over \alpha}\left(f(x)-f(x^*)\right)$ and the fact that $\|x_k-x^*+x^*\|^2\leq 2\|x_k-x^*\|^2+2\|x^*\|^2$, we obtain the following}
\begin{align*}
\mathbb E[{f_{{\eta},{\mu}}(x_{k+1})-f^*}] 
& \leq \left(1-{{\ulambda}}{\mu\gamma}{+{2(1+\mu \eta)\olambda^2\nu_1^2\gamma^2\over \alpha N_k\eta}}\right)\mathbb E[{f_{\eta,{\mu}}(x_k)-f^*}]
 + {\ulambda \mbox{dist}^2(x_0,X^*)\over 2}\mu^2\gamma\\& +{(1+\mu\eta)\olambda^2( {2\nu_1^2\|x^*\|^2+\nu_2^2})\gamma^2\over 2N_k\eta}
\end{align*}
\begin{align*}
\implies \left( \ulambda{\mu\gamma}{-{2(1+\mu \eta)\olambda^2\nu_1^2\gamma^2\over \alpha N_k\eta}}\right)\mathbb E[{f_{\eta,{\mu}}(x_k)-f^*}] 
&  \leq \mathbb E[{f_{\eta,{\mu}}(x_k)-f^*}]-   \mathbb E[{f_{{\eta},{\mu}}(x_{k+1})-f^*}] \\ 
 +{\ulambda \mbox{dist}^2(x_0,X^*)\mu^2\gamma\over 2} 
& +{(1+\mu\eta)\olambda^2( {2\nu_1^2\|x^*\|^2+\nu_2^2})\gamma^2\over 2N_k\eta} .\end{align*}
Summing from $k=0$ to $K-1$ and by invoking {Jensen's inequality}, we obtain the following
 \begin{align*}
\left(K\ulambda \mu \gamma{-\sum_{k=0}^{K-1}{C\over N_k}}\right)\mathbb E[{f_{\eta,\mu}(\bar x_{K})-f^*}] 
&\leq\mathbb E[f_{\eta,\mu}(x_0)-f^*]-\mathbb E[f_{\eta,\mu}(x_K)-f^*]\\&+{\ulambda \mbox{dist}^2(x_0,X^*)\over 2}\mu^2\gamma K+\sum_{k=0}^{K-1}{(1+\mu\eta)\olambda^2( {2\nu_1^2\|x^*\|^2+\nu_2^2})\gamma^2\over 2N_k\eta},
\end{align*}
where {$C={2(1+\mu\eta)\olambda^2\nu_1^2\gamma^2\over \alpha\eta}$} and {$\bar x_K \triangleq \frac{\sum_{k=0}^{K-1}x_k(\ulambda\mu\gamma-C/N_k)}{\sum_{k=0}^{K-1}(\ulambda\mu\gamma-C/N_k)}$}.  Since $\mathbb E[{f(x)}]\leq  \mathbb
E[f_{\eta}(x)]+\eta_kB^2$ and $f_\mu(x)=f(x)+{\mu\over 2}\|x-x_0\|^2$, {we have that} $-\mathbb E[f_{\eta,\mu}(x_K)-f^*]\leq -{\mathbb{E}}[f_\mu(x_K)-f^*]+\eta B^2\leq \eta B^2$. Therefore, we obtain the following:
 \begin{align*}
\left(K\ulambda \mu \gamma{-\sum_{k=0}^{K-1}{C\over N_k}}\right)\mathbb E[{f_{\eta,\mu}(\bar x_{K})-f^*}] &
\leq\mathbb E[f_{\eta,\mu}(x_0)-f^*]+\eta B^2\\&+{\ulambda \mbox{dist}^2(x_0,X^*)\over 2}\mu^2\gamma K+\sum_{k=0}^{K-1}{(1+\mu\eta)\olambda^2( {2\nu_1^2\|x^*\|^2+\nu_2^2})\gamma^2\over 2N_k\eta}. 
\end{align*}
\end{proof}

{We refine this result for a set of parameter sequences.} 
\begin{theorem}[{\bf Rate statement and oracle complexity}]\label{thm:rate K}
Consider \eqref{rsVS-SQN} and  suppose Assumptions ~\ref{assum:convex2}, \ref{state noise} (NS-M), \ref{state noise} (NS-B), \ref{assump:Hk} (NS),  {and \ref{growth}}  hold, $\gamma {\triangleq} c_\gamma K^{-1/3+\bar \varepsilon}$, $\mu {\triangleq} {K^{-1/3}}$, $\eta \triangleq K^{-1/3}$ and $N_k\triangleq \lceil N_0{(k+1)}^{a}\rceil$, where $\bar \varepsilon \triangleq \tfrac{5\varepsilon}{3}$, $\varepsilon>0$, {$N_0>{C\over \ulambda \mu \gamma}$},  {$C={2(1+\mu\eta)\olambda^2\nu_1^2\gamma^2\over \alpha\eta}$} and $a>1$. Let $\delta={\varepsilon\over n+m-1}$ and $\bar \delta={\varepsilon\over n+m}$. 

\noindent (i) For any $K \geq 1$, 
$ \mathbb{E}[f(\bar x_{ K})]-f^*\leq {\mathcal O}(K^{-1/3}).$ 

\noindent  (ii) Let  ${\epsilon>0}$, {$a = (1+\epsilon)$}, and $ K\geq 1$ such that $\mathbb E[f(\bar x_{ K})]-f^*\leq {\epsilon}$. {Then{,} {$\sum_{k=0}^{ K}N_k\leq  \mathcal O\left({ {\epsilon}^{-{(2+\varepsilon)\over 1/3}}}\right)$}}. 
\end{theorem}
\begin{proof}
(i) {First, note that for $a>1$ and $N_0>{C\over \ulambda\mu\gamma}$ we have $\sum_{k=0}^{K-1} {C\over N_k}<\infty$. Therefore we can let $C_4\triangleq \sum_{k=0}^{K-1}{C\over N_k}. $} { Dividing both sides} of \eqref{non_smooth_lemma} by $K\ulambda\mu\gamma{-C_4}$ {and by recalling} that $f_\eta(x)\leq f(x)\leq f_\eta(x)+\eta B^2$ and $f(x)\leq f_\mu(x)$, we obtain
\begin{align*}
\mathbb E[{f(\bar x_{K})-f^*}] 
& \leq{\mathbb E[f_\mu(x_0)-f^*]\over K\ulambda \mu \gamma{-C_4}}+{\eta B^2\over K\ulambda \mu \gamma{-C_4}}+\frac{{\ulambda \mbox{dist}^2(x_0,X^*)\over 2}\mu^2\gamma K}{K\ulambda\mu\gamma{-C_4}} \\&+\frac{\sum_{k=0}^{K-1}{(1+\mu\eta)\olambda^2( {2\nu_1^2\|x^*\|^2+\nu_2^2})\gamma^2\over 2N_k\eta}}{K\ulambda\mu\gamma{-C_4}}+\eta B^2. 
\end{align*}
Note that by choosing $\gamma=c_\gamma K^{-1/3+\bar \varepsilon}$, $\mu={K^{-1/3}}$ and $\eta=K^{-1/3}$, where $\bar \varepsilon=5/3\varepsilon$, inequality \eqref{mainLemmaCond} is satisfied for sufficiently small $c_\gamma$. By choosing$N_k=\lceil N_0{(k+1)}^a\rceil\geq N_0 (k+2)^a$ for any $a>1$ and {$N_0>{C\over \ulambda\mu\gamma}$}, we have that 
\begin{align*}
& \sum_{k=0}^{K-1}{1\over (k+1)^a}
\leq 1+\int_{0}^{K-1} (x+1)^{-a}dx\leq 1+{K^{1-a}\over 1-a} \\ 
\implies &\mathbb E[{f(\bar x_{K})-f^*}] 
\leq{C_1\over K\ulambda \mu \gamma{-C_4}}+{\eta B^2\over K\ulambda \mu \gamma{-C_4}}+{C_2\ulambda\mu^2\gamma K \over K\ulambda\mu\gamma{-C_4}}+{C_3(1+\mu\eta)\olambda^2\gamma^2\over \eta N_0(K \mu \gamma{-C_4})}(1+K^{1-a})+\eta B^2,
\end{align*}
where $C_1=\mathbb E[f_\mu(x_0)-f^*]$, $C_2={ \mbox{dist}^2(x_0,X^*)\over 2}$ and $C_3={ {2\nu_1^2\|x^*\|^2+\nu_2^2}\over 2(1-a)}$. Choosing the parameters $\gamma,\mu$ and $\eta$ as stated and noting that  {${{\ulambda}= {1\over (m+n)(1/\eta^\delta+\mu^{\bar \delta})}}=\mathcal O(\eta^\delta)= \mathcal O(K^{-\delta/3})$ and $\olambda={(m+n)^{n+m-1}(1/\eta^\delta+\mu^{\bar \delta})^{n+m-1}\over (n-1)!\mu^{(n+m)\bar \delta}}=\mathcal O(\eta^{-\delta(n+m-1)/\mu^{\bar \delta(n+m)}})= \mathcal O(K^{2\varepsilon/3})$, where we used the assumption that $\delta={\varepsilon\over n+m-1}$ and $\bar \delta={\varepsilon\over n+m}$}. Therefore, we obtain
$\mathbb E[{f(\bar x_{K})-f^*}] 
\leq \mathcal O(K^{-1/3-5\varepsilon/3}+\delta/3)+\mathcal O(K^{-2/3-5\varepsilon/3+\delta/3})+\mathcal O(K^{-1/3})+\mathcal O(K^{-2/3+3\varepsilon})+\mathcal O(K^{-1/3})= \mathcal O(K^{-1/3}).$ 

(ii) The proof is similar to part (ii) of Theorem \ref{oracle smooth}. 
\end{proof}

\begin{remark}
{Note that in Theorem \ref{thm:rate K} we choose steplength, regularization,
and smoothing parameters  {as constant parameters in accordance with the length of the simulation trajectory $K$, i.e.  $\gamma,\mu,\eta$ are constants.}  This is akin
to the avenue chosen by Nemirovski et al.~\cite{nemirovski_robust_2009} where
the steplength is chosen in accordance with the length of the simulation
trajectory $K$.}   
\end{remark}

Next, we relax Assumption
\ref{growth} (quadratic growth property) and impose a stronger  bound on the
conditional second moment of the sampled gradient.
\begin{assumption}\label{non growth}
\vvs{Let $\bar{w}_{k,N_k} \triangleq \nabla_x f_{\eta_k}(x_k) - \tfrac{\sum_{j=1}^{N_k} \nabla_x F_{\eta_k}(x_k,\omega_{j,k})}{N_k}$. Then 
there exists $\nu>0$ such that $\mathbb{E}[\|\bar{w}_{k,N_k}\|^2\mid \mathcal{F}_k] \leq \tfrac{\nu^2}{N_k}$ and $\mathbb{E}[\bar{w}_{k,N_k} \mid \mathcal{F}_k] = 0$ {hold} almost surely 
 for all $k$ and $\eta_k > 0$, where $\mathcal{F}_k
	\triangleq \sigma\{x_0, x_1, \hdots, x_{k-1}\}$.}
	\end{assumption}
\begin{corollary}
[{\bf Rate statement and Oracle complexity}]
Consider the \eqref{rsVS-SQN}  scheme.  Suppose Assumptions ~\ref{assum:convex2}, \ref{assump:Hk} (NS) and \ref{non growth} hold and $\gamma {\triangleq} c_\gamma K^{-1/3+\bar \varepsilon}$, $\mu {\triangleq} {K^{-1/3}}$, $\eta \triangleq K^{-1/3}$ and $N_k\triangleq \lceil{(k+1)}^{a}\rceil$, where $\bar \varepsilon \triangleq \tfrac{5\varepsilon}{3}$, $\varepsilon>0$ and $a>1$.

\noindent (i) For any $K \geq 1$, $
 \mathbb E[f(\bar x_{ K})]-f^*\leq \mathcal O(K^{-1/3}). $
 
\noindent  (ii) Let  ${\epsilon>0}$, $a = (1+\epsilon)$, and $ K\geq 1$ such that $\mathbb E[f(\bar x_{ K})]-f^*\leq {\epsilon}$. {Then{,} {$\sum_{k=0}^{ K}N_k\leq  \mathcal O\left({ {\epsilon}^{-{(2+\varepsilon)\over 1/3}}}\right)$}}. 

\end{corollary}

\jal{\begin{remark}
It is worth emphasizing that the unavailability of problem parameters such as the
Lipschitz constant and the strong convexity modulus may render some methods 
unimplementable. {In fact, this is a prime motivation for utilizing
smoothing, regularization, and diminishing steplengths in this work; smoothing and regularization  address the unavailability
of Lipschitz and strong convexity constants, respectively while  
diminishing steplengths obviate the need for knowing both sets of parameters. Remark 2 cover some instances where  such
constants may indeed be available. We briefly summarize avenues for contending with the absence of
such parameters.} 

\begin{itemize}
	\item [1.] {\em Absence/unavailability of Lipschitz constant in strongly-convex regimes ({\bf sVS-SQN}).} In the absence of a Lipschitz constant but in the presence of strong convexity, we propose two distinct avenues. We can employ Moreau-smoothing with fixed parameter $\eta$ or iterative smoothing which relies on a diminishing sequence $\{\eta_k\}$. This framework relies on a modified L-BFGS update, referred to as smoothed L-BFGS and denoted by ({\bf sL-BFGS}). {This case is discussed in section 3.2 and the main results are provided by Theorem 3.} 
	\item [2.] {\em Absence/unavailability of strong convexity constants in smooth regimes ({\bf rVS-SQN}).} 
If the problem is either merely convex or has an unknown strong convexity modulus but satisfies $L$-smoothness, then we present a regularization scheme, extending our prior work in this area to the variance-reduced arena. The resulting scheme, referred to as the regularized VS-SQN and denoted by ({\bf rVS-SQN}), is reliant on the analogous L-BFGS update. {This case is discussed in section 4.1 and the main results are provided by Theorems 4 and 5.} 
	\item [3.] {\em Absence/unavailability of Lipschitz constants and strong convexity constants.}   
If both the Lipschitz and the strong convexity constants are unavailable, then we may overlay smoothing and regularization. This requires a regularized and smoothed L-BFGS update, referred to as ({\bf rsL-BFGS}), and the resulting scheme is referred to as the regularized smoothed VS-SQN and denoted by ({\bf rsVS-SQN}). {This case is discussed in section 4.2 and the main results are provided by Theorems 6 and 7.} 
	\item [4.] {\em Addressing unavailability of Lipschitz and strong convexity constants via diminishing steplength schemes.} One may also obviate the need for knowing the constants by utilizing diminishing steplength sequences. This allows for deriving rate statements for a shifted recursion via some well known results on deterministic recursions. This has now been added in Theorem 1 (iii) ({\bf (VS-SQN)} in the smooth and strongly convex regime) and Theorem 3 (iii) ({\bf (sVS-SQN)} in the nonsmooth and strongly convex regime).   
 	\item[5.] {\em Stochastic line-search techniques.}  One avenue to contend with the absence of
Lipschitzian guarantees lies in leveraging line-search techniques in stochastic
regimes~\cite{jofre2017variance,cartis18global,paquette20stochastic,shashaani18astro}.
This remains a goal of future work.  
\end{itemize}
\end{remark}}
\section{Numerical results}\label{sec:5} In this section, we compare the behavior of the proposed VS-SQN schemes with their accelerated/unaccelerated gradient counterparts on a class of strongly convex/convex and smooth/nonsmooth stochastic optimization problems {with the intent of examining empirical error and sparsity of estimators (in machine learning problems) as well as the ability to contend with ill-conditioning.}  

\noindent {\bf Example 1.} First, we consider the logistic regression problem, defined  as follows:
\begin{align}\tag{LRM}
\min_{x \in \mathbb R^n} \ f(x) \triangleq \frac{1}{N}\sum_{i=1}^N\ln \left(1+{\exp} \left(-u_i^Txv_i\right)\right),
\end{align}
where $u_i \in \mathbb R^n$ is the input binary vector associated with article
$i$ and $v_i \in \{-1,1\}$ represents the class of the $i$th article. A {$\mu$-}regularized variant of such a problem is defined as follows.
\begin{align}\label{logisticReg}\tag{reg-LRM}
\min_{x \in \mathbb R^n} \ f(x) \triangleq \frac{1}{N}\sum_{i=1}^N\ln \left(1+{\exp}\left(-u_i^Txv_i\right)\right)+\frac{\mu}{2}\|x\|^2.
\end{align}
We consider the {\sc sido0} dataset~\cite{lewis2004rcv1} where $N = 12678$ and $n = 4932$. 

\noindent {\bf (1.1) Strongly convex and smooth problems}: To apply \eqref{VS-SQN}, we consider (Reg-LRM) where the problem is strongly convex and 
$\mu=0.1$. We compare the behavior of the scheme with an accelerated gradient scheme~\cite{jalilzadeh2018optimal} and set the overall sampling buget equal to $1e4$. {We observe that \eqref{VS-SQN} competes well with ({\bf VS-APM}).} (see Table~\ref{SC_tab_smooth} and Fig.~\ref{fig} (a)).  

\begin{table}[htb]
	\centering
	\scriptsize
	\begin{tabular}{|c|c|c||c|c|} \hline 
	&\multicolumn{2}{|c||}{SC, smooth}&\multicolumn{2}{|c|}{SC, nonsmooth \aj{(Moreau smoothing)}}\\ \hline
	& {\bf VS-SQN}& {\bf VS-APM} &{\bf sVS-SQN}& {\bf sVS-APM} \\ \hline  \hline
	sample size: $N_k$& $\rho^{-k}$&$\rho^{-k}$&$\lfloor q^{-k}\rfloor$&$\lfloor q^{-k}\rfloor$\\ \hline
	steplength: $\gamma_k$&0.1&0.1&$\eta_k^2$&$\eta_k^2$\\ \hline
	smoothing: $\eta_k$&-&-&$0.1$&$0.1$\\ \hline
	$f(x_k)$& $5.015$e-$1$&$5.015$e-$1$&$8.905$e-$1$&$1.497$e+$0$\\ \hline
	 	\end{tabular}	
			\caption{{\bf sido0:} SC, smooth and nonsmooth}
	\label{SC_tab_smooth}
	\vspace{-0.2in}
\end{table}
\jal{Next, \uvs{in Table \ref{cons_dim}, we compare the behavior of {\bf (VS-SQN)} in constant versus diminishing steplength regimes and observe that the distinctions are modest for these problem instances.}}
\begin{table}[htb]
	\centering
	\scriptsize
	\begin{tabular}{|c|c|c|c||c|c|c|} \hline 
	&\multicolumn{3}{|c||}{SC, smooth}&\multicolumn{3}{|c|}{SC, nonsmooth \aj{(Moreau smoothing)}}\\ \hline
	& {\bf VS-SQN}& {\bf VS-SQN}& {\bf {VS-SQN}} &{\bf sVS-SQN}& {\bf sVS-SQN} &{\bf sVS-SQN} \\ \hline  \hline
	sample size: $N_k$& $\rho^{-k}$&$k$&$k^2$&$\lfloor q^{-k}\rfloor$&$k$&$k^2$\\ \hline
	steplength: $\gamma_k$&0.1&$1/\sqrt k$&$1/\sqrt k$&$\eta_k^2$&$1/ \sqrt{k}$&$1/\sqrt k$\\ \hline
	smoothing: $\eta_k$&-&-&-&$0.1$&$0.1$&$0.1$\\ \hline
	$f(x_k)$& $5.015$e-$1$&$5.019$e-$1$&$5.098$e-$1$&$8.905$e-$1$&$1.358$e+$0$&$8.989$e-$1$\\ \hline
	 	\end{tabular}\caption{{\bf sido0:} SC, smooth and nonsmooth}\label{cons_dim}	
	\vspace{-0.2in}
\end{table}

\noindent {\bf (1.2) Strongly convex and nonsmooth}: We consider a nonsmooth
variant where an $\ell_1$ regularization is added with  $\lambda=\mu=0.1$:
\begin{align}\label{SC nonsmooth LRM}
\min_{x \in \mathbb R^n} f(x):=\frac{1}{N}\sum_{i=1}^N\ln \left(1+\mathbb{E}\left(-u_i^Txv_i\right)\right)+{\mu\over 2}\|x\|^2+\lambda\|x\|_1.
\end{align}
From~\cite{beck12smoothing}, a smooth approximation of $\|x\|_1$ is given by the following
$$\sum_{i=1}^n H_\eta (x_i) = \begin{cases} x_i^2/2\eta, &\mbox{if } |x_i|\leq \eta \\ 
|x_i|-\eta/2, & \mbox{o.w.}  \end{cases},$$ where $\eta$ is a smoothing parameter.
The perfomance of \eqref{sVS-SQN} is shown in Figure \ref{fig} (b) while
parameter choices are provided in Table \ref{SC_tab_smooth} and the total
sampling budget is $1e5$. {We see that empirical behavior of \eqref{VS-SQN} } and
\eqref{sVS-SQN} is similar to {\bf (VS-APM)}{~\cite{jalilzadeh2018optimal} and {\bf  (rsVS-APM)}~\cite{jalilzadeh2018optimal}, respectively. Note that while
in the strongly convex regimes, both schemes {display} similar (linear) rates, we do
not have a rate statement for smoothed ({\bf sVS-APM})~\cite{jalilzadeh2018optimal}.}  

	\begin{figure}[htb]
\vspace{-0.1in}
	\centering
	{
\includegraphics[scale=0.085]{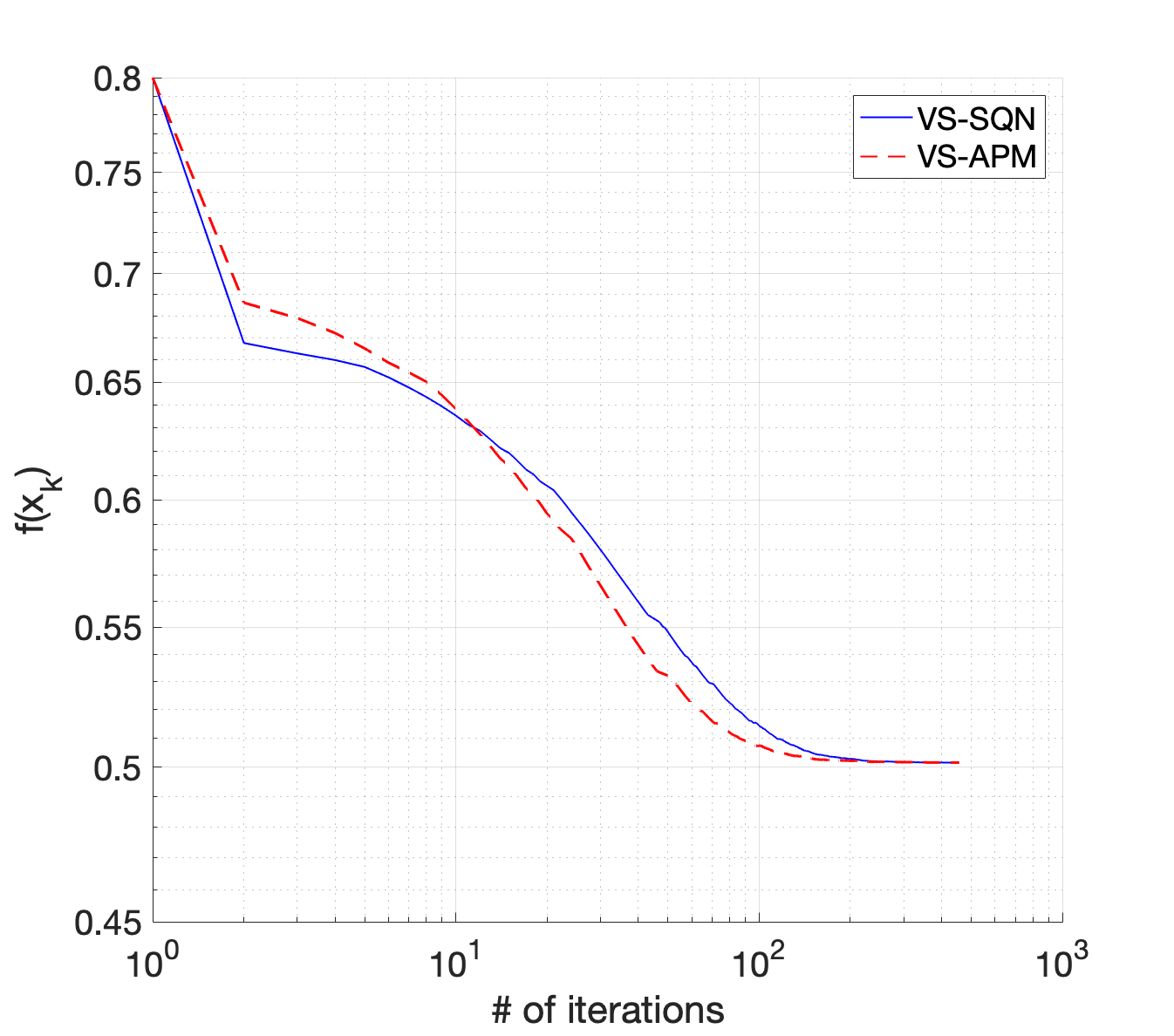}
	\includegraphics[scale=0.085]{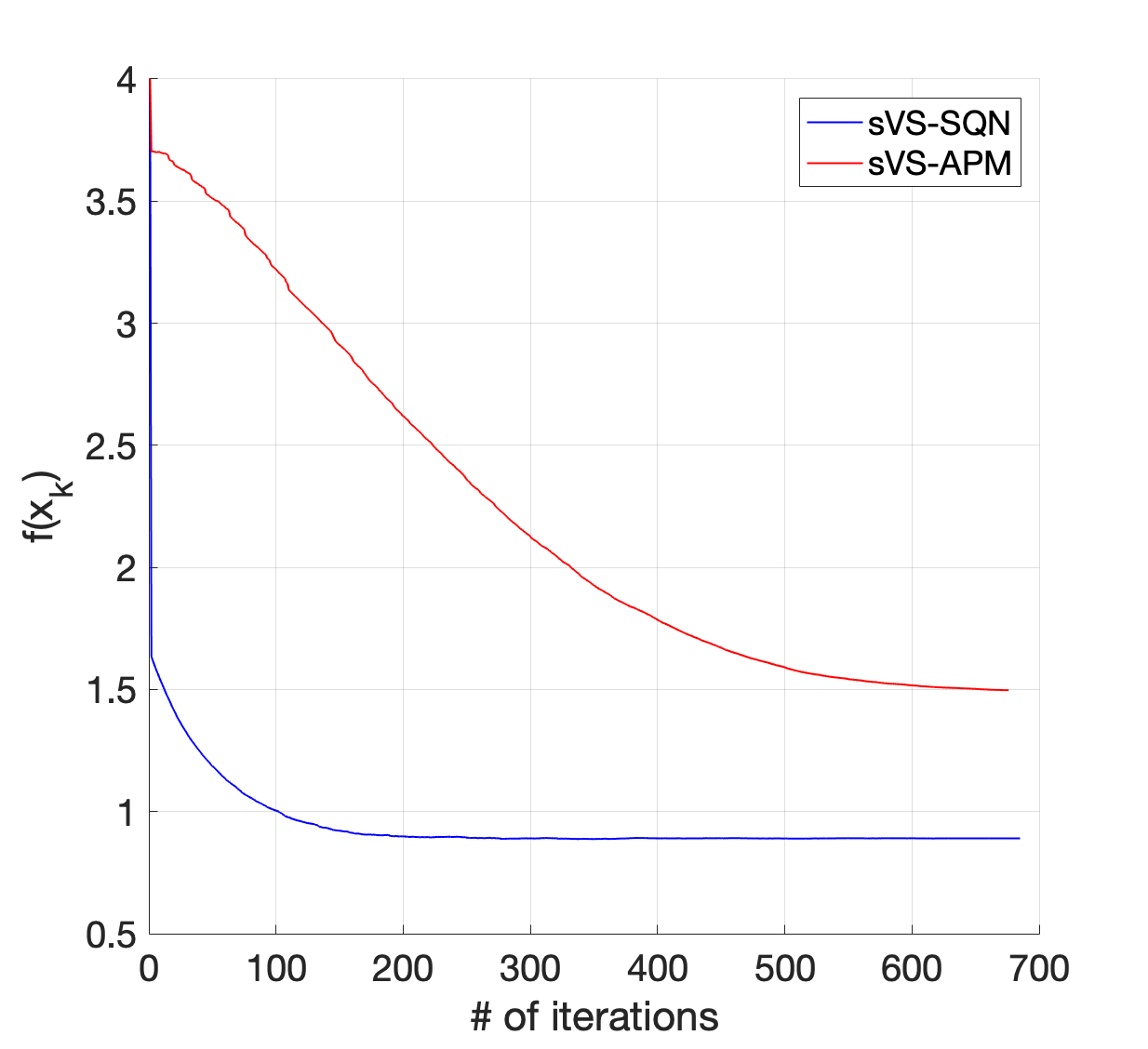}
	\includegraphics[scale=0.085]{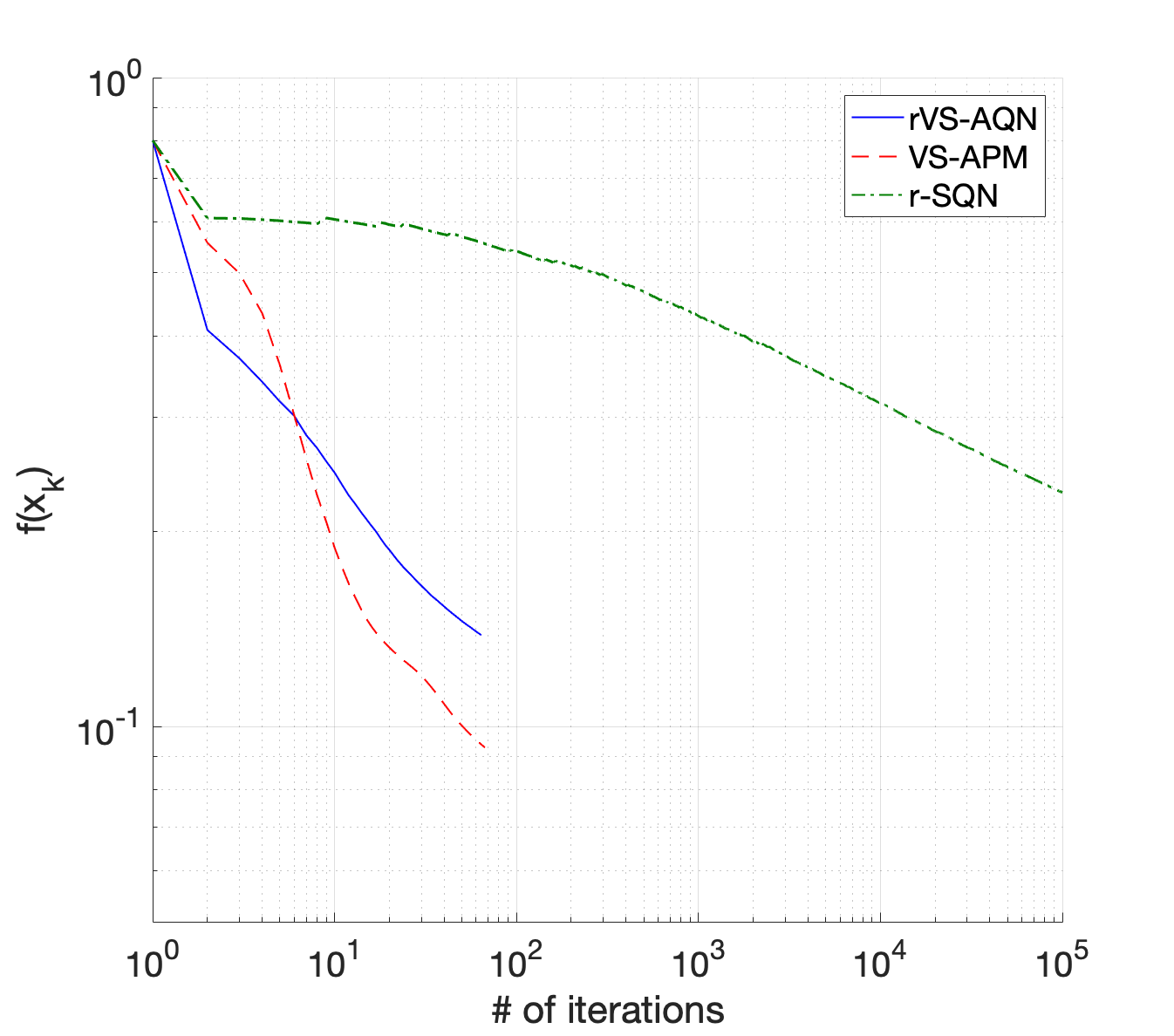}
	\includegraphics[scale=0.085]{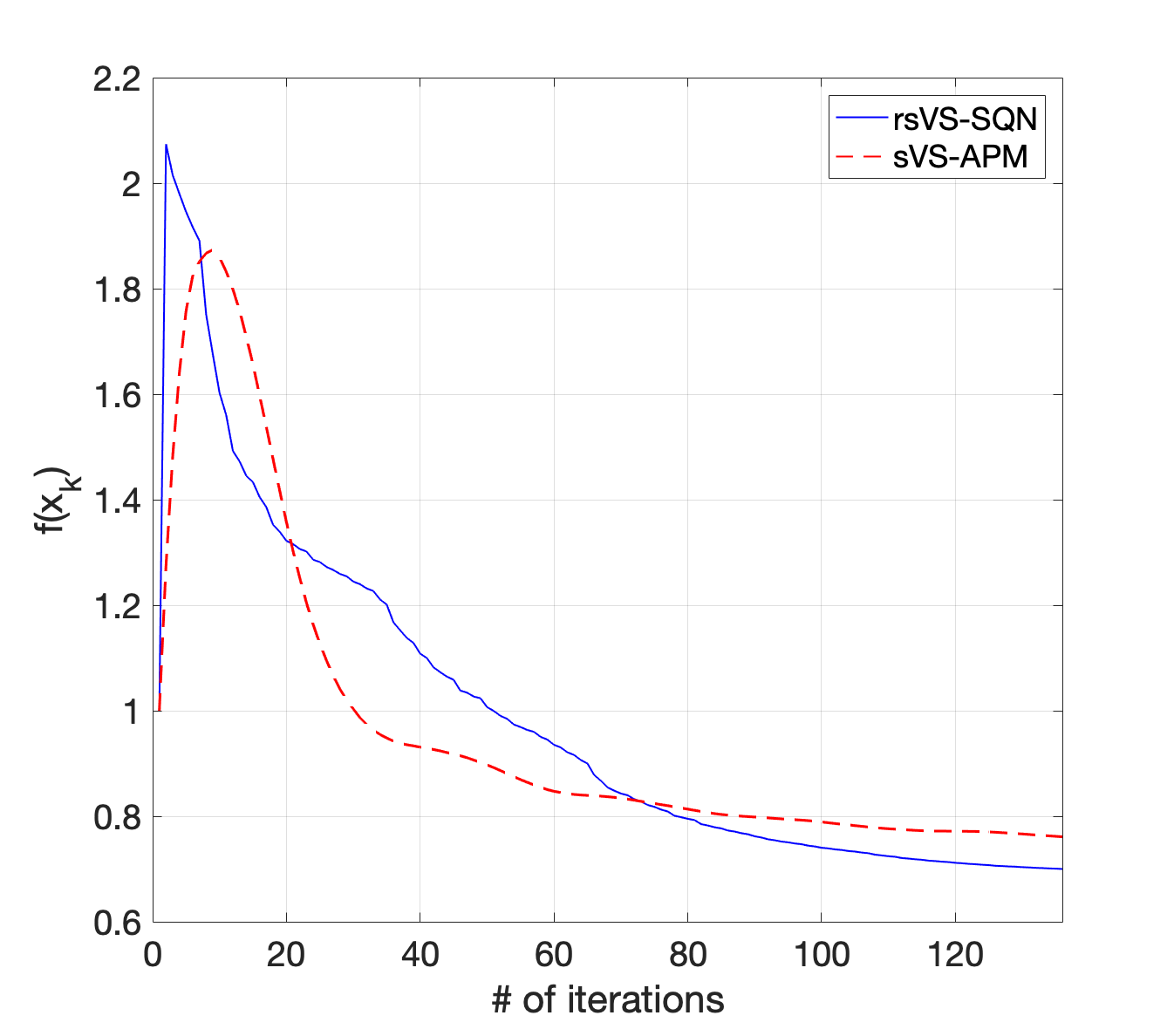}}
\caption{Left to right: (a) SC smooth, (b) SC nonsmooth, (c) C smooth, (d) C nonsmooth\label{fig}}{}
	\end{figure}
\noindent {\bf (1.3) Convex and smooth}: We implement \eqref{rVS-SQN} on the
(LRM) problem and compare the result with VS-APM~\cite{jalilzadeh2018optimal}
and r-SQN~\cite{yousefian2017stochastic}. We again consider the {\sc sido0}
dataset with a total budget of $1e5$ while the parameters are tuned to ensure
good performance.  In Figure \ref{fig} (c) we compare three different methods
while the choices of steplength and sample size can be seen in
Table~\ref{compare_tab}. \us{We note that (VS-APM) produces slightly better solutions, which is not surprising since it enjoys a rate of $\mathcal{O}(1/k^2)$ with an optimal oracle complexity. However, \eqref{rVS-SQN} is competitive and appears to be better than (r-SQN) by a significant margin in terms of the function value.} \jal{In addition, the last three columns of Table~\ref{cons_dim} capture the distinctions in performance between employing constant versus  diminishing steplengths.}

\begin{table}[htb]
\centering
\scriptsize
	\begin{tabular}{|c|c|c|c||c|c|} \hline 
	&\multicolumn{3}{|c||}{convex, smooth}&\multicolumn{2}{|c|}{convex, nonsmooth}\\ \hline
	& {\bf rVS-SQN}& r-SQN & VS-APM & {\bf rsVS-SQN}&  sVS-APM \\ \hline  \hline
	sample size: $N_k$& $k^{2+\varepsilon}$&1&$k^{2+\varepsilon}$& $(k+1)^{1+\varepsilon}$&$(k+1)^{1+\varepsilon}$\\ \hline
	steplength: $\gamma_k$&$k^{-\varepsilon}$&$k^{-2/3}$&$1/(2L)$&$K^{-1/3+\varepsilon}$&$1/(2k)$\\ \hline
	regularizer: $\mu_k$&$k^{2/3\varepsilon-1}$&$k^{-1/3}$&-&$K^{-1/3}$&-\\ \hline
	smoothing: $\eta_k$&-&-&-&$K^{-1/3}$&$1/k$\\ \hline
	$f(x_k)$&1.38e-1&2.29e-1&9.26e-2&6.99e-1&7.56e-1\\ \hline
	 	\end{tabular}
	\caption{	{\bf sido0:} C, smooth and nonsmooth}
	\label{compare_tab}
	\end{table}

\noindent {\bf (1.4.) Convex and nonsmooth}: Now we consider the nonsmooth problem in which $\lambda=0.1$.
\begin{align}\label{nonsmooth LRM}
\min_{x \in \mathbb R^n} f(x):=\frac{1}{N}\sum_{i=1}^N\ln \left(1+\exp\left(-u_i^Txv_i\right)\right)+\lambda\|x\|_1. 
\end{align}
We implement {\bf rsVS-SQN} scheme with a  total budget of $1e4$. (see Table~\ref{compare_tab} and Fig.~\ref{fig} (d)) \us{observe that it competes well with (sVS-APM)~\cite{jalilzadeh2018optimal}, which has a superior convergence rate of $\mathcal{O}(1/k)$.}    

\blue{\noindent {\bf (1.5.) Sparsity} {We now compare} the sparsity of the estimators obtained via (\ref{rVS-SQN}) scheme with averaging-based stochastic gradient schemes. Consider the following example where we consider the smooth approximation of $\|.\|_1$, leading to a convex and smooth problem. 
\begin{align*}
\min_{x \in \mathbb R^n} f(x):=\frac{1}{N}\sum_{i=1}^N\ln \left(1+\exp\left(-u_i^Txv_i\right)\right)+\lambda\aj{\sum_{i=1}^n\sqrt{x_i^2+\lambda_2}},
\end{align*}
where we set $\lambda=1$e-$4$. We chose the parameters according to Table
\ref{compare_tab}, total budget is $1e5$ and $\|x_K\|_0$ denotes the number of
entries in $x_K$ that are greater than $1$e-$4$. Consequently, {$n_0 \triangleq n - \|x_K\|_0$}
denotes the number of ``zeros'' in the vector. As it can be seen in Table
\ref{compare_spars}, the solution obtained by (\ref{rVS-SQN}) is significantly
{sparser than that obtained by} ({\bf VS-APM}) and standard stochastic gradient. In fact,
SGD produces nearly dense vectors while (\ref{rVS-SQN}) produces vectors,
$10\%$ of which are sparse for $\lambda_2 = 1e$-$6.$} \begin{table}[htb]
\centering
\scriptsize
	{\begin{tabular}{|c|c|c|c|c|} \hline 
	&{\bf rVS-SQN}&({\bf VS-APM})&SGD\\ \hline 
	$N_k$&$k^{2+\epsilon}$&$k^{2+\epsilon}$&1\\ \hline
	$\#$ of iter.&66&66&1e5 \\ \hline
	$n_0$ for $\lambda_2=1$e-$5$&144&31&0\\ \hline
	$n_0$ for $\lambda_2=1$e-$6$&497&57&2\\ \hline
	 	\end{tabular}}
		\caption{{\bf sido0:} Convex, smooth }
	\label{compare_spars}
	\end{table}
	
\noindent {\bf Example 2. Impact of size and ill-conditioning.} {In Example {1}, we observed that \eqref{rVS-SQN} {competes well} with VS-APM for a subclass of machine learning problems. We now consider a stochastic quadratic program over a general probability space and observe similarly competitive behavior. In fact,  \eqref{rVS-SQN} often outperforms ({\bf VS-APM})~\cite{jalilzadeh2018optimal} (see Tables~\ref{sc_tab_example} and \ref{c_tab_example})}. We consider the following problem.
\begin{align*}
\min_{x\in \mathbb R^n} \mathbb E\left[{1\over 2}x^TQ(\omega)x+c(\omega)^Tx\right],
\end{align*}
where $Q(\omega)\in \mathbb R^{n\times n}$  is a random symmetric matrix such that the eigenvalues are chosen uniformly at random and the minimum eigenvalue is one and zero for strongly convex and convex problem, respectively.  Furthermore, $ {c_\omega}=-Q(\omega)x^0$, where $x^0\in \mathbb R^{n\times 1}$ is a vector whose elements are chosen randomly from the standard Gaussian distribution. 
\begin{table}[htb]
\begin{minipage}[b]{0.5\linewidth}
	
	\scriptsize
	\begin{tabular}{|c|c|c|} \hline 
	&\eqref{VS-SQN}&({\bf VS-APM})\\ \hline
	n&$\mathbb E[f(x_k)-f(x^*)]$&$\mathbb E[f(x_k)-f(x^*)]$\\ \hline
	20&$3.28$e-$6$& $5.06$e-$6$ \\ \hline
	60&$9.54$e-$6$& $1.57$e-$5$\\ \hline
	100&$1.80$e-$5$&$2.92$e-$5$\\ \hline
	\end{tabular}
	\caption{Strongly convex: \\  \eqref{VS-SQN} vs ({\bf VS-APM})}
	\label{sc_tab_example}
	\end{minipage}
	\begin{minipage}[b]{0.45\linewidth}
		
	\scriptsize
	\begin{tabular}{|c|c|c|} \hline 
	&\eqref{rVS-SQN}&({\bf VS-APM})\\ \hline
	n&$\mathbb E[f(x_k)-f(x^*)]$&$\mathbb E[f(x_k)-f(x^*)]$\\ \hline
	20&$9.14$e-$5$&$1.89$e-$4$ \\ \hline
	60&$2.67$e-$4$&$4.35$e-$4$\\ \hline
	100&$5.41$e-$4$&$8.29$e-$4$\\ \hline
	\end{tabular}
	\caption{Convex: \\ \eqref{rVS-SQN} vs ({\bf VS-APM})}
	\label{c_tab_example}
	\end{minipage}
\end{table}
{In Tables \ref{quad_ill} and \ref{convex_ill}, we compare the behavior of \eqref{rVS-SQN} and ({\bf VS-APM}) when
the problem is ill-conditioned {in strongly convex and convex regimes,
respectively}. {In strongly convex regimes}, we set the total budget equal
to $2e8$ and maintain the steplength as equal for both schemes. The sample size
sequence is chosen to be $N_k=\lceil 0.99^{-k}\rceil$, leading to $1443$ steps
for both methods. {We observe that as $m$ grows, the relative quality of the
solution compared to ({\bf VS-APM}) improves even further.} {These findings are reinforced in 
Table \ref{convex_ill}, where for merely convex problems, although  the convergence rate for ({\bf VS-APM})
is $\mathcal O(1/k^2)$ (superior to $\mathcal O(1/k)$ for (\ref{rVS-SQN}), (\ref{rVS-SQN})
outperforms ({\bf VS-APM}) in terms of empirical error. Note that parameters are chosen
similar to Table \ref{compare_tab}. }

\begin{table}[htbp]
\begin{minipage}[b]{0.5\linewidth}
	\centering
	\tiny
	\begin{tabular}{|c|c|c|c|} \hline 
	&\multicolumn{3}{|c|}{$\mathbb E[f(x_k)-f(x^*)]$}\\ \hline
	$\kappa$&\eqref{VS-SQN}, $m=1$&\eqref{VS-SQN}, $m=10$&({\bf VS-APM})\\ \hline
	$1e5$ &$9.25$e-$4$&$2.656$e-$4$& $2.600$e-$3$\\ \hline
	$1e6$ &$9.938$e-$5$&$4.182$e-$5$&$4.895$e-$4$\\ \hline
	$1e7$ &$1.915$e-$5$&$1.478$e-$5$&$1.079$e-$4$\\ \hline
	$1e8$ &$1.688$e-$5$&$6.304$e-$6$&$4.135$e-$5$\\ \hline
	\end{tabular}
	\caption{Strongly convex: \\Performance vs Condition number (as $m$ changes)}
	\label{quad_ill}
\end{minipage}
	\begin{minipage}[b]{0.5\linewidth}	
	\centering
	\tiny
	{\begin{tabular}{|c|c|c|c|} \hline 
	&\multicolumn{3}{|c|}{$\mathbb E[f(x_k)-f(x^*)]$}\\ \hline
	$L$&\eqref{rVS-SQN}, $m=1$&\eqref{rVS-SQN}, $m=10$&({\bf VS-APM})\\ \hline
	$1e3$ &$4.978$e-$4$&$1.268$e-$4$&$1.942$e-$4$\\ \hline
	$1e4$ &$3.288$e-$3$&$2.570$e-$4$&$3.612$e-$2$\\ \hline
	$1e5$ &$8.571$e-$2$&$3.075$e-$3$&$2.794$e+$0$\\ \hline
	$1e6$ &$3.367$e-$1$&$3.203$e-$1$&$4.293$e+$0$\\ \hline
	\end{tabular}}
	\caption{Convex: \\Performance vs Condition number (as $m$ changes)}
	\label{convex_ill}
		\end{minipage}
	\end{table}

\noindent {\bf Example 3. Constrained Problems.} We consider the isotonic constrained LASSO problem. 
\begin{align}\label{isotonic}
\min_{x =[x_i]_{i=1}^n\in \mathbb R^n}~ \left\{ \frac{1}{2}\sum_{i=1}^p \|A_ix-b_i\|^2 \mid x_1\leq x_2\leq \hdots\leq x_n \right\},
\end{align}
where $A=[A_i]_{i=1}^p\in\mathbb{R}^{n\times p}$ is a matrix whose elements are chosen randomly from standard Gaussian distribution such that the $A^\top A\succeq 0$ 
and $b=[b_i]_{i=1}^p\in\mathbb{R}^p$ such that $b=A(x_0+ {\sigma})$ where
$x_0\in\mathbb{R}^n$ is chosen such that the first and last $\frac{n}{4}$ of
its elements are chosen from $U([-10,0])$ and $U([0,10])$ in ascending order,
respectively, while the other elements are set to zero. Further,
${\sigma}\in\mathbb{R}^n$ is a random vector whose elements are independent
normally distributed random variables with mean zero and standard deviation
$0.01$.  Let $C\in\mathbb{R}^{n-1\times n}$ be a matrix that captures the
constraint, i.e., $C(i,i)=1$ and $C(i,i+1)=-1$ for $1\leq i\leq n-1$ and its
other components are zero and let $X\triangleq \{x~:~Cx\leq 0\}$. Hence, we
can rewrite the problem
\eqref{isotonic} as $\min_{x \in \mathbb R^n} f(x):=\frac{1}{2}\sum_{i=1}^p
\|A_ix-b_i\|^2+\mathcal{I}_{X}(x)$. We know that the smooth approximation of
the indicator function is $\mathcal I_{X,\eta}={1\over 2\eta} d^2_{X}(x)$.
Therefore, we apply \eqref{rsVS-SQN} on the following problem 
\begin{align}\label{isotonic_smooth}
\min_{x \in \mathbb R^n} f(x) & \triangleq \frac{1}{2}\sum_{i=1}^p \|A_ix-b_i\|^2+{1\over 2\eta} d^2_{X}(x).
\end{align}
{Parameter choices are similar to those in Table \ref{compare_tab} and we note from Fig.~\ref{fig_isotonic} (Left) that empirical behavior appears to be favorable. }
\begin{figure}[htb]
	\centering
	\includegraphics[scale=0.1]{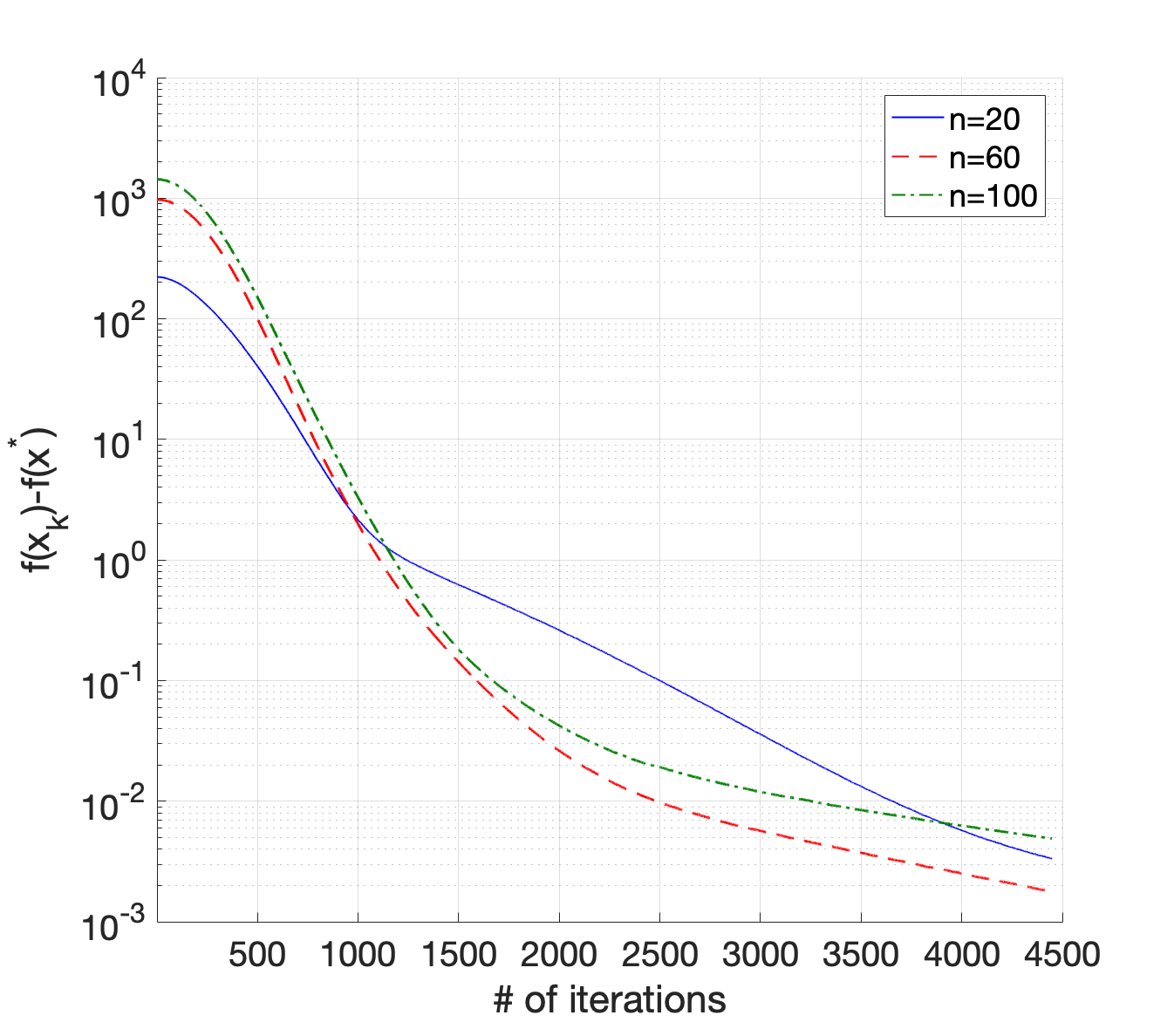}
	\includegraphics[scale=0.1]{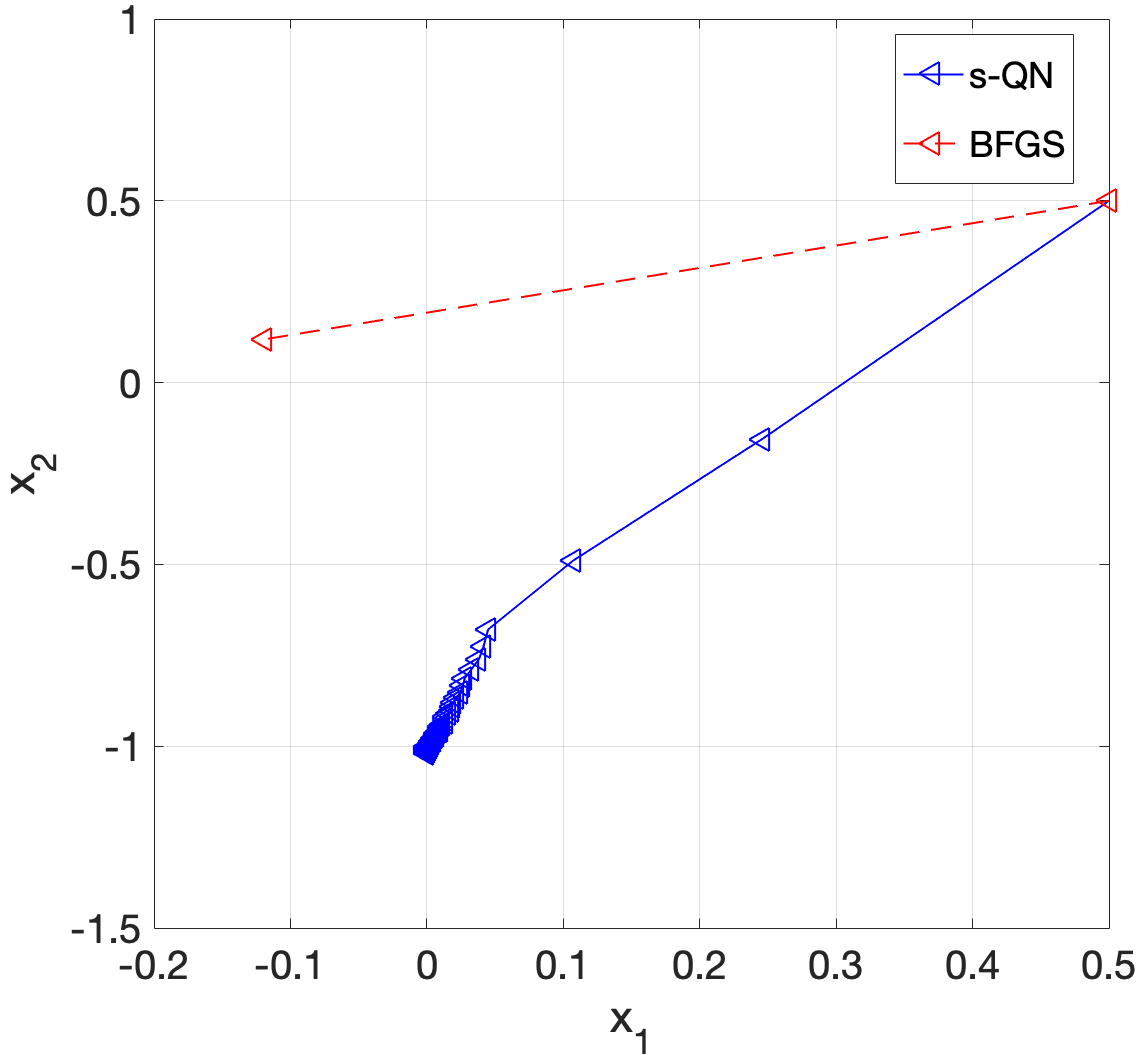}\caption{Left: \eqref{sVS-SQN} Right: \eqref{sVS-SQN}~vs.~BFGS}
		\label{fig_isotonic}
\end{figure}


\noindent {\bf Example 4. Comparison of ({\bf s-QN}) with BFGS}
In~\cite{lewis2008behavior}, the authors show that a nonsmooth BFGS scheme may
take null steps and fails to converge to the optimal solution
(See~Fig.~\ref{fig:nssqn}) and consider the following problem.  \begin{align*}
\min_{x {\in \mathbb R^2}}  \qquad {1\over 2}\|x\|^2+\max\{2|x_1|+x_2,3x_2\}.
\end{align*} In this problem, {BFGS takes a null step after two iterations
(steplength is zero)}; however ({\bf s-QN}) (the deterministic version of \eqref{sVS-SQN}) converges to the optimal solution.
Note that the optimal solution is $(0,-1)$ and ({\bf s-QN}) reaches
$(0,-1.0006)$ in just $0.095$ seconds (see Fig.~\ref{fig_isotonic} (Right)).
\section{Conclusions}
Most SQN schemes can process  smooth and strongly convex stochastic optimization problems
and there appears be a gap in the asymptotics and rate statements in
addressing merely convex and possibly nonsmooth settings. Furthermore, a clear
difference exists between deterministic rates and their stochastic
counterparts, paving the way for developing variance-reduced schemes.  In
addition, much of the available statements rely on a somewhat stronger
assumption of uniform boundedness of the conditional second moment of the
noise, which is often difficult to satisfy in unconstrained regimes.
Accordingly, the present paper makes three sets of contributions. First, a
regularized smoothed L-BFGS update is proposed that combines regularization
and smoothing, providing a foundation for addressing nonsmoothness and a lack
of strong convexity. Second, we develop a variable sample-size SQN scheme
\eqref{VS-SQN} for strongly convex problems and its Moreau smoothed variant
\eqref{sVS-SQN} for nonsmooth (but smoothable) variants, both of which attain a
linear rate of convergence and an optimal oracle complexity. \jal{To contend with the possible unavailability of strong convexity and Lipschitzian parameters, we also derive sublinear rates of convergence for diminishing steplength variants.} Third, in merely convex regimes, we develop a regularized VS-SQN
\eqref{rVS-SQN} and its smoothed variant \eqref{rsVS-SQN} for smooth and
nonsmooth problems respectively. The former achieves a rate of
$\mathcal{O}(1/K^{1-\epsilon})$ while the rate degenerates to
$\mathcal{O}(1/K^{1/3-\epsilon})$ in the case of the latter. Finally, numerics
suggest that the SQN schemes compare well with their variable sample-size
accelerated gradient counterparts and perform particularly well in comparison
when the problem is afflicted by ill-conditioning.       
 


\bibliographystyle{siam} 
\bibliography{demobib,wsc11-v02}
 \section{Appendix}{\bf Proof of Proposition \ref{thm:mean:smooth:strong}.}  \usr{From Assumption~\ref{assum:as_convex_smooth} (a,b), $f$ is $\tau$-strongly convex and $L$-smooth.}  From Lipschitz continuity of $\nabla f(x)$ and update rule \eqref{VS-SQN}, we have the following:
\begin{align*}
f(x_{k+1})&\leq f(x_k)+\nabla f(x_k)^T(x_{k+1}-x_k)+{L\over 2 }\|x_{k+1}-x_k\|^2\\&
=f(x_k)+\nabla f(x_k)^T\left(-\gamma_kH_k(\nabla f(x_k)+\bar w_{k,N_k})\right)+{L\over 2 }\gamma_k^2\left\|H_k(\nabla f(x_k)+\bar w_{k,N_k})\right\|\uvs{^2},
\end{align*}
{where $\bar{w}_{k,N_k} \triangleq \frac{\sum_{j=1}^{N_k} \left({\nabla_{x}}
{F}(x_k,\omega_{j,k})-\nabla f(x_k)\right)}{N_k}$}. By taking expectations \uvs{conditioned on} $\mathcal F_k$, using Lemma {\ref{H_k sc}}, and  Assumption  \ref{state noise} (S-M) and (S-B), we obtain the following.
\begin{align*}
& \quad \mathbb E\left[f(x_{k+1})-f(x_k)\mid \mathcal F_k\right]\leq -\gamma_k \nabla f(x_k)^TH_k\nabla f(x_k)+{L\over 2 }\gamma_k^2\|H_k\nabla f(x_k)\|^2+{\gamma_k^2\olambda^2L\over 2 }\mathbb E[\|\bar w_{k,N_k}\|^2\mid \mathcal F_k]\\&
={\gamma_k}\nabla f(x_k)^TH_k^{1/2}\left(-I+{L\over 2 }\gamma_k\uvs{H_k}\right)H_k^{1/2}\nabla f(x_k)+{\gamma_k^2\olambda^2L(\nu_1^2\|x_k\|^2+\nu_2^2)\over 2 N_k}\\&
\leq -\gamma_k \left(1-{L\over 2 }\gamma_k\olambda\right)\|H_k^{1/2}\nabla f(x_k)\|^2+{\gamma_k^2\olambda^2L(\nu_1^2\|x_k\|^2+\nu_2^2)\over 2 N_k}
= {-\gamma_k\over 2}\|H_k^{1/2}\nabla f(x_k)\|^2+{ \nu_1^2\|x_k\|^2+\nu_2^2\over 2LN_k},
\end{align*} 
where {the last equality follows from} $\gamma_k=  \tfrac{1}{L\olambda}$ for all $k$. Since $f$ is strongly convex with modulus $\tau$, $\|\nabla f(x_k)\|^2\geq 2\tau \left(f(x_k)-f(x^*)\right)$. \uvs{Therefore by subtracting $f(x^*)$} from both sides, we obtain:
\begin{align}\label{strong:smooth}
\mathbb E\left[f(x_{k+1})-f(x^*)\mid \mathcal F_k\right]\nonumber&\leq f(x_{k})-f(x^*)-{\gamma_k\ulambda\over 2}\|\nabla f(x_k)\|^2+{ \nu_1^2\|x_k-x^*+x^*\|^2+\nu_2^2\over 2LN_k}\\&
\leq\left(1-\tau\gamma_k\ulambda+{2\nu_1^2\over L\tau N_k}\right) (f(x_{k})-f(x^*))+{ 2\nu_1^2\|x^*\|^2+\nu_2^2\over 2LN_k},
\end{align}
where the last inequality \uvs{is a consequence of} $f(x_k)\geq f(x^*)+{\tau\over 2}\|x_k-x^*\|^2$. Taking unconditional expectations on both sides of  \eqref{strong:smooth}, choosing  $\gamma_k={ 1\over L\olambda}$ for all $k$ and \uvs{invoking} the assumption that $\{N_k\}$ is an increasing sequence,  \uvs{we obtain the following.}
\begin{align*}
\mathbb E\left[f(x_{k+1})-f(x^*)\right]&
\leq\left(1-{\tau  \ulambda\over L\olambda}+{2\nu_1^2\over L\tau N_0}\right)\mathbb E\left[f(x_{k})-f(x^*)\right]+{  2\nu_1^2\|x^*\|^2+\nu_2^2\over 2LN_k}.
\end{align*}

{\bf Proof of Theorem \ref{th1}.} \usr{From Assumption~\ref{assum:as_convex_smooth} (a,b), $f$ is $\tau$-strongly convex and $L$-smooth.}
{\bf (i)} Let $a  \triangleq   \left(1-{\tau  \ulambda\over L\olambda}+{2\nu_1^2\over L\tau N_0}\right)$, $b_k \triangleq { 2 \nu_1^2\|x^*\|^2+\nu_2^2\over 2LN_k}$, and $N_k \triangleq \lceil N_0\rho^{-k}\rceil\geq N_0\rho^{-k}$. Note that, choosing $N_0\geq {2\nu_1^2\olambda\over \tau^2\ulambda}$ leads to $a<1$. Consider $C \triangleq \uvs{ \mathbb{E}[f(x_0)-f(x^*)]}+\left({2 \nu_1^2\|x^*\|^2+\nu_2^2\over 2N_0L}\right){1\over 1-( \min\{a,\rho\}/\max\{a,\rho\})}$. {Then}  {by Prop.~\ref{thm:mean:smooth:strong},} we obtain {the following for every $k \geq 1$.} 
\begin{align*}
\mathbb E&\left[f(x_{K+1})-f(x^*)\right]
\leq a^{K+1}{\mathbb{E}\left[f(x_0)-f(x^*)\right]}+\sum_{i=0}^{K}a^{K-i}b_{i}\\
&\leq a^{K+1}{\mathbb{E}\left[f(x_0)-f(x^*)\right]}+{(\max\{a,\rho\})^K(2 \nu_1^2\|x^*\|^2+\nu_2^2)\over 2N_0L}\sum_{i=0}^{K}\left({\min\{a,\rho\}\over \max\{a,\rho\}}\right)^{K-i}\\
&\leq a^{K+1}{\mathbb{E}\left[f(x_0)-f(x^*)\right]}+\left({(2 \nu_1^2\|x^*\|^2+\nu_2^2)\over 2N_0L}\right){{(\max\{a,\rho\})^K}\over 1-( \min\{a,\rho\}/\max\{a,\rho\})}\leq C(\max\{a,\rho\})^{K}.
\end{align*}
Furthermore, {we may derive the number of steps $K$ to obtain an $\epsilon$-solution. Without loss of generality, suppose $\max\{a,\rho\}=a$. Choose $N_0 { \ \geq \ } {4\nu_1^2{\ulambda}\over \tau^2{\olambda}}$, then $a=\left(1-\left({\tau  \ulambda\over 2L\olambda}\right)\right)=1-{1\over \alpha\kappa}$, where $\alpha={2\olambda\over \ulambda}$}. Therefore, since $\frac{1}{a}  = \frac{1}{(1-\frac{1}{\alpha {\kappa}})}$, by using the definition of $\ulambda$ and $\olambda$ in Lemma \ref{H_k sc} to get $\alpha= {2\olambda\over \ulambda}=\mathcal O({\kappa^{m+n}})$, we obtain that 
\begin{align}
 \left(\frac{ \ln(C) - \ln(\epsilon)} {\ln(1/a)}\right) =
\left(\frac{\ln (C/\epsilon)}{\ln(1/(1-{1\over \alpha \kappa}))}\right)
= \left(\frac{\ln (C/\epsilon)}{-\ln((1-{1\over \alpha \kappa}))}\right) \leq \left(\frac{\ln (C/\epsilon)}{{1\over \alpha \kappa}}\right)   
\notag   = \mathcal{O} {({\kappa^{m+n+1}} \ln(\tilde {C}/\epsilon))},
\end{align}
{where the {bound} holds when $\alpha \kappa > 1$. It follows that the iteration complexity of computing an $\epsilon$-solution is $\mathcal{O}(\kappa^{m+1} \ln(\tfrac{C}{\epsilon}))$.}
{\bf(ii)} To compute a vector $x_{K+1}$ satisfying $\mathbb{E}[f(x_{{K+1}})-f^*]\leq \epsilon$, we {consider the case where $a > \rho$ while the other case follows similarly.} Then we have that $C{a}^{K}\leq \epsilon$, implying that  
$K = \lceil \ln_{(1/  {a})}(C/\epsilon)\rceil.$
To obtain the optimal oracle complexity, we require $\sum_{k=1}^{K} N_k$ gradients. If $N_k=\lceil N_0a^{-k}\rceil\leq 2N_0a^{-k}$, we obtain the following since $(1-a) = 1 \slash (\alpha {\kappa})$.
\begin{align*}
& \quad \sum_{k=1}^{\ln_{(1/{a})}\left(C/\epsilon\right)+1} 2N_0a^{-k} 
 \leq \frac{2N_0}{\left(\frac{1}{{a}} -1\right)}\left({1\over a}\right)^{3+\ln_{(1/  {a})}\left(C/\epsilon\right)}  \leq \left( C \over \epsilon\right)\frac{2N_0}{a^2(1-{a})} 
 = \frac{ 2N_0\alpha {\kappa} C}{a^2\epsilon}.
\end{align*}
Note that $a=1-{1\over \alpha\kappa}$ {and $\alpha=\mathcal O{(\kappa^{m+n})}$}, implying that
\begin{align*}
  a^2 & = 1-2/(\alpha \kappa)+1/(\alpha^2\kappa^2)\geq {\alpha^2\kappa^2-2\alpha\kappa^2+1\over \alpha^2\kappa^2}\geq{ \alpha^2\kappa^2-2\alpha\kappa^2\over \alpha^2\kappa^2}={(\alpha^2-2\alpha)\over \alpha^2}\\
  \implies & {\kappa\over a^2}\leq {\alpha^2  \kappa\over (\alpha^2-2\alpha)}=\left(\alpha\over \alpha-2\right)\kappa
 \implies   \sum_{k=1}^{\ln_{(1/{a})}\left(C/\epsilon\right)+1} a^{-k} \leq {2N_0\alpha^2\kappa C\over (\alpha-2)\epsilon}=\mathcal O\left({{ \kappa^\jal{m+n+1}}\over \epsilon}\right).
\end{align*}

\noindent (iii) Similar to ~\eqref{strong:smooth}, using strong convexity of $f$ and choosing $N_k = \lceil{k^{p-s}}\rceil$ and ${\gamma_k=k^{-s}}$, we can obtain the following.
{\begin{align*}
&\mathbb E\left[f(x_{k+1})-f(x_k)\mid \mathcal F_k\right]  \leq -\gamma_k \left(1-{L\over 2 }\gamma_k\olambda\right)\|H_k^{1/2}\nabla f(x_k)\|^2+{\gamma_k^2\olambda^2L(\nu_1^2\|x_k\|^2+\nu_2^2)\over 2 N_k}
\\
		&\implies \mathbb E\left[f(x_{k+1})-f(x^*)\mid \mathcal F_k\right] \leq \left(1-\gamma_k(1-\tfrac{L\olambda\gamma_k}{2})\ulambda\tau+\tfrac{\gamma_k^2\olambda^2L\nu_1^2}{\tau N_k}\right)(f(x_k)-f(x^*))+\tfrac{\gamma_k^2\olambda^2L(2\nu_1^2\|x^*\|^2+\nu_2^2)}{2N_k}. 
\end{align*}}
\jal{Since $\gamma_k$ is a diminishing sequence and $N_k$ is an increasing sequence, for {sufficiently large} $K$ we have that $\tfrac{L\olambda\gamma_k^2\ulambda\tau}{2}+\tfrac{\gamma_k^2\olambda^2L\nu_1^2}{\tau N_k}\leq {1\over 2}\gamma_k\ulambda\tau$. Therefore, we obtain:
 \begin{align*}
\mathbb E\left[f(x_{k+1})-f(x^*)\mid \mathcal F_k\right] &  \leq \left(1-\tfrac{\gamma_k\ulambda\tau}{2} \right) (f(x_{k})-f(x^*))+ \tfrac{\gamma_k^2\olambda^2L(2\nu_1^2\|x^*\|^2+\nu_2^2)}{2N_k}\\
	& =  \left(1-\tfrac{c}{k^s}  \right) (f(x_{k})-f(x^*))+ \tfrac{d}{k^{p+s}}, \mbox{ for } k \geq K,
\end{align*}
where $c \triangleq \tfrac{\ulambda\tau}{2}$ and $d \triangleq { \olambda^2L(2\nu_1^2\|x^*\|^2+\nu_2^2)\over 2}$.} Then by taking unconditional expectations and recalling that $0 < s < 1$, $s < p$, we may invoke Lemma~\ref{poly-rate} to claim that there exists $K$ such that 
{\begin{align*}
\mathbb E\left[f(x_{k+1})-f(x^*)\right] &  \leq \left(\frac{d}{ck^{p}}\right)  + o\left(\frac{1}{k^{p}}\right), \quad k \geq K.
\end{align*}} 
 
 {\bf Proof of Lemma \ref{H_k sc} and Lemma \ref{rsLBFGS-matrix} :}
\afr{First we prove Lemma \ref{rsLBFGS-matrix} and then we show that how the result in Lemma \ref{H_k sc} can be proved similarly}. Recall that $\ulambda_k$ and $\olambda_k$  denote the minimum and maximum eigenvalues of $H_k$, respectively. Also, we denote the inverse of matrix $H_k$ by $B_k$. 
\begin{lemma}\label{lemma appen}\cite{yousefian2017stochastic}
Let $0 < a_1 \leq a_2 \leq \hdots \leq a_n$, $P$ and $S$ be positive scalars such that $\sum_{i=1}^n a_i \leq S$ and $\Pi_{i=1}^n a_i \geq P$ . Then, we have
$a_1 \geq \frac{(n-1)!P}{S^{n-1}}.$
\end{lemma}
{\bf Proof of Lemma \ref{rsLBFGS-matrix}:} It can be seen, by induction on $k$, that $H_k$ is symmetric and $\mathcal{F}_k$ measurable, assuming that all matrices are well-defined. We use induction on odd values of $k>2m$ to show that the statements of part (a), (b) and (c) hold and that the matrices are well-defined. Suppose $k>2m$ is odd and for any odd value of $t<k$, we have $s_t^T{y_t} >0$, $H_{t}{y}_t=s_t$, and part (c) holds for $t$. We show that these statements also hold for $k$. First, we prove that the secant condition holds. 
\aj{\begin{align*}
& s_k^T{y_k}=(x_{k}-x_{k-1})^T\left(\tfrac{\sum_{j=1}^{N_{k-1}}\left(\nabla F_{\eta_{k}^\delta}(x_k,\omega_{j,k-1})- \nabla F_{\eta_{k}^\delta}(x_{k-1},\omega_{j,k-1})\right)}{N_{k-1}}+\mu_k^\delta(x_k-x_{k-1})\right)\\
&=\tfrac{\sum_{j=1}^{N_{k-1}}\left[(x_{k}-x_{k-1})^T(\nabla F_{\eta_{k}^\delta}(x_k,\omega_{j,k-1})- \nabla F_{\eta_{k}^\delta}(x_{k-1},\omega_{j,k-1}))\right]}{N_{k-1}}+\mu_k^\delta\|x_k-x_{k-1}\|^2
\geq \mu_k^\delta\|x_k-x_{k-1}\|^2,
\end{align*}}
where the inequality follows from the monotonicity of the gradient map $\nabla F(\cdot,\omega)$. {From the induction hypothesis, $H_{k-2}$ is positive definite, since $k-2$ is odd.
Furthermore, since $k-2$ is odd, we have $H_{k-1}=H_{k-2}$ by the update rule \eqref{eqn:H-k}.  
Therefore, $H_{k-1}$ is positive definite.
Note that since $k-2$ is odd, the choice of $\mu_{k-1}$ is such that ${1\over N_{k-1}}\sum_{j=1}^{N_{k-1}}\nabla F_{\eta_{k}^\delta}(x_{k-1},\omega_{j,k-1})+\mu_{k-1}x_{k-1}\neq 0$ (see the discussion
following~\eqref{eqn:mu-k}).
Since $H_{k-1}$ is positive definite, 
we have $H_{k-1}\left({1\over N_{k-1}}\sum_{j=1}^{N_{k-1}}\nabla F_{\eta_{k}^\delta}(x_{k-1},\omega_{j,k-1})+\mu_{k-1}x_{k-1}\right) \neq 0$, implying that 
$x_{k} \neq x_{k-1}$. Hence
$s_k^T{y_k} \geq  \mu_k^\delta\|x_k-x_{k-1}\|^2 >0,$
where \us{the second inequality is a consequence of} $\mu_k>0$.
Thus, the secant condition holds.}
Next, we show that part (c) holds for $k$. Let $\ulambda_k$ and $\olambda_k$ denote the minimum and maximum eigenvalues of $H_k$, respectively. Denote the inverse of matrix $H_k$ in \eqref{eqn:H-k-m} by $B_k$. It is well-known that using the Sherman-
Morrison-Woodbury formula, $B_k$ is equal to $B_{k,m}$ given by 
\begin{align}\label{equ:B_kLimited}
B_{k,j}=B_{k,j-1}-\frac{B_{k,j-1}s_is_i^TB_{k,j-1}}{s_i^TB_{k,j-1}s_i}+\frac{y_iy_i^T}{y_i^Ts_i}, \quad i:=k-2(m-j) \quad 1 \leq j \leq m,
\end{align}
where $s_i$ and $y_i$ are defined by \eqref{equ:siyi-LBFGS} and $B_{k,0}=\frac{y_k^Ty_k}{s_k^Ty_k}\mathbf{I}$. First, we show that for any $i$, \begin{align}\label{equ:boundsForB0}
\mu_k^\delta \leq \frac{\|y_i\|^2}{y_i^Ts_i} \leq 1/\eta_k^{\delta}+\mu_k^\delta,
\end{align}
 Let us consider the function $h(x):={1\over N_{i-1}}\sum_{j=1}^{N_{i-1}}F_{\eta_{k}^\delta}(x,\omega_{j,i-1})+\frac{\mu_k^\delta}{2}\|x\|^2$ for fixed $i$ and $k$. Note that this function is strongly convex and has a gradient mapping of the form ${1\over N_{i-1}}\sum_{j=1}^{N_{i-1}}\nabla F_{\eta_{k}^\delta}(x_{i-1},\omega_{j,i-1})+\mu_k^\delta\mathbf{I}$ that  is Lipschitz with parameter ${1\over \eta_k^\delta}+\mu_k^\delta$. For a convex function $h$ with Lipschitz gradient with parameter $1/\eta_k^\delta+\mu_k^\delta$, the following inequality, referred to as co-coercivity  property, holds for any $x_1,x_2 \in \mathbb{R}^n$(see \cite{polyak1987introduction}, Lemma 2):
 $\|\nabla h(x_2)-\nabla h(x_1)\|^2 \leq (1/\eta_k^\delta+\mu_k^\delta)(x_2-x_1)^T(\nabla h(x_2)-\nabla h(x_1)).$
Substituting $x_2$ by $x_i$, $x_1$ by $x_{i-1}$, and recalling \eqref{equ:siyi-LBFGS}, the preceding inequality yields
\begin{align}\label{ineq:boundsForB0-1}\|y_i\|^2  \leq (1/\eta_k^\delta+\mu_k^\delta)s_i^Ty_i.\end{align}
Note that function $h$ is strongly convex with parameter $\mu_k^\delta$. Applying the Cauchy-Schwarz inequality, we can write
\[\frac{\|y_i\|^2}{s_i^Ty_i} \geq \frac{\|y_i\|^2}{\|s_i\|\|y_i\|} =\frac{\|y_i\|}{\|s_i\|}\geq  \frac{\|y_i\|\|s_i\|}{\|s_i\|^2} \geq \frac{y_i^Ts_i}{\|s_i\|^2}\geq \mu_k^\delta.\]
Combining this relation with \eqref{ineq:boundsForB0-1}, we obtain \eqref{equ:boundsForB0}. Next, we show that the maximum eigenvalue of $B_k$ is bounded. Let $Trace(\cdot)$ denote the trace of a matrix. Taking trace from both sides of \eqref{equ:B_kLimited} and summing up over index $j$, we obtain
\begin{align}\label{ineq:trace}
& \quad Trace(B_{k,m})=Trace(B_{k,0})-\sum_{j=1}^m Trace\left(\frac{B_{k,j-1}s_is_i^TB_{k,j-1}}{s_i^TB_{k,j-1}s_i}\right)+\sum_{j=1}^m Trace\left(\frac{y_iy_i^T}{y_i^Ts_i}\right)\\ \nonumber
& =Trace\left(\frac{\|y_i\|^2}{y_i^Ts_i}\mathbf{I}\right) - \sum_{j=1}^m \frac{\|B_{k,j-1}s_i\|^2}{s_i^TB_{k,j-1}s_i} + \sum_{j=1}^m \frac{\|y_i\|^2}{y_i^Ts_i}\leq n \frac{\|y_i\|^2}{y_i^Ts_i} +\sum_{j=1}^m (1/\eta_k^\delta+\mu_k^\delta) = (m+n)(1/\eta_k^\delta+\mu_k^\delta),
\end{align}
where the third relation is obtained by positive-definiteness of $B_k$ (this can be seen by induction on $k$, and using \eqref{equ:B_kLimited} and $B_{k,0}\succ 0$). Since $B_k=B_{k,m}$, the maximum eigenvalue of the matrix $B_k$ is bounded. As a result, 
\begin{align}\label{proof:lowerbound}
\ulambda_k\geq \frac{1}{(m+n)(1/\eta_k^\delta+\mu_k^\delta)}.\end{align}
 In the next part of the proof, we establish the bound for $\olambda_k$. From Lemma 3 in \cite{mokhtari2015global}, we have $
det(B_{k,m})=det(B_{k,0})\prod_{j=1}^m\frac{s_i^Ty_i}{s_i^TB_{k,j-1}s_i}.$
Multiplying and dividing by $s_i^Ts_i$, using the strong convexity of the function $h$, and invoking \eqref{equ:boundsForB0} and the result of \eqref{ineq:trace}, we obtain
\begin{align}\label{ineq:detBk}
det(B_{k})&=det\left(\frac{y_k^Ty_k}{s_k^Ty_k}\mathbf{I}\right)\prod_{j=1}^m\left(\frac{s_i^Ty_i}{s_i^Ts_i}\right)\left(\frac{s_i^Ts_i}{s_i^TB_{k,j-1}s_i}\right)  
 \geq\left(\frac{y_k^Ty_k}{s_k^Ty_k}\right)^n\prod_{j=1}^m\mu_k^\delta\left(\frac{s_i^Ts_i}{s_i^TB_{k,j-1}s_i}\right)\cr
& \geq  (\mu_k)^{(n+m)\delta} \prod_{j=1}^m \frac{1}{(m+n)(1/\eta_k^\delta+\mu_k^\delta)} = \frac{\mu_k^{(n+m)\delta}}{(m+n)^{m}(1/\eta_k^\delta+\mu_k^\delta)^m}.
\end{align}
Let $\alpha_{k,1}\leq \alpha_{k,2}\leq\ldots\leq\alpha_{k,n}$ be the eigenvalues of $B_k$ sorted non-decreasingly. Note that since $B_k\succ0$, all the eigenvalues are positive. Also, from \eqref{ineq:trace}, we know that $\alpha_{k,\ell}\leq (m+n)(L+\mu_0^\delta)$. Taking \eqref{ineq:boundsForB0-1} and \eqref{ineq:detBk} into account, and employing Lemma \ref{lemma appen}, we obtain \[\alpha_{k,1}\geq \frac{(n-1)!\mu_k^{(n+m)\delta}}{(m+n)^{n+m-1}(1/\eta_k^\delta+\mu_k^\delta)^{n+m-1}}.\]
This relation and that $\alpha_{k,1}=\olambda_k^{-1}$ imply that 
\begin{align}\label{proof:upperbound}
\olambda_k\leq \frac{(m+n)^{n+m-1}(1/\eta_k^\delta+\mu_k^\delta)^{n+m-1}}{(n-1)!\mu_k^{(n+m)\delta}}.
\end{align}
Therefore, from \eqref{proof:lowerbound} and \eqref{proof:upperbound} and that $\mu_k$ is non-increasing, we conclude that part (c) holds for $k$ as well. Next, we show that $H_ky_k=s_k$. From \eqref{equ:B_kLimited}, for $j=m$ we obtain
\[B_{k,m}=B_{k,m-1}-\frac{B_{k,m-1}s_ks_k^TB_{k,m-1}}{s_k^TB_{k,m-1}s_k}+\frac{y_ky_k^T}{y_k^Ts_k},\]
where we used $i=k-2(m-m)=k$. Multiplying both sides of the preceding equation by $s_k$, and using $B_k=B_{k,m}$, we have
$B_{k}s_k=B_{k,m-1}s_k-B_{k,m-1}s_k+y_k=y_k.$ 
Multiplying both sides of the preceding relation by $H_k$ and invoking $H_k=B_k^{-1}$, we conclude that $H_ky_k=s_k$. Therefore, we showed that the statements of (a), (b), and (c) hold for $k$, assuming that they hold for any odd $2m<t<k$. In a similar fashion to this analysis, it can be seen that the statements hold for $t=2m+1$. Thus, by induction, we conclude that the statements hold for any odd $k>2m$. To complete the proof, it is enough to show that part (c)  holds for any even value of $k>2m$. Let $t=k-1$. Since $t>2m$ is odd, relation part (c) holds. Writing  it for $k-1$, and taking into account that $H_k=H_{k-1}$, and $\mu_k<\mu_{k-1}$, we can conclude that part (c) holds for any even value of $k>2m$ and this completes the proof.

\jal{{\bf Proof of Lemma \ref{H_k sc}:} {We observe} that Lemma \ref{H_k sc} is the special case of Lemma \ref{rsLBFGS-matrix}, where {the objective $f$ function is $L$-smooth and $\tau$-strongly convex}. 
\jal{Similar to} \eqref{proof:lowerbound}, \jal{by noting that $\mu_k = 0$ and $\tfrac{1}{\eta_k} = L$ for all $k$,  we may show that} $\ulambda_k = \ulambda=\frac{1}{L(m+n)}$ for all $k$. {Furthermore from the $\tau$-strongly convex nature of $f$, akin} to \eqref{proof:upperbound}, we obtain {that} $\olambda=\frac{\left((m+n)L\right)^{n+m-1}}{(n-1)!\tau^{n+m}}$. }






\end{document}